\documentclass[reqno,a4paper]{amsart}

\usepackage{array}
\usepackage[T1]{fontenc}
\usepackage{enumerate, amsmath, amsfonts, amssymb, amsthm, mathrsfs}
\usepackage{bbm}

\usepackage[lmargin=3cm,rmargin=3cm,top=2.9cm,bottom=2.9cm]{geometry}

\usepackage[all]{xy}
\usepackage{tikz}
\usetikzlibrary{arrows, decorations, shapes}
\usepackage{tikz-cd}
\usepackage{colonequals}
\usetikzlibrary{arrows}
\usetikzlibrary{decorations.pathmorphing}
\usepackage[colorlinks, citecolor=blue]{hyperref}
\usepackage[noabbrev,capitalise]{cleveref}
\usepackage{thmtools}
\usepackage{multirow}
\usepackage[normalem]{ulem}
\usepackage{enumitem}
\setlist[enumerate]{format=\normalfont}
\usepackage{caption}
\usepackage{cancel}
\usepackage{setspace}
\newcommand{\marginparstretch}{0.6}
\let\oldmarginpar\marginpar
\renewcommand\marginpar[1]{\-\oldmarginpar[\framebox{\setstretch{\marginparstretch}\begin{minipage}{\marginparwidth}{\raggedleft\tiny #1}\end{minipage}}]{\framebox{\setstretch{\marginparstretch}\begin{minipage}{\marginparwidth}{\raggedright\tiny #1}\end{minipage}}}}

\DeclareRobustCommand{\SkipTocEntry}[5]{}

 \newcommand{\arrow}[2][20]
 {
  \hspace{-5pt}
  \begin{tikzpicture}
   \node (A) at (0,0) {};
   \node (B) at (#1pt,0) {};
   \draw [#2] (A) -- (B);
  \end{tikzpicture}
  \hspace{-5pt}
 }
\renewcommand{\leadsto}[1][20]{\arrow[#1]{->, decorate, decoration={snake, segment length=4pt, amplitude=1pt}}}

\newcommand{\theoremlistfix}{\phantomsection\leavevmode}

\theoremstyle{plain} % italic
\declaretheorem[
  numberwithin=section,
  name=Theorem,
  refname={Theorem,Theorems},
  Refname={Theorem,Theorems}
]{theorem}

\declaretheorem[
  sibling=theorem,
  name=Corollary,
  refname={Corollary,Corollaries},
  Refname={Corollary,Corollaries}
]{corollary}

\declaretheorem[
  sibling=theorem,
  name=Proposition,
  refname={Proposition,Propositions},
  Refname={Proposition,Propositions}
]{proposition}

\declaretheorem[
  sibling=theorem,
  name=Lemma,
  refname={Lemma,Lemmas},
  Refname={Lemma,Lemmas}
]{lemma}

\theoremstyle{definition} % upright
\declaretheorem[
  sibling=theorem,
  name=Definition,
  refname={Definition,Definitions},
  Refname={Definition,Definitions}
]{definition}

\declaretheorem[
  sibling=theorem,
  name=Example,
  refname={Example,Examples},
  Refname={Example,Examples}
]{example}

\declaretheorem[
  sibling=theorem,
  name=Remark,
  refname={Remark,Remarks},
  Refname={Remark,Remarks}
]{remark}

\declaretheorem[
  sibling=theorem,
  name=Setup,
  refname={Setup,Setups},
  Refname={Setup,Setups}
]{setup}

\newcommand{\tc}{\mathcal{T}}

\newcommand{\ac}{\mathcal{A}}

\newcommand{\mc}{\mathcal{M}}
\newcommand{\uc}{\mathcal{U}}

\newcommand{\ec}{\mathcal{E}}
\newcommand{\fc}{\mathcal{F}}
\newcommand{\pc}{\mathcal{P}}
\newcommand{\xc}{\mathcal{X}}
\newcommand{\vc}{\mathcal{V}}

\newcommand{\infl}{\rightarrowtail}
\newcommand{\defl}{\twoheadrightarrow}

\newcommand{\modA}{\operatorname{mod}A}
\renewcommand{\mod}{\operatorname{mod}}
\newcommand{\modB}{\operatorname{mod}B}
\newcommand{\fac}{\operatorname{Fac}}

\newcommand{\Ext}{\operatorname{Ext}}
\newcommand{\End}{\operatorname{End}}

\newcommand{\Ker}{\operatorname{Ker}}
\newcommand{\Coker}{\operatorname{Coker}}
\newcommand{\Hom}{\operatorname{Hom}}
\newcommand{\Rad}{\operatorname{Rad}}
\newcommand{\Soc}{\operatorname{Soc}}
\newcommand{\add}{\operatorname{add}}
\newcommand{\Ann}{\operatorname{Ann}}
\newcommand{\proj}{\operatorname{proj}}

\newcommand{\Thick}{\operatorname{Thick}}
\newcommand{\kk}{\mathbbm{k}}
\newcommand{\os}{\operatorname{os}}

\newcommand\msmall[1]{\mbox{\small\ensuremath{#1}}}%make a special "math mode" version of small to avoid error messages
\newcommand{\Rep}[1]{%
  {%Make sure some parameters are set to the default
    \renewcommand{\arraystretch}{1}
    \setlength{\extrarowheight}{0pt}
    \msmall{%
    \begin{matrix}%
      #1%
    \end{matrix}%
  }}%
}
\newcommand\mtiny[1]{\mbox{\tiny\ensuremath{#1}}}%make a special "math mode" version of tiny to avoid error messages
\newcommand{\rep}[1]{%
  {%Make sure some parameters are set to the default
    \renewcommand{\arraystretch}{1}
    \setlength{\extrarowheight}{0pt}
    \mtiny{%
    \begin{matrix}%
      #1%
    \end{matrix}}%
  }%
}

\title{Higher torsion classes, $\tau_d$-tilting theory and silting complexes}

\author[August]{Jenny August}
\address{University of Glasgow\\ University Place\\ Glasgow \\ G12 8QQ \\ UK}
        \email{jenny.august@glasgow.ac.uk}
\author[Haugland]{Johanne Haugland}
\address{Department of Mathematical Sciences\\ 
        NTNU\\ 
        NO-7491 Trondheim\\ 
        Norway}
\email{johanne.haugland@ntnu.no}
\author[Jacobsen]{Karin M.\ Jacobsen}
\address{Department of Mathematics\\ Aarhus Universitet\\ Ny Munkegade 118\\ DK-8000 Aarhus C\\ Denmark}
\email{karin.m.jacobsen@gmail.com}
\author[Kvamme]{Sondre Kvamme}
\address{Department of Mathematical Sciences\\ 
        NTNU\\ 
        NO-7491 Trondheim\\ 
        Norway}
\email{sondre.kvamme@ntnu.no}
\author[Palu]{Yann Palu}
\address
{LAMFA, Universit\'e de Picardie Jules Verne, 33, rue Saint-Leu 80039 Amiens, France}
\email{yann.palu@u-picardie.fr}

\author[Treffinger]{Hipolito Treffinger}
\address{Universidad de Buenos Aires. Facultad de Ciencias Exactas y Naturales. Departamento de Matemática. Buenos Aires, Argentina.
CONICET - Universidad de Buenos Aires. Instituto de Investigaciones Matemáticas “Luis A. Santaló”  (IMAS). Buenos Aires, Argentina.
}
\email{htreffinger@dm.uba.ar}

\usepackage{url}
\begin{document}

\begin{abstract}
Initiated in work by Adachi, Iyama and Reiten, the area known as $\tau$-tilting theory plays a fundamental role in contemporary representation theory. In this paper we explore a higher-dimensional analogue of this theory, formulated with respect to the higher Auslander--Reiten translation $\tau_d$. In particular, we associate to any functorially finite $d$-torsion class a maximal $\tau_d$-rigid pair and a $(d+1)$-term silting complex. In the case $d=1$, the notions of maximal $\tau_d$-rigid and support $\tau$-tilting pairs coincide, and our theory recovers the classical bijections. However, the proof strategies for $d>1$ differ significantly. As an intermediate step, we prove that a \mbox{$d$-cluster} tilting subcategory of a module category induces a \mbox{$d$-cluster} tilting subcategory of the category of $(d+1)$-term complexes, producing novel examples of $d$-exact categories. We introduce the notion of a $d$-torsion class in the exact setup, and use this to obtain the aforementioned $(d+1)$-term silting complex. We moreover apply our theory to study $d$-APR tilting modules and slices. To illustrate our results, we provide explicit combinatorial descriptions of maximal $\tau_d$-rigid pairs and $(d+1)$-term silting complexes for higher Auslander and higher Nakayama algebras.
\end{abstract}

\dedicatory{We dedicate this paper to the memory of Idun Reiten, \\ whose mathematical and non-mathematical contributions continue to inspire us.}


\maketitle
\setcounter{tocdepth}{1}
\tableofcontents

\section{Introduction}\label{sec:intro}
Cluster algebras were introduced by Fomin and Zelevinsky in \cite{FominZelevinskyI} as a combinatorial tool to study total positivity and Lusztig's dual canonical bases. Soon after, a deep connection between cluster algebras and the  representation theory of finite-dimensional algebras was established via categorification \cite{Amiot,BuanMarshReinekeReitenTodorov,CalderoChapotonSchiffler,GeissLeclercSchroer4,KellerYang,Plamondon1}. Ever since then, the two fields have achieved great success through mutual interaction. On the one hand, representation theoretic tools led to proofs of several conjectures in cluster theory; see e.g.\ \cite{CalderoReineke,Keller,Plamondon2,Qin}. On the other hand, cluster algebras, and the key concept of mutation in particular, motivated the development of two areas that are central in contemporary representation theory: higher Auslander--Reiten theory and $\tau$-tilting theory. This paper lies at the intersection of these two topics. 

\addtocontents{toc}{\SkipTocEntry}
\subsection*{Motivation and background}
Higher Auslander--Reiten theory \cite{Iyama2007a, Iyama2007b,Iyama2011} studies so-called \mbox{$d$-cluster} tilting subcategories, which arise naturally when classifying algebras with certain homological properties. These $d$-cluster tilting subcategories are often subcategories of module categories of finite-dimensional algebras, but instead of being abelian like the ambient category, they turn out to be $d$-abelian \cite{Jasso}. This means that their homological structure is controlled by exact sequences with $d$ middle terms for some positive integer $d$. When $d=1$, one recovers abelian categories and short exact sequences, while the case $d>1$ opens up a whole new avenue of research studying higher analogues of important objects in representation theory \cite{GrantIyama,HerschendIyamaMinamotoOppermann,HerschendIyamaOppermann,HerschendJorgensenVaso,HauglandSandoy,IyamaOppermann,Higher Nakayama}. 
This perspective, often called higher homological algebra, has further been expanded to higher analogues \cite{GeissKellerOppermann,HerschendLiuNakaoka} of triangulated and extriangulated categories \cite{NakaokaPalu}.

On the other hand, $\tau$-tilting theory \cite{AdachiIyamaReiten} (cf.\ \cite{DerksenFei}) was developed to mirror the properties of mutation seen in cluster algebras and give a generalisation of classical tilting modules. Remarkably, broadening the definition of tilting modules to allow for mutation also led to bijections with other important objects in classical representation theory, such as functorially finite torsion classes and $2$-term silting complexes. 

The importance of higher Auslander--Reiten theory and $\tau$-tilting theory is not limited to representation theory of finite-dimensional algebras. For example, $\tau$-tilting theory allowed for a better understanding of stability conditions \cite{Asai, AugustWemyss, BrustleSmithTreffinger}, while higher Auslander--Reiten theory has connections to lattices of order ideals \cite{Gottesman}, Fukaya categories \cite{DiDedda1, DiDedda2, DyckerhoffJassoLekili}, and provided a key ingredient in the first proof of the Donovan--Wemyss conjecture in algebraic geometry \cite{JassoKellerMuro}.

Given the significance of these two theories, it is natural to investigate the higher homological analogue of $\tau$-tilting theory. We refer to this area as $\tau_d$-tilting theory. It was initiated in \cite{JJ}, where the notion of a (maximal) $\tau_d$-rigid pair is defined, and further studied in \cite{ZZ,AHSV}. However, as the significance of classical $\tau$-tilting theory stems from how it serves to provide a link between otherwise unrelated concepts, exploring how $\tau_d$-rigid pairs relate to other notions is of crucial importance. Although this has been investigated in the special case of linear Nakayama algebras \cite{RundsveenVaso}, little is known in general. We aim at filling this gap by establishing fundamental connections between $\tau_d$-tilting theory and other notions in higher homological algebra.

One such higher notion is that of $d$-torsion classes, as introduced in \cite{Jorgensen} and further studied in \cite{AJST, AHJKPT1}. It is well known that torsion classes in a finite-length abelian category are precisely those subcategories that are closed under extensions and quotients \cite{Dickson}. Inspired by this, it is shown in \cite{AHJKPT1} that $d$-torsion classes in $d$-cluster tilting subcategories are characterised by closure under $d$-extensions and $d$-quotients (see \cref{sec:higher world} for definitions). This characterisation further opened up the study of the lattice of $d$-torsion classes; see \cite[\mbox{Section 4}]{AHJKPT1} and \cite{Segovia}.

The final higher notion we consider is that of $(d+1)$-term silting complexes; see \cite{ALLT,Gupta,MM,Zhou}. Classical silting theory plays a major role in contemporary homological algebra, as it simultaneously provides a connection with t-structures \cite{KellerVos, KoenigYang} and resolves the deficiencies in the mutation of tilting complexes \cite{AiharaIyama}. The $2$-term silting complexes are closely linked with \mbox{$\tau$-tilting} theory, as well as $g$-vector fans, stability conditions and braid group actions \cite{AiharaMizuno, Asai, BrustleSmithTreffinger, JY}.  The higher version of these objects, namely $(d+1)$-term silting complexes, play a crucial role in this paper. 
\addtocontents{toc}{\SkipTocEntry}
\subsection*{Main results}

Let $A$ be a finite-dimensional algebra and consider a $d$-cluster tilting subcategory $\mc \subseteq \mod A$ for some $d \geq 1$ (see \cref{sec:background} for precise definitions). Inside $\mc$ we can study both the maximal $\tau_d$-rigid pairs and the functorially finite $d$-torsion classes. Note that for $d=1$, a maximal $\tau_d$-rigid pair is usually known as a support $\tau$-tilting pair; see \cref{rem: terminology support}. In the classical case, there is a bijection between these two sets, where the support $\tau$-tilting pair associated to a functorially finite torsion class $\tc$ may be constructed using a so-called $\Ext$-projective generator of $\tc$ \cite{AdachiIyamaReiten}. We obtain the following higher analogue, relating $d$-torsion classes to the maximal $\tau_d$-rigid pairs introduced in \cite{JJ}.

\begin{theorem}[\cref{thm:maximaltaunrigid}, \cref{prop: injective,cor: maximality}] \label{MainResultIntro}
There is an \mbox{injective map} 
\begin{align*}
\renewcommand{\arraystretch}{1.8}
\begin{array}{ccc}
\renewcommand{\arraystretch}{1.1} \phi_d \colon
\begin{Bmatrix}
    \text{functorially finite} \\
    \text{$d$-torsion classes in $\mc$}
\end{Bmatrix}
&
    \xrightarrow{}
    &
    \renewcommand{\arraystretch}{1.2}
\begin{Bmatrix}
\text{basic maximal $\tau_d$-rigid pairs} \\
\text{in $\mc$ with $|A|$ summands}
\end{Bmatrix} \\ 
\uc &\mapsto  &(M_\uc, P_\uc)
\end{array} 
\renewcommand{\arraystretch}{1}
\end{align*}
where:
\theoremlistfix
\begin{enumerate}
    \item $M_\uc$ is a basic $\Ext^d$-projective generator of $\uc$.
    \item $P_\uc$ is the maximal basic projective module satisfying $\Hom_A(P_\uc,-)|_\uc = 0$. 
\end{enumerate}
\end{theorem}

In the result above, we say that a $\tau_d$-rigid pair $(M,P)$ has $|A|$ summands if $|M|+|P|=|A|$, where $|X|$ denotes the number of pairwise non-isomorphic indecomposable direct summands of a module $X$. Note that while the map $\phi_d$  looks very similar to the classical bijection in \cite{AdachiIyamaReiten}, and in fact recovers that map for $d=1$, \cref{Example:Phi_dNotSurj} demonstrates that $\phi_d$ is typically not surjective for $d>1$. 

Let $\uc \subseteq \mc$ be a functorially finite $d$-torsion class. The construction of the pair $(M_\uc, P_\uc)$ in \cref{MainResultIntro} 
follows a similar approach as in the classical result. However, proving that the pair is maximal $\tau_d$-rigid requires a very different strategy that is sketched below. Along the way, we lift several key intermediate results to the higher setting, such as \cref{thm:intro d-tilting} below, which generalises parts of \cite{Smalo}. Here, a tilting module of projective dimension less than or equal to $d$ is called $d$-tilting. 

\begin{theorem}[\cref{thm: d-tilting}]
\label{thm:intro d-tilting}
The module $M_\uc$ is $d$-tilting over $A/\Ann \uc$.
\end{theorem}

An application of \Cref{thm:intro d-tilting} is given in \Cref{sec:slices}, where we generalise the notion of a slice from \cite{IyamaOppermann} by removing the assumption that the algebra has global dimension $d$. In \Cref{Corollary:SliceIsTilting},
we see that these slices give rise to $d$-tilting modules, and hence generate new derived equivalences of algebras; see \Cref{Example:Slice1} and \Cref{Example:Slice2}.

Since tilting modules have a fixed number of summands, it follows from the above result that $|M_\uc| + |P_\uc| =|A|$; see \cref{thm:maximaltaunrigid}. Unlike in the classical case, it is not known whether this is enough to conclude that the $\tau_d$-rigid pair $(M_\uc,P_\uc)$ is maximal. To circumvent this problem, we turn to the realm of silting complexes. This translation is possible due to the following theorem.

\begin{theorem}[\cref{thm:silting}]\label{Theorem:Silting} 
Let $P_\bullet^{M_\uc}$ be the complex formed by the first $d+1$ terms in the minimal projective resolution of $M_\uc$. Then $P_\bullet^{M_\uc}\oplus P_\uc[d]$ is a silting complex in $K^b(\proj A)$.
\end{theorem}

\Cref{Theorem:Silting} allows us to take advantage of the strong maximality property satisfied by silting complexes (see \cref{lem:maximal}) in order to show that the $\tau_d$-rigid pair $(M_\uc, P_\uc)$ is maximal. This provides a crucial step in the proof of \cref{MainResultIntro}. 

It is worth noting that \Cref{Theorem:Silting} also provides a way to use a $d$-cluster tilting subcategory as a tool to study the ambient algebra itself, by generating silting complexes and hence $t$-structures and derived equivalences from $d$-torsion classes. Furthermore, the notion of $(d+1)$-term silting complexes for $d\geq 2$ has recently attracted significant attention \cite{ALLT,Gupta,GuptaZhou,MM,Zhou}, and \Cref{Theorem:Silting} can be seen as a part of this study. One big advantage of our approach is that we get explicit constructions of silting complexes by just working in the $d$-cluster tilting subcategory of the algebra. This is a much simpler task than working in its derived category, where silting complexes normally live. In particular, we provide explicit combinatorial descriptions of the silting complexes constructed in \Cref{Theorem:Silting} for higher Auslander and higher Nakayama algebras of \mbox{type $\mathbb{A}$}. These silting complexes give information on the associated derived categories, which in the case of higher Auslander algebras are equivalent to certain Fukaya categories \cite{DiDedda1, DiDedda2, DyckerhoffJassoLekili} and to the derived categories of a family of fractionally Calabi-Yau lattices \cite{Gottesman}. Our combinatorial descriptions, see \Cref{cor: higher nakayama}, are implemented in a Python code, allowing the computation of a plethora of explicit examples; see \cref{rem: code}. 

We can summarise the relationship between (pre)silting complexes, $d$-torsion classes and $\tau_d$-rigid pairs as in the following diagram:

\begin{equation}\label{diag:silting intro}
\begin{tikzpicture}[align=center, xscale = 3.6, yscale = 2.5, baseline=(current  bounding  box.center)]
\node (ff-torsion-M) at (1,1) 
	{$\begin{Bmatrix}
	\text{basic inner acyclic $(d+1)$-term} \\
	\text{silting complexes $S$ in $K^{b}(\proj A)$} \\
    \text{satisfying $H_0(S)\in\mc$}
	\end{Bmatrix} $};
\node (ff-torsion-modA) at (-1,1)  {$\begin{Bmatrix}
	\text{functorially finite} \\
	\text{$d$-torsion classes in $\mc$}\end{Bmatrix}$};
\node (maximal) at (1, 0)  {$\begin{Bmatrix}
	\text{basic inner acyclic $(d+1)$-term} \\
	\text{presilting complexes $S$ in $K^{b}(\proj A)$} \\
    \text{satisfying $H_0(S)\in\mc$ and $|S|=|A|$}
	\end{Bmatrix}$};
\node (support) at (-1, 0)  {$\begin{Bmatrix} 
	\text{basic maximal $\tau_d$-rigid pairs} \\
	\text{ in $\mc$ with $|A|$ summands}
	\end{Bmatrix}$};
\draw[right hook->] (ff-torsion-M) -- node[right] {} (maximal);
\draw[right hook->] (ff-torsion-modA) -- node[left] {$\phi_d$} (support);
\draw[right hook->] (support) -- node[left] {} (maximal);
\draw[right hook->] (ff-torsion-modA) -- node[above] {$\psi_d$} (ff-torsion-M);
\end{tikzpicture}
\end{equation}

\noindent The bottom map in the diagram is given by sending a maximal $\tau_d$-rigid pair $(M,P)$ to $P_\bullet^M \oplus P[d]$, where $P_\bullet^M$ is the complex formed by the first $d+1$ terms in the minimal projective resolution \mbox{of $M$}. This yields a presilting complex in $K^{b}(\proj A)$; see \cref{prop:presilting}. The map $\psi_d$ is defined by sending a functorially finite $d$-torsion class to the associated silting complex in $K^{b}(\proj A)$ as in \cref{Theorem:Silting}. It follows directly from the definitions of the maps that the diagram commutes. The map $\psi_d$ is injective, but in general not surjective; see \cref{prop: injective silting} and \cref{ex:running7}.

For $d=1$, \cref{Theorem:Silting} is known via the bijection between support $\tau$-tilting pairs and $2$-term silting complexes from \cite{AdachiIyamaReiten}. However, the proof of this result does not generalise well to the higher setting as certain key properties of $2$-term silting complexes remain unknown for silting complexes of higher length, as discussed in \cref{rmk:howmanyisenough}. We therefore need a different strategy to prove \cref{Theorem:Silting} generally. The central ingredient is given as \Cref{Thm:dCT intro} below. In particular, this result exhibits novel examples of $d$-cluster tilting subcategories, which are interesting in their own right. Note that here we think about the subcategory $C^{[-d,0]}(\proj A)$ of $C^b(\proj A)$, consisting of complexes concentrated in degrees $-d,\ldots,0$, as an exact category; see \cref{ssec:exact and d-tilting}.

\begin{theorem}[\cref{theorem:dCT exact}]\label{Thm:dCT intro} 
The category 
\begin{align*}
\pc_d(\mc) & = \{P_\bullet \in C^{[-d,0]}(\proj A) \mid H_0(P_\bullet) \in \mc \text{ and } H_i(P_\bullet)=0 \text{ for } i=1,\dots,d-1 \} 
\end{align*}
is a $d$-cluster tilting subcategory of $C^{[-d,0]}(\proj A)$.
\end{theorem}

We moreover show that the essential image of $\pc_d(\mc)$ in the homotopy category $K^{[-d,0]}(\proj A)$, with its extriangulated structure, is a $d$-cluster tilting subcategory in the sense of \cite{HerschendLiuNakaokaII}; see \cref{theorem: d-CT extriangulated}.

Note that it follows from \cref{Thm:dCT intro} that $\pc_d(\mc)$ is a $d$-exact category in the sense of \cite{Jasso}. We introduce the definition of a $d$-torsion class in the exact setup, see \cref{ssec:d-torsion in exact setup}, generalising the notion from \cite{Jorgensen}. We believe this is interesting in its own right, in particular since $d$-exact categories occur much more frequently than their $d$-abelian counterparts; see \cite[Corollary 3.19]{AHJKPT1} as well as e.g.\ \mbox{\cite{Hanihara,HerschendIyamaMinamotoOppermann,HerschendIyamaOppermann,Iyama2011,IyamaWemyss,IyamaYoshino,HauglandJacobsenSchroll,SandoyThibault}}. Our key motivation, however, is to show that $d$-torsion classes in $\mc$ correspond bijectively to faithful $d$-torsion classes in $\pc_d(\mc)$, see \cref{thm: dtorsion to faithfuldtorsion}, as this bijection provides a key step in the proof of \cref{Theorem:Silting}.

As a final remark, we illustrate how the techniques used for the construction in \cref{MainResultIntro} build a bridge between classical $\tau$-tilting theory and its higher-dimensional analogue. This connection is of a similar nature as the link between torsion classes and their higher counterparts established in \cite{AJST}. We restrict our attention to the torsion classes $\tc$ in $\mod A$ for which $\tc \cap \mc$ is a $d$-torsion class and where torsion and $d$-torsion subobjects coincide. In this case, we say that $\tc$ induces a $d$-torsion class in $\mc$; see \cref{def: induces}. This property is characterised homologically in \cref{Proposition:CharacterizationInducingdTorsion}, where we show that a torsion class $\tc$ induces a $d$-torsion class if and only if $\Ext^{d-1}_A(t M,f M')=0$ and $t M \in \mc$ for all $M,M' \in \mc$. The connection between $\tau$-tilting and $\tau_d$-tilting theory is illustrated in the following diagram, where the map $\alpha$ takes a support $\tau$-tilting pair $(T,P)$ to the sum of the non-isomorphic indecomposable summands of the object $\bigoplus_{i=0}^{d-1}U_i^T$ obtained from the minimal $\mc$-coresolution 
$
0\to T\to U_0^T\to \cdots \to U_{d-1}^T\to 0
$
of $T$: 

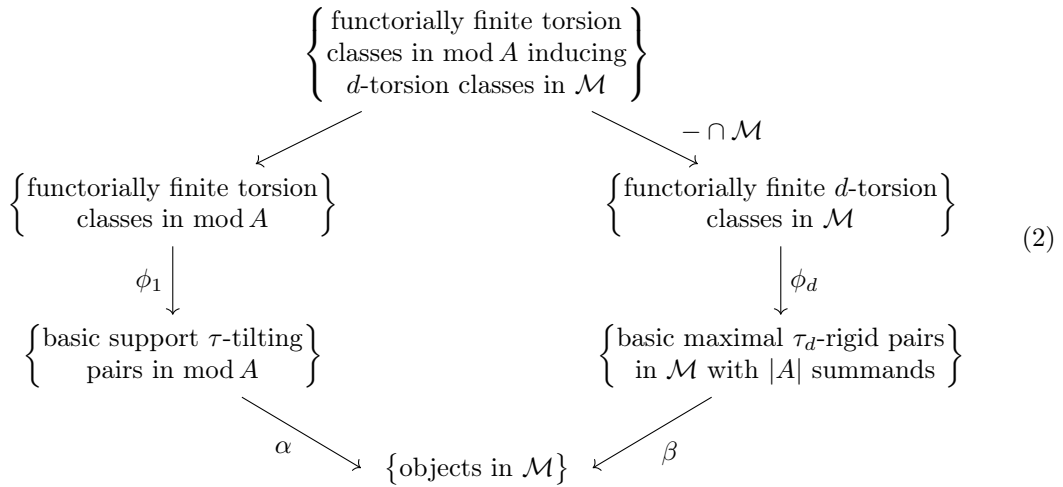
\begin{equation}\label{diagram:introduction}
\begin{tikzpicture}[align=center, xscale = 4, yscale = 2, baseline=(current  bounding  box.center)]
\node (ff-torsion-ext) at (0,3) 
	{$\begin{Bmatrix}
	\text{functorially finite torsion} \\
	\text{classes in $\modA$ inducing}\\
	\text{$d$-torsion classes in $\mc$}
	\end{Bmatrix} $};
\node (ff-torsion-M) at (1,2) 
	{$\begin{Bmatrix}
	\text{functorially finite $d$-torsion} \\
	\text{classes in $\mc$}
	\end{Bmatrix} $};
\node (ff-torsion-modA) at (-1,2)  {$\begin{Bmatrix}
	\text{functorially finite torsion} \\
	\text{classes in $\modA$}\end{Bmatrix}$};
\node (maximal) at (1, 1)  {$\begin{Bmatrix}
	\text{basic maximal $\tau_d$-rigid pairs} \\
	\text{ in $\mc$ with $|A|$ summands}
	\end{Bmatrix}$};
\node (support) at (-1, 1)  {$\begin{Bmatrix} 
	\text{basic support $\tau$-tilting} \\
	\text{ pairs in $\modA$ } 
	\end{Bmatrix}$};
\node (M) at (0, .25)  {$\begin{Bmatrix} \text{objects in } \mc  \end{Bmatrix}$};
\node (Mlt) at (-0.35,.22) {};
\node (Mrt) at (.35,.22) {};
\draw[->] (ff-torsion-ext) -- node[above right, pos = .75]{$-\cap\mc$} (ff-torsion-M);
\draw[->] (ff-torsion-ext) -- (ff-torsion-modA);
\draw[->] (ff-torsion-M) -- node[right] {$\phi_d$} (maximal);
\draw[->] (ff-torsion-modA) -- node[left] {$\phi_1$} (support);
\draw[->] (maximal) -- node[below right] {$\beta$} (Mrt);
\draw[->] (support) -- node[below left] {$\alpha$} (Mlt);
\end{tikzpicture}
\end{equation}

\noindent The map $\beta$ above takes a $\tau_d$-rigid pair $(M,P)$ to $M$. If a torsion class $\tc$ is functorially finite in $\mod A$, then $\tc \cap \mc$ is functorially finite in $\mc$ by \cref{prop:T ff implies U ff}. Hence, all the maps above are well-defined, and the diagram commutes by \cref{cor:sameeitherway}. We furthermore show that when $\tc$ is given by an APR tilting module, then we end up with the associated $d$-APR tilting module introduced in \cite{IyamaOppermann} at the bottom of the diagram; see \Cref{subsec: d-APR}. 

\addtocontents{toc}{\SkipTocEntry}
\subsection*{Structure of the paper}
We start by giving an overview of necessary background and terminology needed in the rest of the paper in \cref{sec:background}. In \cref{sec: char. results} we characterise when a classical torsion class induces a higher torsion class; see \cref{Proposition:CharacterizationInducingdTorsion}. We observe that any split torsion class gives rise to a split $d$-torsion class, as introduced in \cref{def:split d-torsion}. Furthermore, we present a new characterisation of $d$-torsion classes in terms of right $d$-exact sequences; see \cref{Reformulation:d-Tors}. In \cref{sec: ff-d-torsion} we restrict our attention to functorially finite $d$-torsion classes, for which we use approximations to construct $\Ext^d$-projective objects. These are again used in \cref{sec:d-tors to tau_d,sec: tau-tilting to tau_d-tilting} to construct $\tau_d$-rigid pairs associated to $d$-torsion classes and support $\tau$-tilting modules, respectively. In \cref{sec: tau-tilting to tau_d-tilting} we moreover apply the theory of functorially finite $d$-torsion classes to study $d$-APR tilting modules and slices. In \cref{subsec:dCT} we give the proof of \cref{Thm:dCT intro}, preparing the ground for proving \cref{Theorem:Silting} in \cref{sec:silting}. As a consequence, this allows us to conclude that the $\tau_d$-rigid pair induced by a functorially finite $d$-torsion class is indeed maximal, finishing the proof of \cref{MainResultIntro}. Finally, in \cref{subsec: higher nakayama} we employ our theory to give explicit combinatorial descriptions of maximal $\tau_d$-rigid pairs and silting complexes associated to functorially finite $d$-torsion classes in the setup of higher Auslander and higher Nakayama algebras of type $\mathbb{A}$. We furthermore show that there is an injective, but in general not surjective, map from the functorially finite $d$-torsion classes associated to a higher Auslander algebra to Oppermann--Thomas cluster tilting objects in the $d$-dimensional cluster category.

\addtocontents{toc}{\SkipTocEntry}
\subsection*{Conventions and notation}

Throughout this paper, we fix a positive integer $d$ and we let $A$ be a finite-dimensional algebra over a field $\kk$. We write $\modA$ for the category of finitely presented right $A$-modules and $\proj A \subseteq \mod A$ for the subcategory consisting of projective modules. We let $\operatorname{pd}_A M$ denote the projective dimension of a module $M \in \modA$ and write $DM \colonequals \Hom_{\kk}(M,\kk)$ for the $\kk$-dual of $M$. We use lower indices to denote a complex $C_\bullet=(\cdots \to C_{-1}\to C_0 \to C_1\to\cdots)$. The notation $C^b(\proj A)$ is used for the category of bounded complexes in $\proj A$, while we write $K^b(\proj A)$ for the homotopy category of bounded complexes in $\proj A$. The number of pairwise non-isomorphic indecomposable summands of an object $X$ in $\mod A$ or $K^b(\proj A)$ is denoted \mbox{by $|X|$}, and we write $\add(X)$ for the smallest subcategory which contains $X$ and is closed under finite direct sums and direct summands. 

We compose arrows in a quiver from left to right, i.e.\ we write $ab$ for the path starting at the source of $a$ and ending in the target of $b$.  All subcategories are assumed to be full and closed under isomorphisms, finite direct sums and direct summands. 

\section{Background and preliminaries} \label{sec:background}
In this section we give an overview of key concepts and results that will be used in the rest of the paper. In particular, we introduce the setup for higher homological algebra and consider both the classical versions and higher analogues of objects we are interested in. We start with some basic terminology.

\subsection{Basic terminology} \label{ssec: basic terminology}

Let $\xc$ be a subcategory of an additive category $\ac$. 
Define 
\[
\fac \xc \colonequals \{Y \in \ac \mid \text{there exists an epimorphism } X\to Y \text{ with } X\in\xc\}. 
\]

A morphism $f \colon Y \rightarrow X$ with $X\in\xc$ is called a \textit{left $\xc$-approximation} of $Y \in \ac$ if any morphism $Y \rightarrow X'$ with $X'\in\xc$ factors through $f$. The subcategory $\xc$ is called \textit{covariantly finite} if every object in $\ac$ admits a left $\xc$-approximation. \textit{Right $\xc$-approximations} and \textit{contravariantly finite} subcategories are defined dually, and $\xc$ is called \textit{functorially finite} if it is both covariantly and contravariantly finite.

A morphism $f \colon X \rightarrow Y$ in $\ac$ is called \textit{left minimal} if any endomorphism $g$ of $Y$ satisfying $g \circ f = f$ is an isomorphism. This leads to the notion of a \textit{minimal left $\xc$-approximation}, meaning a left $\xc$-approximation which is also left minimal. \textit{Right minimal morphisms} and \textit{minimal right $\xc$-approximations} are defined dually. If $\ac$ is a Krull--Schmidt category, an object has a left (resp.\ right) $\xc$-approximation if and only if it has a minimal left (resp.\ right) $\xc$-approximation; see e.g.\ \cite[Example 2.1.25]{Krause22}.

The \textit{Jacobson radical} of $\ac$ is given by
\begin{align*}
    \Rad_\ac(X,Y)=\{ f \in \Hom_\ac(X,Y) \mid 1_X - g \circ f \ \text{is invertible for all $g \in \Hom_\ac(Y,X)$}\}
\end{align*}
for $X,Y \in \ac$. If $\ac=\modA$, we use the notation $\Rad_A(X,Y)= \Rad_\ac(X,Y)$. For more details and basic properties, see \cite[Section 2]{Krause}. The following observation connecting the radical and minimality of morphisms will be repeatedly used.

\begin{lemma}\label{lem: minimal iff radical}
Let $\ac$ be a Krull--Schmidt additive category and consider two composable morphisms $X\xrightarrow{f}Y\xrightarrow{g}Z$ in $\ac$. The following statements hold:
\theoremlistfix
\begin{enumerate}
    \item If the induced sequence
\[
\Hom_{\ac}(Y,X)\to \Hom_{\ac}(Y,Y)\to\Hom_{\ac}(Y,Z)
\] 
is exact, then $f\in \Rad_{\ac}(X,Y)$ if and only if $g$ is right minimal.
\item If the induced sequence 
\[
\Hom_{\ac}(Z,Y)\to \Hom_{\ac}(Y,Y)\to\Hom_{\ac}(X,Y)
\]
is exact, then $g\in\Rad_{\ac}(Y,Z)$ if and only if $f$ is left minimal.
\end{enumerate}
\end{lemma}

\begin{proof}
    This follows by the proof of \cite[Lemma 1.1]{higher survey}.
\end{proof}

In the case where $\ac = \modA$, the \textit{annihilator} of a subcategory $\xc \subseteq \ac$ is defined to be $$\Ann \xc=\{a\in A\mid X\cdot a =(0) \text{ for all }X\in \xc\}.$$ The subcategory $\xc$ is \textit{faithful} if $\Ann\xc=(0)$. Note that this is equivalent to the existence of a monomorphism $A\to X$ with $X\in \xc$. A module $M\in \modA$ is called faithful if $\operatorname{add}(M)$ is a faithful subcategory of $\modA$.

A pair $(\tc, \fc)$ of subcategories of $\modA$ is called a \textit{torsion pair} if $\Hom_{A}(X,Y)=0$ for all $X \in \tc$ and $Y \in \fc$, and if for every $X \in \modA$, there exists a short exact sequence 
\[
0\to tX \rightarrow X \rightarrow fX \to 0
\]
with $tX\in\mathcal{T}$ and $fX\in\mathcal{F}$. This short exact sequence, which is unique up to isomorphism, is referred to as the \textit{canonical short exact sequence} of $X$ with respect to $(\tc,\fc)$. The object $tX$ is known as the \textit{torsion subobject} of $X$. Given a torsion pair $(\tc, \fc)$, the subcategories $\tc$ and $\fc$ are called a \textit{torsion class} and \textit{torsion free class}, respectively. It follows from the definition that 
\begin{align*}\tc&=\{X\in \modA\mid \Hom_A(X,Y)=0 \text{ for all } Y\in \fc\} \\
\fc&=\{Y\in \modA\mid \Hom_A(X,Y)=0 \text{ for all } X\in \tc\}.
\end{align*} 
Hence, a torsion pair is completely determined by its torsion class or torsion free class.

An $A$-module $T$ is called \textit{tilting} if it has finite projective dimension, $\Ext^i_A(T,T)=0$ for all $i>0$, and there exists an exact sequence 
\[
0 \to A \to T_0 \to T_1 \to \cdots \to T_d \to 0
\]
with $T_i \in \add (T)$ for $i=0,\dots,d$. It is furthermore called $d$\textit{-tilting} if $\operatorname{pd}_A T \leq d$. If $T \in \mod A$ is a tilting module, then $|T|=|A|$ \cite[Section 2.1]{Bongartz}.

\subsection{Higher homological algebra}\label{sec:higher world}
Throughout the rest of this section, we work in the module category $\mod A$. Note that higher homological algebra in an exact setup will be discussed in \cref{subsec:dCT,sec:silting}.

We first recall the definition of a $d$-cluster tilting subcategory of $\mod A$. This notion was introduced in \cite{Iyama2007a,Iyama2007b} and forms the basic setting for higher homological algebra.

\begin{definition}\label{def:d-CT abelian}
A functorially finite subcategory $\mc$ of $\modA$ is \textit{$d$-cluster tilting} if
\begin{align*}
\mc  &= \{X \in \modA \mid \Ext_{A}^i(X,M)=0 \text{ for all $M \in \mc$ and all $i = 1, \dots, d-1$}\}\\
&= \{Y \in \modA \mid \Ext_{A}^i(M,Y)=0 \text{ for all $M \in \mc$ and all $i = 1, \dots, d-1$}\}.
\end{align*}
\end{definition}

\begin{remark}
    The extension vanishing property in \cref{def:d-CT abelian} leads to a situation where the shortest non-trivial exact sequences in a $d$-cluster tilting subcategory $\mc$ have $d$ middle terms, and such sequences play a key role in the study of $\mc$. These ideas were made precise in \cite{Jasso} (see also \cite{GeissKellerOppermann}), through the notion of \textit{$d$-abelian} categories. Note in particular that any $d$-cluster tilting subcategory $\mc \subseteq \modA$ is $d$-abelian \cite[Theorem 3.16]{Jasso}. It is furthermore shown in \cite[Theorem 4.7]{Ebrahimi} and \cite[Theorem 7.3]{Kvamme} that a converse of this result also holds, namely that any $d$-abelian category $\mc$ embeds into an ambient abelian category as a $d$-cluster tilting subcategory. Moreover, this abelian category is unique up to equivalence by \cite[Proposition 4.8]{Kvamme}.
\end{remark}

If $\mc \subseteq \modA$ is a $d$-cluster tilting subcategory, then any $M\in \modA$ admits exact sequences
\[
0\to X_{-d}\to \cdots \to X_{-1}\to M\to 0 \quad \text{and} \quad 0\to M\to Y_1\to \cdots \to Y_d\to 0\quad 
\]
with $X_{-d},\dots, X_{-1},Y_1,\dots,Y_d\in \mc$; see \cite[Proposition 3.17]{Jasso}. We call these sequences an \mbox{$\mc$\textit{-resolution}} and an $\mc$\textit{-coresolution} of $M$, respectively.

A \textit{$d$-exact sequence} or \textit{$d$-extension} in a $d$-cluster tilting subcategory $\mc \subseteq \mod A$ is a complex 
\begin{equation}\label{eq: d ext}
0 \to X_0\xrightarrow{}X_1 \xrightarrow{} X_2 \xrightarrow{} \cdots \xrightarrow{} X_{d} \xrightarrow{} X_{d+1} \to 0
\end{equation}
with $X_i \in \mc$ for $i=0,\dots,d+1$ which is exact in the ambient category $\mod A$. Note that this notion of a $d$-exact sequence in $\mc$ agrees with \cite[Definition 2.4]{Jasso}; see \cite[Proposition 2.9]{HerschendJorgensen}.

Similarly, an exact complex
\[
X_0\xrightarrow{f_0}X_1 \xrightarrow{f_1} X_2 \xrightarrow{f_2} \cdots \xrightarrow{f_{d-1}} X_{d} \xrightarrow{f_d} X_{d+1} \to 0
\]
in $\mc$ is called a \textit{right $d$-exact sequence}. In this case, we say that the sequence
\[
X_1 \xrightarrow{f_1} X_2 \xrightarrow{f_2} \cdots \xrightarrow{f_{d-1}} X_{d} \xrightarrow{f_d} X_{d+1} \to 0
\]
is a \textit{$d$-cokernel} of $f_0$ in $\mc$. Such a $d$-cokernel is sometimes simply denoted by the $d$-tuple $(f_1, \dots, f_{d})$. The notions of \textit{left $d$-exact sequences} and \textit{$d$-kernels} are defined dually. A \textit{$d$-quotient} of $X_1 \in \mc$ is a $d$-cokernel $(f_1, \dots, f_{d})$ of some morphism $f_0\colon X_0\to X_1$ in $\mc$.

Note that if an exact sequence in $\mod A$ with $d$ middle terms has end terms in $\mc$, then the sequence is Yoneda equivalent to a $d$-exact sequence with all terms in $\mc$; see \cite[Proposition 2.9]{HerschendJorgensen} and \cite[Proposition A.1]{Iyama2007a} for details. Furthermore,  a $d$-exact sequence in $\mc$ of the form \eqref{eq: d ext} is Yoneda equivalent to a $d$-exact sequence 
\[
0 \to X_0 \rightarrow X'_1 \rightarrow \cdots \rightarrow X'_{d} \rightarrow X_{d+1} \to 0
\]
in $\mc$ if and only if there exists a commutative diagram
\[
	\begin{tikzcd}[column sep=22, row sep=30]
	0 \arrow[r] & X_0 \arrow[r] \arrow[d, equal] & X_1 \arrow[r] \arrow[d] & \cdots \arrow[r] & X_d \arrow[r] \arrow[d] & X_{d+1} \arrow[d, equal]  \arrow[r]  & 0\phantom{.}\\
	0 \arrow[r] & X_0 \arrow[r] & X'_1 \arrow[r] & \cdots \arrow[r] & X'_{d} \arrow[r] & X_{d+1} \arrow[r]  & 0
	\end{tikzcd}
\]
by \cite[Proposition 2.9]{HerschendJorgensen}. In this case, we simply say that the sequences are \textit{equivalent}. 

Since $d$-kernels and $d$-cokernels are unique only up to homotopy, we will make use of the following notion of minimality.

\begin{definition}\cite[Definition 2.5]{HerschendJorgensen}\label{def: minimal}
    Let $\mc \subseteq \mod A$ be a $d$-cluster tilting subcategory.
    \theoremlistfix
    \begin{enumerate}
        \item A $d$-cokernel
        \(
        X_1 \xrightarrow{f_1} X_2 \xrightarrow{f_2} \dots \xrightarrow{f_{d-1}} X_{d} \xrightarrow{f_d} X_{d+1} \to 0
        \)
        of a morphism $f_0 \colon X_0 \to X_1$ in $\mc$ is \textit{minimal} if we have $f_i \in \Rad_{A}(X_i,X_{i+1})$ for \mbox{$i=2,\ldots,  d$}.
        \item A $d$-kernel
        \(
        0 \to X_0 \xrightarrow{f_0} X_1 \xrightarrow{f_1} \dots \xrightarrow{f_{d-2}} X_{d-1} \xrightarrow{f_{d-1}} X_{d}
        \)
        of a morphism $f_d \colon X_d \to X_{d+1}$ in $\mc$ is \textit{minimal} if we have $f_i \in \Rad_{A}(X_i,X_{i+1})$ for $i=0,\ldots,  d-2$.
        \item A $d$-extension
        \(
        0 \to X_0 \xrightarrow{f_0} X_1 \xrightarrow{f_1} \dots \xrightarrow{f_{d-1}} X_{d} \xrightarrow{f_d} X_{d+1}\to 0
        \)
        in $\mc$ is \textit{minimal} if we have $f_i \in \Rad_{A}(X_i,X_{i+1})$ for $i=1,\ldots, d-1$.
    \end{enumerate}
\end{definition}

It should be noted that the definition above could equivalently have been stated in terms of minimality of morphisms; see \cref{lem: minimal iff radical}. For instance, the $d$-cokernel in part (1) is minimal if and only if the morphisms $f_1,\dots,f_{d-1}$ are left minimal. 

Each equivalence class of \mbox{$d$-extensions} in $\mc$ contains a unique \textit{minimal} representative, which is a direct summand of any other equivalent $d$-exact sequence. Similar statements hold for $d$-kernels and $d$-cokernels, and we refer the reader to \cref{prop:exact analogue of HJ}, as well as \cite[Proposition 2.4]{HerschendJorgensen} and \cite[Section 2.2]{AHJKPT1}, for more details. A \textit{minimal \mbox{$d$-quotient}} of $X_1\in \mc$ is defined to be a minimal $d$-cokernel of a morphism $X_0\to X_1$ in $\mc$. We will always choose to work with minimal $d$-exact sequences, $d$-kernels, \mbox{$d$-cokernels} and $d$-quotients.

A central notion in this paper is the following generalisation of torsion classes.

\begin{definition}\cite[Definition 1.1]{Jorgensen}\label{def:ntorsionclass} 
Let $\mc$ be a $d$-cluster tilting subcategory of $\modA$. A subcategory $\uc$ of $\mc$ is called a \textit{$d$-torsion class} if for every $M$ in $\mc$, there exists a $d$-exact sequence
\begin{equation*}
  0 \rightarrow U_M \rightarrow M \rightarrow V_1 \rightarrow \cdots \rightarrow V_d \rightarrow 0
\end{equation*}
in $\mc$ such that the following hold:
\theoremlistfix
\begin{enumerate}
    \item The object $U_M$ is in $\uc$.
    \phantomsection
    \label{def:dtors:approx}
    \item The sequence $0 \xrightarrow{} \Hom_A(U,V_1) \xrightarrow{  } \cdots \xrightarrow{  } \Hom_A(U,V_d) \xrightarrow{} 0$ is exact for every $U$ \mbox{in $\uc$}.
    \phantomsection
    \label{def:dtors:sequence}
\end{enumerate}
The object $U_M$ above is unique up to isomorphism, and is called the \textit{$d$-torsion subobject} of $M$ with respect to $\uc$. 
\end{definition}
Note that for $d=1$, \cref{def:ntorsionclass} is equivalent to the classical definition of a torsion class. However, for $d>1$, there is no obvious higher analogue of the corresponding torsion free class. 

In \cite{AHJKPT1} we showed that $d$-torsion classes can be characterised using $d$-extensions and $d$-quotients. To state this result, we need the following notions.
Let $\mc \subseteq \modA$ be a $d$-cluster tilting subcategory and consider a subcategory $\uc \subseteq \mc$. The subcategory $\uc$ is said to be \textit{closed under $d$-extensions} if for every pair of objects $X,Z\in \uc$ and every minimal $d$-extension
\[
0 \xrightarrow{} X \xrightarrow{} Y_1 \xrightarrow{} \cdots \xrightarrow{} Y_d \xrightarrow{} Z \xrightarrow{} 0,
\]
we have $Y_i \in \uc$ for $i=1,\dots,d$. Similarly, we say that $\uc$ is \textit{closed under $d$-quotients} if for every minimal $d$-quotient 
\[
X \xrightarrow{} Y_1 \xrightarrow{} \cdots \xrightarrow{} Y_{d} \to 0
\]
of an object $X \in \uc$, we have $Y_i \in \uc$ for $i=1,\dots,d$. Note that these definitions are equivalent to the ones given in \cite[Section 3.2]{AHJKPT1} by \cite[Lemma 3.8]{AHJKPT1}. The following result is a generalisation of \cite[Theorem 2.3]{Dickson} to higher torsion classes, and will be used frequently.

\begin{theorem} \cite[Theorem 3.14]{AHJKPT1}\label{theorem: extension closed}
Let $\mc$ be a $d$-cluster tilting subcategory of $\modA$. A subcategory of $\mc$ is a $d$-torsion class if and only if it is closed under $d$-extensions and $d$-quotients.
\end{theorem} 

\subsection{Higher $\tau$-tilting theory}\label{sec: tilting background}

To mirror mutation properties seen for cluster algebras, a generalisation of tilting modules, known as $\tau$-tilting modules, was introduced in \cite{AdachiIyamaReiten}. We recall relevant notions from \mbox{$\tau$-tilting} theory, together with their higher analogues.

The $d$-\textit{Auslander--Reiten translation} of $M \in \modA$, denoted by $\tau_d M$, is defined by taking a minimal projective resolution
\[
\cdots \xrightarrow{}P_{-d}\xrightarrow{f_{-d}}P_{-(d-1)}\xrightarrow{f_{-(d-1)}}\cdots \xrightarrow{f_{-1}}P_0\to M
\]
and setting $\tau_d M \colonequals \Ker \nu(f_{-d})$, where $\nu=-\otimes_ADA$ denotes the Nakayama functor. This induces a functor
\[
\tau_d\colon \underline{\operatorname{mod}}\, A\to \overline{\operatorname{mod}}\, A,
\]
where $\underline{\operatorname{mod}}\, A$ and $\overline{\operatorname{mod}}\, A$ denote the stable category modulo projectives and injectives, respectively. As a functor, we have $\tau_d\cong \tau\circ \Omega^{d-1}$, where $\Omega\colon \underline{\operatorname{mod}}\, A\to \underline{\operatorname{mod}}\, A$ denotes Heller's syzygy functor and $\tau\colon \underline{\operatorname{mod}}\, A\to \overline{\operatorname{mod}}\, A$ is the classical Auslander--Reiten translation. We use the notation $\overline{\Hom}_A(-,-)$ for morphism spaces in $\overline{\operatorname{mod}}\, A$. Note that we have an isomorphism
\begin{equation*}
    D\overline{\Hom}_A(Y, \tau_d X)\cong \Ext^d_A(X,Y)
\end{equation*} 
given by $D\overline{\Hom}_A(Y, \tau_d X) \cong D\overline{\Hom}_A(Y, \tau\Omega^{d-1} X)\cong \Ext^1_A(\Omega^{d-1}X,Y) \cong\Ext^d_A(X,Y)$.
This isomorphism is called the \textit{higher Auslander--Reiten duality.}

Using the higher Auslander--Reiten translation, we may define higher analogues of notions from $\tau$-tilting theory. Let us first recall what it means for a module to be $\tau_d$-rigid. 

\begin{definition}\label{def:tau_d-rigid}
    A module $M \in \mod A$ is called \textit{$\tau_d$-rigid} if $\Hom_A(M,\tau_dM)=0$.
\end{definition}

Note that the notion of $\tau_d$-rigidity in \cite{MM} is stronger than that given in \cref{def:tau_d-rigid}, as they additionally require $\Ext_A^i(M,M)=0$ for $i=1,\dots,d-1$; see \cite[Definition 1.1]{MM}. This property is automatic whenever $M$ is contained in a $d$-cluster tilting subcategory.

We have the following simple generalisation of a classical result in $\tau$-tilting theory that characterises $\tau_d$-rigid modules via vanishing extensions.

\begin{proposition}\label{lem:char tau_d-rigid modules}
    Let $M$ and $M'$ be modules in $\modA$. 
    Then $\Hom_A(M, \tau_d M')=0$ if and only if $\Ext^d_A(M', N)=0$ for all $N \in \fac M$. In particular,  $M$ is $\tau_d$-rigid if and only if $\Ext^d_A(M,N)=0$ for all $N\in \fac M$.
\end{proposition}

\begin{proof}
    By definition, we have $\Hom_A(M, \tau_d M')= \Hom_A(M, \tau\Omega^{d-1} M')$. 
    It follows from \cite[Proposition 5.8]{AS} that $\Hom_A(M, \tau\Omega^{d-1} M')=0$ if and only if $\Ext^1_A(\Omega^{d-1}M', N) = 0$ for all $N \in \fac M$. 
    The claim therefore follows since the functor $\Ext^1_A(\Omega^{d-1}M', - )$ is isomorphic to $\Ext^d_A(M', - )$.
\end{proof}

In particular, the result above shows that any $\tau_d$-rigid module $M$ satisfies \mbox{$\Ext_A^d(M,M)=0$}. Next we state a result that will help us relate $\tau_d$-rigidity back to $d$-tilting modules. 

\begin{lemma}\cite[Lemma 5.1]{MM}\label{Martinez-Mendoza}
Let $M \in \modA$ be a faithful $\tau_d$-rigid module, which additionally satisfies $\Ext^{i}_A(M,M)=0$ for $i=1,\dots,d-1$. Then $\operatorname{pd}_A M \leq d$.
\end{lemma}

If $\mc$ is a $d$-cluster tilting subcategory of $\modA$, then $\tau_d$ restricts to an equivalence $\tau_d\colon\underline{\mc}\to \overline{\mc}$, where $\underline{\mc}$ and $\overline{\mc}$ are the subcategories of $\underline{\operatorname{mod}}\, A$ and $\overline{\operatorname{mod}}\, A$ consisting of the objects in $\mc$ \cite[Theorem 2.3]{Iyama2007a}. In particular, if $M\in \mc$ is indecomposable, then $\tau_dM\in \mc$ is indecomposable. One can consider $\tau_d$-rigidity inside $d$-cluster tilting subcategories. This was done in \cite{JJ}, where they introduce a higher homological analogue of (maximal) $\tau$-rigid pairs; see also \cite{ZZ}.

\begin{definition}\label{def: tau-d-maximality}
    Let $\mc$ be a $d$-cluster tilting subcategory of $\modA$ and consider a pair $(M,P)$ with $M \in \mc$ and $P \in \proj A$. 
    \theoremlistfix
    \begin{enumerate}
        \item $(M,P)$ is called a \textit{$\tau_d$-rigid pair in $\mc$} if $M$ is $\tau_d$-rigid and $\Hom_A(P,M)=0$.
        \item $(M,P)$ is called a \textit{maximal $\tau_d$-rigid pair in $\mc$} if it satisfies:
        \theoremlistfix
        \begin{enumerate}[label=(\roman*)]
            \item If $N$ is in $\mc$, then
            \[
            N \in \add(M) \iff \begin{cases}
            \Hom_A(M,\tau_d N) = 0, \\
            \Hom_A(N,\tau_d M) = 0, \\
            \Hom_A(P,N) = 0.
\end{cases}
            \]
            \item If $Q$ is in $\proj A$, then
            \[
            Q \in \add(P) \iff \Hom_A(Q,M) = 0.
            \]
        \end{enumerate}
    \end{enumerate}
\end{definition}
Notice that if $(M,P)$ is a maximal $\tau_d$-rigid pair in $\mc$, then it is in particular a $\tau_d$-rigid pair and $M$ is a $\tau_d$-rigid module. We say that a maximal $\tau_d$-rigid pair $(M,P)$ is \textit{basic} if $M$ and $P$ are basic.

\begin{remark}
    Being a maximal $\tau_d$-rigid pair in $\mc$ is slightly stronger than just being maximal with respect to the property of being a $\tau_d$-rigid pair, since we do not need to check the condition $\Hom_A(N,\tau_dN)=0$ in (2)(i).
\end{remark}

\begin{remark}\label{rem: terminology support}
In the classical case, a $\tau$-rigid pair $(M,P)$ is maximal in the sense of the above definition if and only if $|M|+|P| = |A|$; see \cite[Corollary 2.13]{AdachiIyamaReiten}. The pair $(M,P)$ is then often called a \textit{support $\tau$-tilting pair}, in which case $M$ is known as a \textit{support $\tau$-tilting module}. Note however that the characterisation by the number of indecomposable summands does not hold for $d>1$, as there are maximal $\tau_d$-rigid pairs which do not have the required number of summands. For an example, see \cite[Section 4]{JJ} (in particular Figure 4). It is furthermore unknown whether having $|A|$ summands is sufficient to imply maximality.
\end{remark}

\subsection{Silting theory}\label{ssec:silting}

In this subsection we give an overview of some basic results for silting and presilting complexes. In particular, we describe a connection between $\tau_d$-rigid pairs and presilting complexes. 

Recall that $S \in K^b(\proj A)$ is \textit{presilting} if $$\Hom_{K^b(\proj A)}(S, S[i])=0$$ for all $i>0$, and is \textit{silting} if in addition the smallest thick subcategory of $K^b(\proj A)$ containing $S$ is $K^b(\proj A)$ itself. For this, it is sufficient to check that $A$ is contained in this thick subcategory. If $S \in K^b(\proj A)$ is a silting complex, then $|S|=|A|$ by \cite[Corollary 2.28]{AiharaIyama}.

\begin{remark}\label{rmk:howmanyisenough}
    There are recent examples of presilting objects $S$ with $|S|<|A|$ that cannot be completed to silting objects; see \cite{LZ,JSW}. There even exist presilting objects $S$ satisfying $|S|=|A|$ that are not silting \cite{Kalck}. Consequently, a converse to \cite[Corollary 2.28]{AiharaIyama} does not hold in general. However, it holds for \mbox{$2$-term} complexes, i.e.\ any \mbox{$2$-term} presilting complex $S \in K^b(\proj A)$ with $|S|=|A|$ is silting \cite[Proposition 3.3]{AdachiIyamaReiten}. We do not know if the analogous statement is true for inner acyclic $(d+1)$-term presilting complexes (see \cref{def:(d+1)-term silting}).
\end{remark}

The following lemma is well-known to experts, and should be considered as a maximality property for silting complexes similar to that in the definition of maximal \mbox{$\tau_d$-rigid} pairs. We include a proof for the convenience of the reader. Note that given two subcategories $\mathcal{X}, \mathcal{Y} \subseteq K^b(\proj A)$, we use the notation $\mathcal{X} \ast \mathcal{Y}$ for the subcategory
\begin{align*}
    \{ Z \in K^b(\proj A) \mid  \text{there exists a triangle} \ X \to Z \to Y \to X[1] \ \text{with} \ X \in \mathcal{X}, Y \in \mathcal{Y} \}.
\end{align*}

\begin{lemma}\label{lem:maximal}
Consider a silting complex $S \in K^b(\proj A)$ and let $C \in K^b(\proj A)$ be such that $\Hom_{K^b(\proj A)}(S, C[i])= 0$ and $\Hom_{K^b(\proj A)}(C, S[i])= 0$ for all $i > 0$. 
Then $C \in \add (S)$.
\end{lemma}

\begin{proof}
Set $\mathcal{S} \colonequals \add(S)$ and consider the subcategories
\begin{align*}
    \mathcal{S}^{<0} \colonequals \bigcup\limits_{\ell>0} \mathcal{S}[1] \ast \cdots \ast \mathcal{S}[\ell] \quad \text{and} \quad    \mathcal{S}^{\geq 0} \colonequals\bigcup\limits_{\ell\geq 0} \mathcal{S}[-\ell] \ast \cdots \ast \mathcal{S}[-1] \ast \mathcal{S}.
\end{align*}
Since $S$ (and thus also $S[1]$) is silting, it follows from \cite[Proposition 2.23]{AiharaIyama} that there exists a distinguished triangle
\[
X \xrightarrow[]{f} C\xrightarrow[]{g} Y \xrightarrow[]{}X[1]
\]
in $K^b(\proj A)$ with $X \in \mathcal{S}^{\geq 0}$ and $Y \in \mathcal{S}^{<0}$. The assumption that $\Hom_{K^b(\proj A)}(C, S[i])=0$ for all $i>0$ implies that $g=0$. Hence, $f \colon X \to C$ is a split epimorphism, giving $C \in \mathcal{S}^{\geq 0}$. Since $\Hom_{K^b(\proj A)}(S[-i], C) = 0$ for all positive $i$, we obtain $C \in \add(S)$ as required.
\end{proof}

We finish this subsection by giving some results relating silting theory and $\tau_d$-rigidity. Note that given a projective resolution 
\[
\cdots \xrightarrow{} P_{-i} \xrightarrow{} \cdots \xrightarrow{} P_{-2} \xrightarrow{} P_{-1} \xrightarrow{} P_0 \xrightarrow{} M \xrightarrow{} 0
\]
of an $A$-module $M$, we say that the object 
\[
\cdots \xrightarrow{} 0 \xrightarrow \quad P_{-d} \xrightarrow{} \cdots \xrightarrow{} P_{-2} \xrightarrow{} P_{-1} \xrightarrow{} P_0 \xrightarrow{} 0 \xrightarrow{} \cdots
\]
in $K^b(\proj A)$, where $P_0$ is in degree $0$, is a \textit{projective $d$-presentation} of $M$. In the case where the projective resolution is minimal, we denote this object by $P_\bullet^M$ and call it the \textit{minimal projective $d$-presentation} of $M$. Moreover, if $(M,P)$ is a $\tau_d$-rigid pair in a $d$-cluster tilting subcategory $\mc \subseteq \mod A$, then we denote by $P_\bullet^{(M,P)}$ the object
$$P_\bullet^{(M,P)}\colonequals P_\bullet^{M} \oplus P [d]$$ in $K^b(\proj A)$, where $P$ is considered as a stalk complex concentrated in degree $0$. An important observation is that the complex $P_\bullet^{(M,P)}$ is concentrated in degrees \mbox{$-d,\dots,0$} and
has homology only in degrees $-d$ and $0$, i.e.\ that $H_i(P_\bullet^{(M,P)})=0$ for $i \notin \{-d,0\}$. This motivates the following.

\begin{definition}\label{def:(d+1)-term silting}
\theoremlistfix
    \begin{enumerate}
        \item An object $X \in K^b(\proj A)$ is called a \textit{$(d+1)$-term complex} if it is concentrated in degrees $-d,\dots,0$. 
        \item A $(d+1)$-term complex $X \in K^b(\proj A)$ is called \textit{inner acyclic} if $H_i(X)=0$ for $i \notin \{-d,0\}$.
        \item A $(d+1)$-term complex $S \in K^b(\proj A)$ is called \textit{$(d+1)$-term (pre)silting} if it is a (pre)silting complex in $K^b(\proj A)$.
    \end{enumerate}
\end{definition}

Note that a $2$-term complex is automatically inner acyclic.

The result below was proved by Martinez and Mendoza in \cite{MM}. 

\begin{lemma}\cite[Theorem 3.4]{MM}\label{lem:tau-rigidKbProj}
For $M,N$ in $\mod A$, the following are equivalent:
\theoremlistfix
\begin{enumerate}
    \item $\Hom_A(M, \tau_d N)=0$ and $\Ext^i_A(N,M)=0$ for $i=1,\dots,d-1$;
    \item $\Hom_{K^b(\proj A)}(P_\bullet^N, P_\bullet^M[i])=0$ for all $i > 0$.
\end{enumerate}
\end{lemma}

The lemma above leads to the following connection between $\tau_d$-rigidity and inner acyclic $(d+1)$-term presilting complexes.

\begin{proposition}\label{prop:presilting}
Let $\mc$ be a $d$-cluster tilting subcategory of $\modA$ and consider a \mbox{$\tau_d$-rigid} pair $(M,P)$ in $\mc$. Then $P_\bullet^{(M,P)}$ is an inner acyclic $(d+1)$-term presilting complex.
\end{proposition}

\begin{proof}
Observe that $P_\bullet^{(M,P)}$ is an inner acyclic $(d+1)$-term complex by construction. The statement then follows from the isomorphism $\Hom_{K^b(\proj A)}(P,P_\bullet^M)\cong \Hom_{A}(P,M)$ and \cref{lem:tau-rigidKbProj}. 
\end{proof}

\begin{remark}
    In the case $d=1$, it is shown in \cite{AdachiIyamaReiten} that sending $(M,P)$ to $P_\bullet^{(M,P)}$ gives a bijection between maximal $\tau$-rigid pairs and $2$-term silting complexes. When $d>1$, we will prove in \cref{sec:silting} that if the maximal $\tau_d$-rigid pair $(M,P)$ comes from a $d$-torsion class, then $P_\bullet^{(M,P)}$ is a $(d+1)$-term silting complex.
\end{remark}

\section{Characterisation results}\label{sec: char. results}

 Throughout this section, we fix a $d$-cluster tilting subcategory $\mc$ of $\modA$. 

\subsection{Inducing $d$-torsion classes}
In this subsection we describe when classical torsion classes give rise to higher torsion classes, and observe that any split torsion class satisfies this property. In fact, they give rise to split \mbox{$d$-torsion} classes; see \cref{def:split d-torsion}. 

\begin{definition}\label{def: induces}
    Let $\tc$ be a torsion class in $\modA$. We say that $\tc$ \textit{induces} a $d$-torsion class in $\mc$ if the following hold:
    \theoremlistfix
    \begin{enumerate}
        \item The subcategory $\uc\colonequals\tc\cap \mc$ is a $d$-torsion class in $\mc$.
        \item The torsion and $d$-torsion subobject of an object in $\mc$ coincide, i.e.\ we have \mbox{$tM = U_M$} for any $M\in\mc$.
    \end{enumerate}
\end{definition}

If a torsion class $\tc$ in $\mod A$ satisfies the conditions in \cref{def: induces}, we say that $\tc$ induces the $d$-torsion class $\uc = \tc \cap \mc$. The following result shows that any $d$-torsion class is induced by a torsion class.

\begin{theorem}\label{thm:torsion-to-dtorsion}
Let $\uc$ be a $d$-torsion class in $\mc$. Then $\uc$ is induced by the minimal torsion class in $\modA$ containing it.
\end{theorem}

\begin{proof}
    This follows from \cite[Theorem 1.1]{AJST}.
\end{proof}

Note that there can be several torsion classes inducing the same $d$-torsion class; see \cite[Example 2.16]{AHJKPT1}. Furthermore, not every torsion class induces a $d$-torsion class; see \cite[Example 3.7]{AJST} and \cite[Example 2.16]{AHJKPT1}. The result below gives necessary and sufficient criteria for when this holds. Given a torsion class $\tc$ in $\mod A$, we use the notation $t\mc = \{tM \mid M \in \mc\}$ and $f\mc = \{fM \mid M \in \mc\}$.

\begin{theorem}\label{Proposition:CharacterizationInducingdTorsion}
    Let $\tc$ be a torsion class in $\modA$. Then $\tc$ induces a $d$-torsion class if and only if the following hold:
    \theoremlistfix
  \begin{enumerate}
      \item The equality $\Ext^{d-1}_A(t\mc,f\mc)=0$ holds.
      \phantomsection
\label{Proposition:CharacterizationInducingdTorsion:1}
    \item We have $t\mc\subseteq \mc$.
    \phantomsection
\label{Proposition:CharacterizationInducingdTorsion:2}
  \end{enumerate}  
\end{theorem}

\begin{proof}
We need the following observation: Let $M\in \mc$ and consider an exact sequence
  \begin{equation}\label{MCoresOffM}
   0\to fM\to V_1\to \cdots \to V_d\to 0   
  \end{equation}
  where $V_i\in \mc$ for $i=1,\dots,d$. Since $\mc$ is rigid in degrees $1,\dots ,d-1$, a dimension shifting argument implies that for $N\in \mc$, the  homology of the complex
  \begin{equation}\label{HomologyExtComplex}
      \cdots\to 0 \xrightarrow{} \Hom_A(N,V_1) \xrightarrow{  } \cdots \xrightarrow{  } \Hom_A(N,V_d) \xrightarrow{} 0\to \cdots
  \end{equation} 
  at $\Hom_A(N,V_{i})$ is isomorphic to the abelian group $\Ext^{i-1}_A(N,fM)$ for $i>1$, and to $\Hom_A(N,fM)$ for $i=1$. In particular, \eqref{HomologyExtComplex} is exact for all $N\in \tc\cap \mc$ if and only if $\Ext^i_A(\tc\cap\mc,fM)=0$ for $i=1,\dots,d-1$. 
  
 Now assume that $\tc$ induces a $d$-torsion class.  Since the torsion and $d$-torsion subobject of an object in $\mc$ coincide, we must have $t\mc\subseteq \mc$. This implies that $\tc\cap \mc=t\mc$. For any $M\in \mc$, we furthermore have an exact sequence 
 \begin{equation*}
  0 \rightarrow U_M \rightarrow M \rightarrow V_1 \rightarrow \cdots \rightarrow V_d \rightarrow 0
\end{equation*}
satisfying the conditions in \cref{def:ntorsionclass}. This gives an exact sequence as in \eqref{MCoresOffM}. 
It now follows from \cref{def:ntorsionclass} and the observation above that $\Ext^{d-1}_A(t\mc,f\mc)=0$.

Assume next that $\tc$ satisfies conditions \eqref{Proposition:CharacterizationInducingdTorsion:1} and \eqref{Proposition:CharacterizationInducingdTorsion:2} in the theorem, and set $\uc=\tc\cap \mc$. Note that $\uc=t\mc$ since $t\mc\subseteq \mc$. For an arbitrary $M\in \mc$, choose an $\mc$-coresolution 
\[
0\to fM\to V_1\to \cdots \to V_d\to 0
\]
of length $d$. We claim that the induced exact sequence
\[
0\to tM\to M\to V_1\to \cdots \to V_d\to 0
\]
satisfies the conditions \eqref{def:dtors:approx} and \eqref{def:dtors:sequence} in \cref{def:ntorsionclass}. Indeed, we have $tM\in t\mc=\uc$, which shows \eqref{def:dtors:approx}. This also implies that the sequence
\[
0\to tM\to M\to fM\to 0
\]
gives an $\mc$-resolution of $fM$. Applying $\Hom_A(M',-)$ with $M'\in \mc$ and considering the associated long exact sequence, we get that $\Ext^i_A(M',fM)=0$ for $1 \leq i<d-1$. Combining this with condition \eqref{Proposition:CharacterizationInducingdTorsion:1} from the theorem yields $\Ext^{i}_A(t\mc,f\mc)=0$ for all $1 \leq i \leq d-1$. For $N\in \uc$, the complex \eqref{HomologyExtComplex} is thus acyclic, which proves the claim.
\end{proof}

We next consider split $d$-torsion classes, which is a higher analogue of split torsion classes. This notion was first introduced under the name \textit{splitting torsion classes} in \cite{Jorgensen}. Split \mbox{$d$-torsion} classes play a key role in \cref{sec: tau-tilting to tau_d-tilting}.

\begin{definition}\label{def:split d-torsion}
A subcategory $\uc$ of $\mc$ is called a \textit{split $d$-torsion class} if for every \mbox{$M$ in $\mc$}, there exists a split exact sequence
\begin{equation*}
  0 \rightarrow U_M \rightarrow M \to V_M\to 0
\end{equation*} 
such that the following hold:
\theoremlistfix
\begin{enumerate}
    \item The object $U_M$ is in $\uc$.
    \item We have $\Hom_A(U,V_M)=0$ for every $U$ in $\uc$. 
\end{enumerate}
\end{definition}

\begin{lemma}
If a subcategory $\uc \subseteq \mc$ is a split $d$-torsion class, then $\uc$ is a $d$-torsion class.
\end{lemma}

\begin{proof}
   Note that $V_M$ is a direct summand of $M\in \mc$ and hence is contained in $\mc$ itself. Setting $V_1=V_M$ and $V_2=\cdots =V_d=0$, we get that $\uc$ satisfies \cref{def:ntorsionclass}. 
\end{proof}

\begin{remark}
In contrast to general $d$-torsion classes, a split $d$-torsion class $\uc$ has an associated \textit{split $d$-torsion free class} $\vc\subseteq \mc$. It is given by $$\vc\colonequals \{V_M\mid M\in \mc \} = \{M\in \mc\mid \Hom_A(U,M)=0 \text{ for all }U\in \uc\}.$$
\end{remark}

We next observe that a split torsion class always induces a split $d$-torsion class. We note that the argument in the proof below already appears in the proof of \cite[Theorem 9.2]{Jorgensen}.

\begin{proposition}\label{Prop:SplitTorsionGivesdTorsion}
    Let $\tc$ be a split torsion class in $\modA$. Then $\tc$ induces a split $d$-torsion class in $\mc$.
\end{proposition}

\begin{proof}
    Set $U_M=tM$ and $V_M=fM$ for each $M\in \mc$. Since $\tc$ is a split torsion class, we have a split exact sequence
    \[
    0\to U_M\to M\to V_M\to 0
    \]
    for each $M\in \mc$. As $U_M$ is a direct summand of $M$, it is contained in $\mc$, and hence in $\uc$. Moreover, since $V_M$ is in the torsion free part, we have $\Hom_A(U,V_M)=0$ for every $U\in \uc$. This proves the claim.
\end{proof}

\subsection{Characterisation by right $d$-exact sequences}

Recall that a right $d$-exact sequence in $\mc$ is an exact sequence 
\begin{equation*}
X_0\xrightarrow{}X_1 \xrightarrow{} X_2 \xrightarrow{} \cdots \xrightarrow{} X_{d} \xrightarrow{} X_{d+1} \to 0
\end{equation*}
with $X_i \in \mc$ for $i=0,\dots,d+1$. In this section we give a new characterisation of $d$-torsion classes in terms of closure under right $d$-exact sequences. The proof uses the characterisation of $d$-torsion classes given in \cref{theorem: extension closed}. The result itself will be used in \cref{sec:silting}.

\begin{theorem}\label{Reformulation:d-Tors}
A subcategory $\uc$ of $\mc$ is a $d$-torsion class if and only if for any right $d$-exact sequence
    \begin{equation}\label{eq: right d-exact}
    X\xrightarrow{} Y_1\to Y_2\to \cdots \to Y_d \to Z\to 0
    \end{equation}
    where $Y_i\to Y_{i+1}$ are radical morphisms for $i = 1,\dots,d-1$ and $X,Z\in \uc$, we have $Y_1,\dots, Y_d\in \uc$.
\end{theorem}

\begin{proof}
Note that a minimal $d$-extension with end terms in $\uc$ is a right $d$-exact sequence as in \eqref{eq: right d-exact} where $X\to Y_1$ is additionally a monomorphism. Moreover, a minimal $d$-quotient of an object $X \in \uc$ is a right $d$-exact sequence of the form \eqref{eq: right d-exact} with $Z=0$. Hence, if $\uc$ satisfies the closure condition in the proposition, then it is closed under $d$-extensions and $d$-quotients, and is therefore a $d$-torsion class by \cref{theorem: extension closed}.

Assume next that $\uc$ is a $d$-torsion class, and consider a right $d$-exact sequence \eqref{eq: right d-exact} as in the statement. Since $X\in \uc$, the morphism $X\to Y_1$ factors through the inclusion $U_{Y_1}\to Y_1$, where $U_{Y_1} \in \uc$ is the $d$-torsion subobject of $Y_1$ with respect to $\uc$. Let 
\[
U_{Y_1}\to V_2\to \cdots \to V_{d+1}\to 0
\]
be a minimal $d$-cokernel of $X\to U_{Y_1}$, and note that $V_2,\dots,V_{d+1}\in \uc$ since $\uc$ is closed under $d$-quotients. Consider the commutative diagram
\begin{equation}\label{Eq:DiagforBij}
	\begin{tikzcd}
	 X \arrow[r] \arrow[d, equal] & U_{Y_1} \arrow[r] \arrow[d] & V_2 \arrow[r] \arrow[d, dashed]  & \cdots \arrow[r] & V_d \arrow[r] \arrow[d, dashed] & V_{d+1} \arrow[d, dashed]  \arrow[r]  & 0\\
	 X\arrow[r] & Y_1 \arrow[r] & Y_2 \arrow[r] & \cdots \arrow[r] & Y_{d} \arrow[r] & Z\arrow[r]  & 0
	\end{tikzcd}
\end{equation}
where the dashed morphisms are obtained by using the commutativity of the leftmost square. Taking the cone of the morphism of complexes represented by the vertical maps, gives the complex
\[
X\to X\oplus U_{Y_1}\to Y_1\oplus V_2\to Y_2\oplus V_3\to \cdots \to Y_{d}\oplus V_{d+1}\to Z\to 0,
\]
which is again homotopy equivalent to 
\begin{equation}\label{Eq:d-ExactforBij}
  0\to U_{Y_1}\to Y_1\oplus V_2\to Y_2\oplus V_3\to \cdots \to Y_{d}\oplus V_{d+1}\to Z\to 0.  
\end{equation}
Since the rows of the diagram \eqref{Eq:DiagforBij} are exact everywhere except at $X$, the cone is exact at $Y_i\oplus V_{i+1}$ for $i=1,\dots,d$ and at $Z$. Combining this with the fact that $U_{Y_1}\to Y_1$ is a monomorphism, it follows that the complex \eqref{Eq:d-ExactforBij} is exact. Hence, it is a $d$-exact sequence, and we can write it as a finite direct sum of a minimal $d$-extension
\begin{equation*}
  0\to U_{Y_1}\to U_1\to U_2\to \cdots \to U_d\to Z\to 0 
\end{equation*}
and complexes of the form $N \xrightarrow[]{1_N} N$. Since the end terms $U_{Y_1}$ and $Z$ in the sequence above are in $\uc$ and $\uc$ is closed under $d$-extensions, it follows that $U_i\in \uc$ for $i=1,\dots,d$. We use this to show that each $Y_i$ is in $\uc$. 

Indeed, let $Y_i'$ be an indecomposable direct summand of $Y_i$. If $Y_i'$ is a direct summand of $U_i$, then $Y_i'$ must be in $\uc$. Otherwise, $Y'_i$ is a direct summand of an object $N$ where $N \xrightarrow[]{1_N} N$ is a summand of \eqref{Eq:d-ExactforBij}. Hence, either $Y_i'\to Y_i\to Y_i\oplus V_{i+1}\to Y_{i+1}\oplus V_{i+2}$ is a split monomorphism or $Y_{i-1}\oplus V_i\to Y_i\oplus V_{i+1}\to Y_i\to Y_i'$ is a split epimorphism. Since $Y_{i-1}\to Y_i$ and $Y_i\to Y_{i+1}$ are radical morphisms by assumption, it follows that $Y_i'\to Y_i\to Y_i\oplus V_{i+1} \to V_{i+2}$ is a split monomorphism or $V_i\to Y_i\oplus V_{i+1}\to Y_i\to Y_i'$ is a split epimorphism. This yields that $Y_i'$ is a direct summand of $V_{i}$ or $V_{i+2}$. In either case, we get that $Y_i'$ is in $\uc$. Since $Y_i'$ was an arbitrary indecomposable direct summand of $Y_i$, it follows that $Y_i\in \uc$ for $i=1,\dots,d$.
\end{proof}

\begin{remark}
In the classical case,  \Cref{Reformulation:d-Tors} is just the well-known statement that a subcategory $\tc$ is a torsion class if and only if for any right exact sequence $X\to Y\to Z\to 0$ with $X,Z\in \tc$, we have $Y\in \tc$.
\end{remark}

\section{Approximations and \texorpdfstring{$\Ext^d$}{$\Ext^d$}-projectives}\label{sec: ff-d-torsion}

Functorially finite torsion classes play a key role in classical $\tau$-tilting theory, and a reason for this is that they can be used to take left approximations and to form coresolutions. In particular, if $\tc$ is a functorially finite torsion class in $\modA$, then for any $A$-module $M$, there is an exact sequence
\begin{align*}
    M \to T_0 \to T_1 \to 0
\end{align*}
with $T_0, T_1 \in \tc$. Such a sequence is used in the construction of the \textit{$\Ext$-projectives} in $\tc$, i.e.\ modules $T\in \tc$ satisfying $\Ext^1_{A}(T,N)=0$ for all $N\in \tc$. These in turn play a key role in the bijection between functorially finite torsion classes and support $\tau$-tilting pairs. 

The goal of this section is to construct higher analogues of $\Ext$-projective objects in a functorially finite $d$-torsion class $\uc$, using $\uc$-approximations and $\uc$-coresolutions. We work under the following setup throughout this section.

\begin{setup}\label{torsion setup}
Let $\mc \subseteq \mod A$ be a $d$-cluster tilting subcategory. Suppose that $\uc$ is a functorially finite $d$-torsion class in $\mc$, and let $\tc$ be a torsion class in $\modA$ which \mbox{induces $\uc$}.
\end{setup}

It follows from \cref{thm:torsion-to-dtorsion} that such a torsion class $\tc$ exists for any $d$-torsion \mbox{class $\uc$.} We do not know whether $\uc$ being functorially finite implies that $\tc$ is functorially finite in $\mod A$. However, if a functorially finite torsion class in $\mod A$ induces a $d$-torsion class \mbox{in $\mc$,} then this $d$-torsion class is functorially finite; see \cref{prop:T ff implies U ff}. Note that as $\mc$ is functorially finite in $\mod A$, we have that $\uc$ is functorially finite in $\mc$ if and only if it is functorially finite in $\mod A$.

\begin{lemma}\cite[Corollary 3.4]{AHJKPT1}\label{cor: minimal M-appr. in U}
If $X \in\tc$, then the minimal left \mbox{$\mc$-approximation} of $X$ is also the minimal left $\uc$-approximation of $X$. Moreover, this approximation is a monomorphism.
\end{lemma}

It is well-known (see \cite{Smalo}) that any functorially finite torsion class in $\modA$ can be written as $\fac T$ for some $T \in \modA$. The following result generalises this by using left $\uc$-approximations to produce a generator of the functorially finite $d$-torsion class $\uc$.

\begin{lemma}\label{lem: generation}
Let $M \in \modA$ be such that $\uc \subseteq \fac M$, and suppose that $M \to U_0^M$ is a left $\uc$-approximation of $M$. Then $\fac U_0^M \cap \mc =\uc.$ 
\end{lemma} 

\begin{proof}
Given any $U\in\uc$, there is an epimorphism $M^t \to U$ for some $t \in \mathbb{N}$ as \mbox{$\uc \subseteq \fac M$}. This epimorphism must factor through the left $\uc$-approximation $M^t\to (U_0^M)^t$ with the corresponding morphism $(U_0^M)^t\to U$ necessarily being an epimorphism. This shows that $\uc \subseteq \fac U_0^M \cap\mc$. Since a torsion class is closed under quotients and $U_0^M \in \tc$, it follows that $\fac U_0^M \subseteq \tc$, and hence $\fac U_0^M\cap\mc\subseteq \tc \cap \mc =\uc$. We can thus conclude that $\fac U_0^M \cap \mc =\uc.$
\end{proof}

Note that since $\fac A = \modA$, we may always choose $M=A$ in \cref{lem: generation}, so any functorially finite $d$-torsion class in $\mc$ may be written as $\fac U \cap \mc$ for some $U \in \mc$.

Next we introduce higher analogues of $\Ext$-projective objects. 

\begin{definition}\label{defn:ExtdProj}
An object $X \in \uc$ is called \textit{$\Ext^d$-projective} in $\uc$ if $\Ext^d_{A}(X, U)=0$ for every $U \in \uc$. It is furthermore called an \textit{$\Ext^d$-projective generator} of $\uc$ if any other $\Ext^d$-projective in $\uc$ is contained in $\add(X)$.
\end{definition}

We now show that left $\uc$-approximations can be used to construct $\Ext^d$-projective objects \mbox{in $\uc$}. 

\begin{lemma}\label{lem:casei=0}
Assume $M \in \modA$ satisfies $\Ext^1_{A}(M,T)=0$ for every $T \in \tc$, and let $f_0 \colon M \to U_0^M$ be the minimal left $\uc$-approximation of $M$. Then $U_0^M$ is $\Ext^d$-projective \mbox{in $\uc$}.
\end{lemma}

\begin{proof}
Let $U\in \uc$, and consider a $d$-extension
\[
0 \to U \to E_1 \to \cdots \to E_{d-1} \to E_d \xrightarrow{g} U_0^M \to 0.
\]
We can assume this extension to be minimal, allowing us to conclude that $E_i \in \uc$ for $i=1,\dots,d$ by \cref{theorem: extension closed}. It follows that the image $K$ of the morphism $E_{d-1}\to E_d$ is in $\tc$. Now consider the short exact sequence
\[
0\to K\to E_d\xrightarrow{g} U_0^M \to 0.
\]
Since $\Ext^1_{A}(M,\tc)=0$, there exists a morphism $h\colon M\to E_d$ satisfying $gh=f_0$. Moreover, as $f_0$ is a left $\uc$-approximation, there exists a morphism $k\colon U_0^M\to E_d$ such that $k f_0=h$. It follows that $f_0=g k f_0$, and since $f_0$ is left minimal, the composition $g k$ is an isomorphism. Hence, the short exact sequence splits, and so does the $d$-extension. This shows that $U_0^M$ is \mbox{$\Ext^d$-projective in $\uc$.}
\end{proof}

We now turn our attention to full $\uc$-coresolutions. A \textit{$\uc$-coresolution} of $M \in \modA$ is given by an exact sequence
\begin{equation}\label{eq:U-cores modules}
 M\xrightarrow{f_0} U^M_0\xrightarrow{f_1} U^M_1\xrightarrow{f_2} \cdots \xrightarrow{f_d} U^M_d\to 0
\end{equation}
with $U^M_0,\dots,U^M_d \in \uc$, and where the morphism $f_0 \colon M\xrightarrow{}U^M_0$ is a left $\uc$-approximation. A $\uc$-coresolution \eqref{eq:U-cores modules} is called \textit{minimal} if $f_0,\dots,f_{d-1}$ are left minimal.

A minimal $\uc$-coresolution of \mbox{$M \in \modA$} can be constructed
by first taking the minimal left \mbox{$\uc$-approximation} $f_0 \colon M \to U_0^M$ of $M$, and then iteratively taking the minimal left $\uc$-approximation of the cokernel of the previous morphism. This is illustrated in the following diagram
	\[
	\begin{tikzcd}[column sep=0.4cm, row sep=0.5cm]
	 M \arrow{rr}{f_0}  && U_0^M \arrow{rr}{f_1} \arrow[two heads]{rd} && U_1^M \arrow{rr}{f_2} \arrow[two heads]{rd} && U_2^M \arrow{rr}{f_3} \arrow[two heads]{rd} && U_3^M \arrow{rr}{f_4} \arrow[two heads]{rd} && \cdots,  \\
	&&& C_0 \arrow[hook]{ru}{g_1} &&  C_1 \arrow[hook]{ru}{g_2} && C_2 \arrow[hook]{ru}{g_3} && C_4\arrow[hook, "g_4"]{ru} & 
	\end{tikzcd}
	\]
where $C_i$ is the cokernel of $f_i$ and $g_i$ is the minimal left $\uc$-approximation of $C_{i-1}$. Since $\uc=\tc \cap \mc$ and $\tc$ is closed under quotients, each cokernel $C_i$ is contained \mbox{in $\tc$}. By \cref{cor: minimal M-appr. in U}, each \mbox{$\uc$-approximation} $g_i$ is thus a monomorphism and is equivalently obtained as a minimal left \mbox{$\mc$-approximation.} Constructing a minimal \mbox{$\uc$-coresolution} of $M$ is hence the same as first considering the minimal left $\uc$-approximation of $M$, and then taking the minimal \mbox{$\mc$-coresolution} of its cokernel. Note that since minimal $\mc$-coresolutions are known to have at most $d$ terms, the sequence constructed above has at most $d+1$ terms.

Building on \cref{lem:casei=0}, we now use $\uc$-coresolutions to construct \mbox{$\Ext^d$-projective} objects \mbox{in $\uc$}. Later in \cref{sec:d-tors to tau_d,sec: tau-tilting to tau_d-tilting} we consider two different choices of $M$ leading to two different constructions of an $\Ext^d$-projective generator of $\uc$. 

\begin{theorem}\label{thm:U_i Ext^d-projective}
Suppose $M \in \modA$ satisfies $\Ext^1_{A}(M,T)=0$ for every $T \in \tc$, and let
$$M \xrightarrow{f_0} U_0^M \xrightarrow{f_1} U_1^M \xrightarrow{f_2} \cdots \xrightarrow{f_d} U_d^M \to 0$$ 
be a minimal $\uc$-coresolution of $M$. Then $U_i^M$ is $\Ext^d$-projective in $\uc$ for $i=0,\dots,d$.
\end{theorem}

\begin{proof}
We prove the claim by induction on $i$, where the case $i=0$ is shown in \cref{lem:casei=0}.

Now suppose that $1 \leq i \leq d$ and that $U_0^M, \dots, U_{i-1}^M$ are $\Ext^d$-projective in $\uc$. Consider a $d$-exact sequence
\begin{align}\label{eq: dexact seq}
    0 \to V_0 \xrightarrow{v_0} V_1 \xrightarrow{v_1} \cdots \xrightarrow{v_{d-1}} V_d \xrightarrow{v_d} U_i^M \to 0 
\end{align}
for an arbitrary $V_0\in \uc$. By \cref{theorem: extension closed}, we can assume that $V_j \in \uc$ for all $1 \leq j \leq d$. 

We proceed by constructing dashed morphisms making the diagram 
\begin{center}
\begin{tikzcd}
	M \arrow[r,"f_0"]\arrow[d,dashed,"k_{i+1}"] & U_0^M \arrow[r,"f_1"] \arrow[d,dashed,"k_{i}"] & \cdots  \arrow[r,"f_{i-2}"] & U_{i-2}^M \arrow[r,"f_{i-1}"] \arrow[d,dashed,"k_{2}"] & U_{i-1}^M \arrow[r,,"f_i"] \arrow[d,dashed,"k_{1}"] & U_i^M \arrow[d, equal] & \\
	V_{d-i} \arrow[r,"v_{d-i}"] & V_{d-i+1} \arrow[r,"v_{d-i+1}"] & \cdots \arrow[r,"v_{d-2}"] & V_{d-1} \arrow[r,"v_{d-1}"] & V_d \arrow[r,"v_{d}"] & U_i^M \arrow[r] & 0
\end{tikzcd}
\end{center}
commute. The morphism $k_1 \colon U_{i-1}^M \to V_d$ exists since $\Hom_A(U_{i-1}^M,V_d) \to \Hom_A(U_{i-1}^M,U_i^M)$ is surjective as $U_{i-1}^M$ is $\Ext^d$-projective; see \cite[Proposition 2.2]{higher survey}. For $2\leq j \leq i$, we use that \eqref{eq: dexact seq} is $d$-exact and $U_{i-j}^M \in \uc \subseteq \mc$ to obtain the morphisms $k_j \colon U_{i-j}^M \to V_{d-j+1}$. We complete the diagram with a morphism $k_{i+1} \colon M \to V_{d-i}$, using that $\Ker(v_{d-i})\in \tc$ and $\Ext^1_A(M,T)=0$ for all $T$ in $\tc$. 

As $f_0$ is a left $\uc$-approximation of $M$ and $V_{d-i} \in \uc$, there is a morphism \mbox{$h_0 \colon U_0^M \to V_{d-i}$} such that $h_0 f_0 =k_{i+1}$. Since the sequence 
\begin{align*}
    M \xrightarrow{f_0} U_0^M \xrightarrow{g_0} K_0 \to 0 
\end{align*}
with $K_0 = \Coker(f_0)$ is exact, there is an exact sequence
\begin{align*}
    0 \to \Hom_A(K_0, V_{d-i+1}) \xrightarrow{- \circ g_0} \Hom_A(U_0^M,V_{d-i+1}) \xrightarrow{- \circ f_0} \Hom_A(M, V_{d-i+1}).
\end{align*}
By commutativity of the above diagram, we have $k_{i}-v_{d-i} h_0 \in \Ker (- \circ f_0)$,  
and thus there is a morphism $k'_{i} \colon K_0 \to V_{d-i+1}$ such that $k'_{i} g_0 = k_{i}-v_{d-i} h_0$. 

By construction, the morphism $f_1$ factors as $U_0^M \xrightarrow{g_0} K_0 \xrightarrow{f_1'} U_1^M$, where $f_1'$ is a minimal left \mbox{$\uc$-approximation}. Consequently, using that $V_{d-i+1} \in \uc$, there exists a morphism \mbox{$h_1 \colon U_1^M \to V_{d-i+1}$} such that $h_1 f_1' = k_{i}'$. In particular, this yields the equalities \(h_1  f_1 = h_1  f_1'  g_0 = k_{i}'  g_0 = k_{i}-v_{d-i}  h_0 \),
so $k_i=h_1f_1+v_{d-i}h_0$. Now since $k_{i-1}- v_{d-i+1} h_1 \in \Ker(- \circ f_1)$, we may repeat the argument, producing morphisms $h_j$ as indicated in the diagram
\begin{center}
\begin{tikzcd}
	M \arrow[r,"f_0"] \arrow[d,"k_{i+1}", swap] & U_0^M \arrow[r,"f_1"] \arrow[d,"k_{i}", swap] \arrow[ld,"h_0", swap] & U_1^M \arrow[r,"f_2"] \arrow[d,"k_{i-1}", swap] \arrow[ld,"h_1", swap]& \cdots  \arrow[r,"f_{i-2}"] & U_{i-2}^M \arrow[r,"f_{i-1}"] \arrow[d,"k_{2}", swap] & U_{i-1}^M \arrow[r,,"f_i"] \arrow[d,"k_{1}", swap] \arrow[ld,"h_{i-1}", swap] & U_i^M \arrow[d, equal] \arrow[ld,"h_i", swap] \\
	V_{d-i} \arrow[r,"v_{d-i}"] & V_{d-i+1} \arrow[r,"v_{d-i+1}"] & V_{d-i+2} \arrow[r,"v_{d-i+2}"] & \cdots \arrow[r,"v_{d-2}"] & V_{d-1} \arrow[r,"v_{d-1}"] & V_d \arrow[r,"v_{d}"] & U_i^M
\end{tikzcd}
\end{center}
satisfying the homotopy relations. In particular, the final homotopy relation is
\begin{align*}
   k_1 = h_i  f_i +v_{d-1} h_{i-1}.
\end{align*}
It follows that we have \( f_i = v_d  k_1 = v_d  h_i  f_i + v_d  v_{d-1}  h_{i-1} = v_d h_i f_i\). Now, recall that $f_i$ factors as $U_{i-1}^M \xrightarrow{g_{i-1}} K_{i-1} \xrightarrow{f_i'} U_i^M$, where $f_i'$ is the minimal left \mbox{$\uc$-approximation} of $K_{i-1} \colonequals \Coker(f_{i-1})$. Since $g_{i-1}$ is an epimorphism, it follows that $f_i'= v_d  h_i  f_i'$. As $f_i'$ is left minimal, this implies that $v_d  h_i$ is an isomorphism. Consequently, the sequence \eqref{eq: dexact seq} splits, proving that $\Ext^d(U_i^M,V_0)=0$ for all $V_0 \in \uc$ as required. 
\end{proof}

\section{From \texorpdfstring{$d$}{d}-torsion classes to \texorpdfstring{$\tau_d$}{tau-d}-tilting theory}\label{sec:d-tors to tau_d}

Throughout \cref{sec:d-tors to tau_d}, we continue to work under \cref{torsion setup}. The main goal is to \mbox{construct a map} 

\begin{align*}
\renewcommand{\arraystretch}{1.8}
\begin{array}{ccc}
\renewcommand{\arraystretch}{1.2}
\begin{Bmatrix}
    \text{functorially finite} \\
    \text{$d$-torsion classes in $\mc$}
\end{Bmatrix}
&
    \to 
    &
    \renewcommand{\arraystretch}{1.2}
\begin{Bmatrix}
\text{basic $\tau_d$-rigid pairs in $\mc$} \\
\text{with $|A|$ summands }
\end{Bmatrix}\\ 
\uc &\mapsto  &(M_\uc, P_\uc)
\end{array} 
\renewcommand{\arraystretch}{1}
\end{align*}
which generalises the known bijection from \cite{AdachiIyamaReiten} in the case $d=1$. This is achieved in \cref{thm:maximaltaunrigid}. As a key step towards this result, we prove \cref{thm:intro d-tilting} from the introduction; see \cref{thm: d-tilting}. We start by providing a link between the properties of $\Ext^d$-projectivity and $\tau_d$-rigidity.

\begin{theorem}\label{thm: equivalence of Extd projectivity}
Let $U \in \uc$. The following statements are equivalent:
\theoremlistfix
\begin{enumerate}
    \item $U$ is $\Ext^d$-projective in $\uc$;
    \phantomsection
    \label{thm:equivalence:projective}
    \item $\tau_d U\in \fc$, where $\fc$ is the torsion free class corresponding to $\tc$;
    \phantomsection
    \label{thm:equivalence:torsion-free}
    \item $\Hom_{A}(U',\tau_d U) = 0$ for all $U'$ in $\uc$. 
    \phantomsection
    \label{thm:equivalence:tau-rigid}
\end{enumerate}
In particular, if $U$ is $\Ext^d$-projective in $\uc$, then $U$ is $\tau_d$-rigid.
\end{theorem}

\begin{proof}
  (\ref{thm:equivalence:projective}) \textit{implies} (\ref{thm:equivalence:torsion-free}):
  Let $U$ be $\Ext^d$-projective in $\uc$. It suffices to show the statement under the assumption that $U$ is indecomposable. If $\tau_d U = \tau\Omega^{d-1} U$ is zero, the result is immediate. Assume hence that $\Omega^{d-1}U$ is non-projective. Since $\tau_d U$ is indecomposable, so is $\Omega^{d-1}U\cong\tau^-\tau_d U$. Consider the almost split sequence
  \begin{equation}\label{AR-seq}
  0 \to \tau_d U\overset{i}{\to} E \overset{p}{\to} \Omega^{d-1} U \to 0
  \end{equation}
  ending in $\Omega^{d-1} U$. 
  We show the result by contradiction, so assume that $\tau_d U$ is not an object in $\fc$. Consider the canonical short exact sequence 
  $$0 \to t\tau_d U\xrightarrow{\iota} \tau_d U\xrightarrow{\pi} f\tau_d U \to 0$$
  of $\tau_d U$ with respect to the torsion pair $(\tc, \fc)$. Thus, we have $t\tau_d U\in\tc\cap\mc = \uc$. Moreover, our assumption implies that $\pi \colon \tau_d U \to f\tau_d U$ is not a split monomorphism. Since $i$ is left almost split, there is a morphism $g \colon E\xrightarrow{} f\tau_d U$ such that $gi=\pi$. Because $\pi$ is an epimorphism, so is $g$, and we obtain a commutative diagram
  \[
  	\begin{tikzcd}[column sep=20, row sep=20]
  %	  & 0\arrow[d]  & 0\arrow[d] & \\
	 &t\tau_d U \arrow[r] \arrow[d, "\iota"] & E' \arrow[r] \arrow[d] & \Omega^{d-1}U \arrow[d, equal]  \\
	 &\tau_d U \arrow[r, "i"] \arrow[d, "\pi"] & E \arrow[r, "p"] \arrow[d, "g"] & \Omega^{d-1}U  \\
%	 & f\tau_d U \arrow[r, equal] \arrow[d] & f\tau_d U\arrow[d] & \\ 
& f\tau_d U \arrow[r, equal]  & f\tau_d U 
%	  & 0  & 0 &
	\end{tikzcd}
  \]
  whose rows and columns are short exact sequences and where $E' = \Ker(g)$. By assumption, the object $U$ is $\Ext^d$-projective in $\uc$, which implies that $\Ext^1_{A}(\Omega^{d-1}U, t\tau_d U)=0$ as $t\tau_d U \in \uc$. This implies that the top exact sequence of the diagram splits. From this, we can conclude that also \eqref{AR-seq} splits, which is a contradiction.
  
(\ref{thm:equivalence:torsion-free}) \textit{implies} (\ref{thm:equivalence:tau-rigid}):
  If $\tau_d U$ is in $\fc$, we have $\Hom_{A}(X, \tau_d U)=0$ for all $X$ in $\tc$. Then the result follows from the fact that $\uc = \tc \cap \mc \subseteq \tc$.
  
  (\ref{thm:equivalence:tau-rigid}) \textit{implies} (\ref{thm:equivalence:projective}):
  If $\Hom_{A}(U', \tau_d U)=0$ for all $U'$ in $\uc$, then $\Ext^d_{A}(U, U')=0$ for all $U'$ in $\uc$ by higher Auslander--Reiten duality.
\end{proof}

Combining \cref{thm:U_i Ext^d-projective} with \cref{thm: equivalence of Extd projectivity} gives the following corollary.

\begin{corollary}\label{cor: extd vs taud} 
Suppose $M \in \modA$ satisfies $\Ext^1_{A}(M,T)=0$ for every $T \in \tc$. If 
$$M \to U_0^M \to U_1^M \to \cdots \to U_d^M \to 0$$ 
is a minimal $\uc$-coresolution of $M$, then $\Hom_A(U_i^M, \tau_d U_j^M)=0$ for all $0 \leq i,j \leq d$.
\end{corollary}

Putting $M=A$ in \cref{cor: extd vs taud}, the minimal $\uc$-coresolution
    \begin{align}\label{eq: cores of A}
    A \xrightarrow{f} U_0^A \to U_1^A \to \cdots \to U_{d-1}^A \to U_d^A \to 0 
    \end{align}
allows us to associate an $\Ext^d$-projective module $U^A \colonequals \bigoplus_{i=0}^d U_i^A$ to $\uc$. Our next step is to show that $U^A$ is an $\Ext^d$-projective generator of $\uc$ in the sense of \cref{defn:ExtdProj}, i.e.\ that any \mbox{$\Ext^d$-projective} in $\uc$ is contained in $\add(U^A)$.

To this end, we consider the algebra 
\[
B \colonequals A/\Ann \uc,
\]
where $\Ann \uc=\{a\in A\mid U\cdot a =(0) \text{ for all }U\in \uc\}$. The category $\modB$ can be identified with the subcategory of $\modA$ consisting of all modules $M$ satisfying $M\cdot a=(0)$ for all $a\in \Ann \uc$. Note that $\uc\subseteq \modB$ and $\Ker f= \Ann \uc$ since $f$ is a left $\uc$-approximation. It follows that the sequence
\begin{align}\label{eq: cores of B}
    0\to B \to U_0^A \to U_1^A \to \cdots \to U_{d-1}^A \to U_d^A \to 0
\end{align}
is exact in $\modB$.

Recall that the $\kk$-dual $DA$ of $A$ is an injective object in $\modA$. Moreover, every finitely generated injective $A$-module is in $\add (DA)$. It follows from the definition of a $d$-cluster tilting subcategory that $DA \in \mc$, and thus $tDA\in \uc \subseteq \modB$ by assumption (see \cref{torsion setup}).

\begin{lemma}\label{Injectives in mod B}
With the above notation, we have $\add (DB) = \add (tDA)$. Moreover, if $I$ is an indecomposable injective $A$-module, then $tI$ is either zero or indecomposable.
\end{lemma}

\begin{proof}
    We start by proving that $ DB \in \add (tDA)$, which yields $\add (DB) \subseteq \add (tDA)$. It suffices to show that there exists a monomorphism $DB\to tI$ for an injective $A$-module $I$, since $DB$ must then be a direct summand of $tI$. To this end, choose an epimorphism $p\colon B^t\to DB$. Taking the pushout of $p$ along the monomorphism $B^t\to (U_0^A)^t$ and using that $\tc$ is closed under quotients, we get a monomorphism $DB\to T$ with $T\in \tc$. Composing with the monomorphism $T\to tI$ obtained from an injective envelope $T\to I$, gives the desired conclusion. 
    
    For the reverse inclusion, consider the torsion subobject $tI$ of an injective $A$-module $I$. Let $tI\to M$ be an arbitrary monomorphism with $M\in \modB$. By a similar pushout argument as above, there exists a monomorphism $M\to T$ with $T\in \tc$. By \cite[Proposition 3.2]{AS}, any monomorphism from $tI$ to an object in $\tc$ must split since $I$ is injective. Hence, the composite $tI\to M\to T$ splits, and therefore $tI\to M$ splits. This means that $tI$ is an injective \mbox{$B$-module}. We can thus conclude that $tI \in \add (DB)$, which yields $tDA \in \add (DB)$.
    
    For the moreover part of the statement, note that since $tI$ is a submodule of $I$, the socle $\Soc(tI)$ of $tI$ is a submodule of the socle $\Soc(I)$ of $I$. By hypothesis, the injective module $I$ is indecomposable. This implies that $\Soc(I)$ is simple. It follows that $\Soc(tI)$ is either zero or simple. The former implies that $tI$ is zero, while the latter implies that $tI$ is indecomposable.
\end{proof}

The proposition below further shows that $tDA$ is $\Ext^d$-injective in $\uc$.

\begin{proposition}\label{Ext^d-injective in d-torsion class} 
If $I\in \modA$ is injective, then $\Ext^d_{A}(U,tI)=0$ for all $U\in \uc$. 
\end{proposition}

\begin{proof}
Let $U\in \uc$. An element in $\Ext^d_{A}(U,tI)$ can be represented by a $d$-exact sequence 
\[
0\to tI\to U_1\to \cdots \to U_d\to U\to 0,
\]
where we can assume $U_1,\dots, U_d\in \uc$ by \cref{theorem: extension closed}. Since $I$ is injective and $U_1\in\uc \subseteq \tc$, the morphism $tI\to U_1$ is split by \cite[Proposition 3.2]{AS}. This proves the claim.
\end{proof}

We now show that certain extension groups over the algebras $A$ and $B$ can be identified.

\begin{lemma}\label{Ext-iso in modA and modB}
Let $U\in \uc$ and $M\in \modB$. For any $j=1,\ldots,d$, there is a natural isomorphism 
\[
\Ext^j_A(U,M)\cong \Ext^j_B(U,M).
\]
\end{lemma}

\begin{proof}
Using \Cref{Injectives in mod B}, we can choose an injective coresolution of $M$ in $\modB$ of the form
\[
0\to M\to tI_0\to tI_1\to \cdots,
\]
where $I_j$ is injective in $\modA$ for all $j\geq 0$. By definition, we compute $\Ext^j_B(U,M)$ as the homology of the complex 
\[
\cdots \to \Hom_B(U,tI_{j-1})\to \Hom_B(U,tI_{j})\to \Hom_B(U,tI_{j+1})\to \cdots
\]
at $\Hom_B(U,tI_{j})$. 
Using that $U, tI_j \in \mc$ combined with \Cref{Ext^d-injective in d-torsion class}, we obtain $\Ext^i_A(U,tI_j)=0$ for all $1\leq i\leq d$ and $j\geq 0$. By dimension shifting, we therefore see that $\Ext^j_A(U,M)$ for $1\leq j\leq d$ is isomorphic to the homology of the complex 
\[
\cdots \to \Hom_A(U,tI_{j-1})\to \Hom_A(U,tI_{j})\to \Hom_A(U,tI_{j+1})\to \cdots
\]
at $\Hom_A(U,tI_{j})$. Since $\Hom_A(U,tI_{j})=\Hom_B(U,tI_{j})$ for all $j\geq 0$, this proves the claim.
\end{proof}

The lemma above allows us to relate $\tau_d$-rigidity in $\modA$ to $\tau_d$-rigidity in $\modB$. 

\begin{proposition}\label{tau_d rigid in mod A and mod B}
Let $U\in \uc$. Then $U$ is $\tau_d$-rigid in $\modA$ if and only if it is $\tau_d$-rigid in $\modB$.
\end{proposition}

\begin{proof}
Since $\fac U \subseteq \modB$, we have $\Ext^d_A(U,N)\cong \Ext^d_B(U,N)$ for all $N\in \fac U$ by \Cref{Ext-iso in modA and modB}. Hence, the claim follows from the characterisation of $\tau_d$-rigid modules in \Cref{lem:char tau_d-rigid modules}. 
\end{proof}

Recall that we use the minimal $\uc$-coresolution \eqref{eq: cores of A} to define $U^A= \bigoplus_{i=0}^{d} U_i^A$.

\begin{theorem}\label{thm: d-tilting}
The following statements hold:
\theoremlistfix
\begin{enumerate}
    \item The object $U^A$ is a $d$-tilting $B$-module.
    \phantomsection
    \label{thm: d-tilting:1}
    \item The object $U^A$ is an $\Ext^d$-projective generator of $\uc$.
    \phantomsection
    \label{thm: d-tilting:2}
\end{enumerate}
\end{theorem}

\begin{proof}
We first prove that given any $\Ext^d$-projective $X \in \uc$, the object $U=U^A\oplus X$ is a $d$-tilting $B$-module. Note that $U$ is a $B$-module as it is contained in $\uc$ and $\uc\subseteq \modB$.

First, observe that $\Ext_A^i(U,U)=0$ for $i = 1, \dots, d-1$ as $\mc$ is $d$-cluster tilting, as well as for $i=d$ as $U$ is $\Ext^d$-projective in $\uc$ by \cref{thm:U_i Ext^d-projective}. This implies that $\Ext_B^i(U,U) \cong 0$ for $i=1,\dots,d$ by \cref{Ext-iso in modA and modB}. 

Next, note that $U$ is $\tau_d$-rigid in $\modA$ by \cref{thm: equivalence of Extd projectivity}. This implies that $U$ is a $\tau_d$-rigid \mbox{$B$-module} by \cref{tau_d rigid in mod A and mod B}. Furthermore, $U$ is faithful since the morphism $B\to U_0^A$ from \eqref{eq: cores of B} is injective and $U_0^A\in \add (U)$. Hence, $\operatorname{pd}_B U^A\leq d$ by \cref{Martinez-Mendoza}. Since the sequence \eqref{eq: cores of B} gives the final property in the definition of a $d$-tilting $B$-module, the claim follows.

Now setting $X=0$ in the claim gives part \eqref{thm: d-tilting:1} of the theorem. For \eqref{thm: d-tilting:2}, let $X$ be an arbitrary $\Ext^d$-projective object in $\uc$. Since $U^A$ and $U^A\oplus X$ are both $d$-tilting by the claim, we must have $X\in \add (U^A)$. This proves the result.
\end{proof}

\begin{remark}
  The previous result is an analogue of \cite[Proposition 2.2 (c)]{AdachiIyamaReiten}, which states that maximal $\tau$-rigid modules in $\modA$ are always ($1$-)tilting modules in $\modB$.
  In our proof, the fact that the $\tau_d$-rigid object $U^A$ generates a $d$-torsion class is key in proving the vanishing of $\Ext^i_B(U^A,U^A)$ for every $i= 1, \dots, d-1$, a condition that is vacuous for $d=1$.
  Note moreover that every $\tau$-rigid object in $\modA$ generates a torsion class in $\mod A$ by \cite{AS}. 
  This is no longer true for $\tau_d$-rigid modules for $d>1$, as can be seen in \cref{ex:running1}.
\end{remark}

We now show that any functorially finite $d$-torsion class $\uc$ in $\mc$ gives rise to a $\tau_d$-rigid pair in $\mc$ with $|A|$ summands.

\begin{theorem}\label{thm:maximaltaunrigid}
Let $M_\uc$ be the basic $\Ext^d$-projective generator of $\uc$ and let $P_\uc$ be the maximal basic projective $A$-module with \mbox{$\Hom_A(P_\uc, U)=0$} for any $U \in \uc$. Then $(M_\uc, P_\uc)$ is a basic $\tau_d$-rigid pair in $\mc$ with $|M_\uc|+|P_\uc|= |A|$.
\end{theorem}

\begin{proof}
Note that $M_\uc$ exists and $\add(U^A) =\add(M_\uc)$ by \cref{thm: d-tilting}. Furthermore, \cref{thm: equivalence of Extd projectivity} shows that $M_\uc$ is $\tau_d$-rigid, and thus $(M_\uc,P_\uc)$ is a basic $\tau_d$-rigid pair in $\mc$.

To count the summands of $(M_\uc,P_\uc)$, note first that  $|A|=|B|+|I_\uc|$ by \cref{Injectives in mod B}, where $B=A/\Ann\uc$ and $I_\uc$ is the maximal basic injective $A$-module such that $\Hom_A(U, I_\uc)=0$ for all $U \in \uc$. Let $P$ be an indecomposable projective $A$-module, and denote by $I$ the indecomposable injective $A$-module whose socle is isomorphic to the top of $P$. It is well-known that for any $X \in \modA$, we have $\Hom_A(X,I)=0$ if and only if \mbox{$\Hom_A(P,X)=0$}; see for instance \cite[Lemma III.2.11]{bluebook1}. This yields $|I_\uc|=|P_\uc|$. Moreover, it follows from \cref{thm: d-tilting} that $M_\uc$ is a basic $d$-tilting module in $\modB$. In particular, this implies that $|M_\uc|=|B|$, so $|M_\uc|+|P_\uc|=|A|$.
\end{proof}

As a consequence, we have a well defined map
\begin{align}\label{map: ffntorsion to taun rigid}
\begin{array}{ccc}
\begin{Bmatrix}
    \text{functorially finite} \\
    \text{$d$-torsion classes in $\mc$}
\end{Bmatrix}
&
    \to 
    &
\begin{Bmatrix}
\text{basic $\tau_d$-rigid pairs in $\mc$} \\
\text{with $|A|$ summands }
\end{Bmatrix}\\\\
\uc &\mapsto  &(M_\uc, P_\uc),
\end{array}
\end{align}
where $M_\uc$ and $P_\uc$ are as in \cref{thm:maximaltaunrigid}.

\begin{proposition}\label{prop: injective}
The map \eqref{map: ffntorsion to taun rigid} is injective, with partial inverse given by
\[
(M_\uc,P_\uc) \mapsto \fac M_\uc \cap\mc.
\]
\end{proposition}
\begin{proof}
Recall that $\add(M_\uc)=\add(U^A)$ by \cref{thm: d-tilting}, and thus $\fac M_\uc = \fac U^A$. Moreover, \cref{lem: generation} shows that $\uc = \fac U_0^A \cap \mc$, which implies that
\(\uc = \fac U^A \cap \mc=\fac M_\uc \cap \mc\).
\end{proof}

Contrary to the classical case, when $d>1$, the map (\ref{map: ffntorsion to taun rigid}) is not always surjective. In particular, given an arbitrary basic $\tau_d$-rigid pair $(M,P)$ in $\mc$ with $|A|$ summands, the subcategory $\fac M \cap \mc$ need not be a $d$-torsion class in $\mc$, as demonstrated in the example below.

\begin{example}\label{ex:running1}
Consider the quiver $1 \xrightarrow{\alpha} 2\xrightarrow{\beta} 3,$
and let $A$ denote the path algebra of this quiver modulo the ideal generated by the relation $\alpha\beta$. The Auslander--Reiten quiver of $\modA$ is
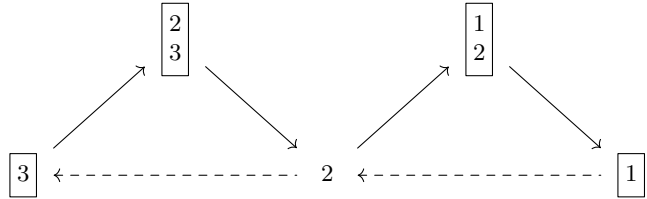
\[
			\begin{tikzpicture}[line cap=round,line join=round ,x=2.0cm,y=1.8cm, scale = 1]
				\clip(-2.3,0.8) rectangle (2.1,2.5);
					\draw [->] (-1.8,1.2) -- (-1.2,1.8);
					\draw [->] (0.2,1.2) -- (0.8,1.8);
					\draw [<-, dashed] (-1.8,1.0) -- (-0.2,1.0);
					\draw [<-, dashed] (0.2,1.0) -- (1.8,1.0);
					\draw [->] (-0.8,1.8) -- (-0.2,1.2);
					\draw [->] (1.2,1.8) -- (1.8,1.2);
				
				\begin{scriptsize}
					\draw (-2,1) node[draw] {$\Rep{3}$};
					\draw (0,1) node {$\Rep{2}$};
					\draw (2,1) node[draw]  {$\Rep{1}$};
					\draw (-1,2) node[draw]  {$\Rep{2\\3}$};
					\draw (1,2) node[draw]  {$\Rep{1\\2}$};
				\end{scriptsize}
			\end{tikzpicture}
\]
where the generators of the $2$-cluster tilting subcategory
$$\mc = \add\left\{ \rep{3} \oplus \rep{2\\3} \oplus \rep{1\\2} \oplus \rep{1}	 \right\} $$
are marked with rectangles. Since $A$ is an algebra of finite representation type, every $2$-torsion class in $\mc$ is functorially finite. A complete list of all $2$-torsion classes $\uc$ in $\mc$, along with their corresponding minimal torsion classes $\tc$ in $\modA$ satisfying $\uc = \tc \cap \mc$ and associated \mbox{$\tau_2$-rigid} pairs $(M_\uc,P_\uc)$ \mbox{in $\mc$}, can be found in \cref{tab:intersection2}. For an explicit example of how the third column in the table is obtained, consider the $2$-torsion class $\uc = \add \left\{ \rep{2\\3} \oplus \rep{1\\2} \oplus \rep{1}\right\}$. The minimal $\uc$-coresolution of $A$ is given by
$$\Rep{1\\2} \oplus \Rep{2\\3} \oplus \Rep{3} \longrightarrow \Rep{1\\2} \oplus \Rep{2\\3} \oplus \Rep{2\\3} \longrightarrow \Rep{1\\2} \longrightarrow \Rep{1} \longrightarrow 0.$$
It follows from \cref{thm:maximaltaunrigid} that $(M_\uc,P_\uc) = \left(\rep{2\\3}\oplus\rep{1\\2}\oplus\rep{1}, 0 \right)$ is the associated $\tau_2$-rigid pair \mbox{in $\mc$}. Note that as $\Ann \uc = (0)$, it moreover follows from \cref{thm: d-tilting} that $M_{\uc}$ is a $2$-tilting $A$-module. To see that the map (\ref{map: ffntorsion to taun rigid}) is in general not surjective, note that $\left(\rep{3}\oplus \rep{2\\3}, \rep{1\\2} \right)$ is a basic $\tau_2$-rigid pair \mbox{in $\mc$} with $|A|$ indecomposable direct summands. However, it does not appear in the third column of the table and is hence not obtained from a $2$-torsion class. Moreover, we see that 
    \[
    \fac \left(\rep{3}\oplus\rep{2\\3}\right) \cap \mc = \add \left\{ \rep{3}\oplus\rep{2\\3}\right\}
    \]
is not a $2$-torsion class in $\mc$ as it is not closed under \mbox{$2$-quotients}.

\begin{table}[t]\label{tab:intersection2}
    \centering
    \setlength{\extrarowheight}{6pt}
    \begin{tabular}{|c|c|c|}
    \hline
        $\uc$ & $\tc$ & 
        $(M_\uc,P_\uc)$\\[6pt]
        \hline
        \hline
         $\mc$ & $\modA$ & $\left(\rep{3}\oplus \rep{2\\3} \oplus \rep{1\\2}, 0\right)$\\[6pt]
         \hline
        $\add\left\{\rep{2\\3}\oplus \rep{1\\2} \oplus \rep{1}\right\}$ & $\add\left\{\rep{2\\3}\oplus \rep{2}\oplus \rep{1\\2} \oplus \rep{1}\right\}$ & $\left( \rep{2\\3} \oplus \rep{1\\2}\oplus \rep{1}, 0\right)$
        \\[6pt]
        \hline
        $\add\left\{\rep{1\\2} \oplus \rep{1}\right\}$ & $\add\left\{ \rep{1\\2} \oplus \rep{1}\right\}$ & $\left( \rep{1\\2} \oplus \rep{1}, \rep{3}\right)$
        \\[6pt]
        \hline
        $\add\left\{\rep{1}\right\}$ & $\add\left\{\rep{1}\right\}$ & $\left(\rep{1},  \rep{3} \oplus \rep{2\\3}\right)$
        \\[6pt]
        \hline
        $\add\left\{\rep{3}\right\}$ & $\add\left\{\rep{3}\right\}$ & $\left(\rep{3}, \rep{2\\3} \oplus \rep{1\\2}\right)$\\[6pt]
        \hline
        $\left\{0\right\}$ & $\left\{0\right\}$ & $\left(0, \rep{3}\oplus \rep{2\\3} \oplus \rep{1\\2}\right)$\\[6pt]
        \hline
    \end{tabular}\\
    \caption{List of all $2$-torsion classes $\uc$ in $\mc$ and their 
    associated minimal torsion classes $\tc$ in $\mod A$ and \mbox{$\tau_2$-rigid} pairs $(M_\uc,P_\uc)$ in $\mc$.}
\end{table}
\end{example}

We finish this section by showing that the basic $\Ext^d$-projective generator $M_\uc$ of $\uc$ can also be characterised as the maximal basic $\tau_d$-rigid module in $\uc$ satisfying $\operatorname{Fac}M_\uc\cap \mc=\uc$.

\begin{proposition}
    Let $U\in \uc$ be a basic $\tau_d$-rigid module satisfying $\operatorname{Fac}U\cap \mc=\uc$, and assume that $U$ is maximal with respect to these properties. Then $U\cong M_\uc$. 
\end{proposition}

\begin{proof}
   Let $U'\in \uc$ and choose an epimorphism $U^n\to U'$. This yields a monomorphism
    \[
    0\to \Hom_A(U',\tau_dU)\to \Hom_A(U^n,\tau_dU).
    \]
    As $U$ is $\tau_d$-rigid, we have $\Hom_A(U^n,\tau_dU)=0$, and hence $\Hom_A(U',\tau_dU)=0$. It follows from \cref{thm: equivalence of Extd projectivity} that $U$ is $\Ext^d$-projective in $\uc$. Since $M_\uc$ is the maximal basic $\Ext^d$-projective in $\uc$, we get that $U \in \add (M_\uc)$. But since $M_\uc$ is $\tau_d$-rigid and satisfies \mbox{$\operatorname{Fac}M_\uc\cap \mc=\uc$} and $U$ is basic and maximal with respect to these properties, we must have $U\cong M_\uc$.
\end{proof}

\section{From \texorpdfstring{$\tau$}{tau}-tilting to \texorpdfstring{$\tau_d$}{tau-d}-tilting theory}\label{sec: tau-tilting to tau_d-tilting}

Given a functorially finite $d$-torsion class $\uc$ in a $d$-cluster tilting subcategory $\mc$ of $\mod A$, we know from \cref{sec:d-tors to tau_d} that we can use a minimal $\uc$-coresolution of $A$ to construct an $\Ext^d$-projective generator of $\uc$. In \cref{sec: tau-tilting to tau_d-tilting} we consider the case where $\uc = \tc \cap \mc$ for a functorially finite torsion class $\tc$ in $\mod A$. It is then well-known that $\tc = \fac T$ for some support $\tau$-tilting module $T$ \cite[Theorem 2.7]{AdachiIyamaReiten}. The aim in \cref{subsec: ff torsion and higher torsion} is to show that an $\Ext^d$-projective generator of $\uc$ can also be obtained by computing a minimal $\mc$-coresolution of $T$. In \cref{subsec: d-APR} we moreover explain how classical APR tilting in the sense of \cite{APRtilting} and $d$-APR tilting in the sense of \cite{IyamaOppermann} fit into this framework. Finally, in \cref{sec:slices} we apply the theory of $d$-torsion classes to show that slices in $\mc$ give rise to $d$-tilting modules, extending a result from \cite{IyamaOppermann} to a setup where no restriction on the global dimension of $A$ is required.

\subsection{From support \texorpdfstring{$\tau$}{tau}-tilting modules to \texorpdfstring{$\tau_d$}{tau-d}-rigid pairs}
\label{subsec: ff torsion and higher torsion}

Throughout this subsection, we work under the following setup.

\begin{setup}\label{ffmodA setup}
Let $\mc \subseteq \mod A$ be a $d$-cluster tilting subcategory. Suppose that $\tc \subseteq \mod A$ is a functorially finite torsion class, and let $T$ be the basic support $\tau$-tilting module satisfying \mbox{$\tc = \fac T$}. Assume that $\tc$ induces a $d$-torsion class $\uc = \tc \cap \mc$ in $\mc$.
\end{setup}

Recall that the support $\tau$-tilting module $T$ corresponds to a basic maximal $\tau$-rigid pair in $\mod A$ \cite[Proposition 2.3 (b)]{AdachiIyamaReiten}. The results in \cref{subsec: ff torsion and higher torsion} will allow us to use a minimal $\uc$-coresolution of $T$ to obtain a basic $\tau_d$-rigid pair in $\mc$ with $|A|$ summands; see \cref{cor:sameeitherway}.

We first show that since the torsion class $\tc$ is functorially finite, so is the $d$-torsion \mbox{class $\uc$}. 

\begin{proposition}\label{prop:T ff implies U ff}
 The $d$-torsion class $\uc$ is functorially finite in $\modA$.
\end{proposition}

\begin{proof}
As any $d$-torsion class in $\mc$ is contravariantly finite, we only need to check that $\uc$ is covariantly finite. Consider $X\in\modA$. Since $\tc$ is functorially finite, we can take a left \mbox{$\tc$-approximation} $X\to T_0$ of $X$. Let $T_0\to M_0$ be a minimal left $\mc$-approximation of $T_0$, which is also a left $\uc$-approximation by \cref{cor: minimal M-appr. in U}. It follows that the composition $X\to T_0\to M_0$ is a left \mbox{$\uc$-approximation} of $X$.
\end{proof}

Note that \cref{prop:T ff implies U ff} allows us to apply results from \cref{sec: ff-d-torsion,sec:d-tors to tau_d} whenever we work under \cref{ffmodA setup}. The next lemma shows that the minimal $\mc$-coresolution and the minimal \mbox{$\uc$-coresolution} of $T$ coincide.

\begin{lemma}\label{Lemma:Mcores=Ucores}
The minimal $\uc$-coresolution of $T$ is an exact sequence of the form
\begin{align}\label{eq: cores of T}
0 \xrightarrow{} T \xrightarrow{} U_0^T \xrightarrow{} U_1^T \xrightarrow{} \cdots \xrightarrow{} U_{d-1}^T \to 0,
\end{align} 
and it is also the minimal $\mc$-coresolution of $T$.
\end{lemma}

\begin{proof}
As $T \in \tc$, the minimal left $\uc$-approximation $T \to U_0^T$ is both a monomorphism and a minimal left $\mc$-approximation by \cref{cor: minimal M-appr. in U}. Hence, by the construction explained above \cref{thm:U_i Ext^d-projective}, the minimal $\uc$-coresolution of $T$ is an exact sequence of the form
$$0\to T \xrightarrow{} U_0^T \xrightarrow{} U_1^T \xrightarrow{} \cdots \xrightarrow{} U_d^T \to 0,$$ 
and it is also a minimal $\mc$-coresolution. Since minimal \mbox{$\mc$-coresolutions} have at most $d$-terms, it follows that $U_d^T=0$. This proves the claim.
\end{proof}

In what follows, we use the sequence \eqref{eq: cores of T} to define $U^T \colonequals \bigoplus_{i=0}^{d-1} U^T_i$. 

\begin{proposition}\label{prop:generators}
We have $\uc = \fac U^T_0 \cap \mc =\fac U^T \cap \mc$.
\end{proposition}

\begin{proof}
Since $\uc = \fac T \cap \mc \subseteq \fac T$, the first equality follows directly from \cref{lem: generation} by taking \mbox{$M=T$}. The second is immediate, as $U^T \in \uc \subseteq \tc$ implies $\fac U^T \subseteq \tc$ and $U_0^T$ is a direct summand of $U^T$.  
\end{proof}

By our work in \cref{sec: ff-d-torsion}, we know that $U^T$ is $\Ext^d$-projective in $\uc$. We now prove an analogue of \cref{thm: d-tilting}. Recall that $B \colonequals A/\Ann \uc$. 

\begin{theorem}\label{thm:U_T-approx}

The following statements hold:
\theoremlistfix
\begin{enumerate}
    \item The object $U^T$ is a $d$-tilting $B$-module.
    \phantomsection
    \label{thm2: d-tilting:1}
    \item The object $U^T$ is an $\Ext^d$-projective generator of $\uc$.
    \phantomsection
    \label{thm2: d-tilting:2}
\end{enumerate}
\end{theorem}

\begin{proof}
Recall that a subcategory of $\modA$ is called \textit{thick} if it closed under direct summands and if for any exact sequence
\[
0\to M_1\to M_2\to M_3\to 0
\]
where two out of the modules $M_1,M_2$ and $M_3$ are in the subcategory, then all of the modules are in the subcategory. We let $\Thick(M)$ denote the smallest thick subcategory of $\mod A$ containing a module $M$. By the discussion after \cite[Proposition 7.2.10]{Krause22}, it follows that $M \in \mod A$ is a tilting module if it has finite projective dimension, satisfies $\Ext^i_A(M,M)=0$ for all $i>0$, and $\Thick(M)$ is equal to the subcategory of $\modA$ consisting of modules of finite projective dimension. We show that $U^T$ is a tilting $B$-module by verifying these three conditions.

Since it is well-known that $\Ext^1_A(T,X)=0$ for every $X \in \tc$, see e.g.\ \cite[Theorem 2.7]{AdachiIyamaReiten}, it follows from \cref{thm:U_i Ext^d-projective} that $U^T$ is $\Ext^d$-projective in $\uc$. This implies that $U^T \in \add (U^A)$, which again yields $\operatorname{pd}_B (U^T) \leq d$ and $\Ext_B^i(U^T,U^T)=0$ for $i > 0$ by \cref{thm: d-tilting}. Now since $T$ is a tilting $B$-module by \cite{Smalo}, the subcategory $\Thick(T)$ is equal to the subcategory of $B$-modules of finite projective dimension. From the exact sequence \eqref{eq: cores of T}, we see that $\Thick(T)\subseteq \Thick(U^T)$. Since $U^T$ has finite projective dimension, the subcategory $\Thick(U^T)$ must therefore also consist of the $B$-modules of finite projective dimension. This shows that $U^T$ is a $d$-tilting $B$-module.

For (\ref{thm2: d-tilting:2}), we again use that $U^T$ is contained in $\add (U^A)$, as $U^A$ is an $\Ext^d$-projective generator of $\uc$. But since both $U^T$ and $U^A$ are tilting $B$-modules, they have the same number of non-isomorphic indecomposable summands, so $U^A$ must also be contained in $\add (U^T)$. This shows that $U^T$ is an $\Ext^d$-projective generator of $\uc$.
\end{proof}

Recall that $U^A$ is defined by using the $\uc$-coresolution \eqref{eq: cores of A} from \cref{sec:d-tors to tau_d}, and that $M_\uc$ is the basic $\Ext^d$-projective generator of $\uc$. 

\begin{corollary}\label{cor:sameeitherway}
The diagram \eqref{diagram:introduction} from the introduction commutes. 
\end{corollary}

\begin{proof}
As both $U^A$ and $U^T$ are $\Ext^d$-projective generators of $\uc$ by \cref{thm: d-tilting} and \cref{thm:U_T-approx}, we have $\add (U^A) = \add (U^T)$, and the result follows.
\end{proof}

\cref{cor:sameeitherway} allows us to associate a $\tau_d$-rigid pair in $\mc$ to the support $\tau$-tilting \mbox{module $T$}, namely the pair $(U^T,P_{\uc})$, where $P_\uc$ is the maximal basic projective $A$-module satisfying that \mbox{$\Hom_A(P_\uc, U)=0$} for any $U \in \uc$ as in \cref{thm:maximaltaunrigid}. The basic part of this pair is the $\tau_d$-rigid pair $(M_{\uc},P_{\uc})$. It will follow from our results in \cref{sec:silting} that these $\tau_d$-rigid pairs are maximal; see \cref{cor: maximality}.

We illustrate the results of this subsection in the setup of \cref{ex:running1}.

\begin{example}\label{ex:running6}
Let $A$ and $\mc$ be as in \cref{ex:running1}, and consider the $2$-torsion class 
\[
\uc = \add \left\{ \rep{2\\3} \oplus \rep{1\\2} \oplus \rep{1}\right\} \subseteq \mc.
\]
Note that $\uc = \tc \cap \mc$ for $\tc = \add\left\{\rep{2\\3}\oplus \rep{2}\oplus \rep{1\\2} \oplus \rep{1}\right\}$ and that $\tc$ corresponds to the support $\tau$-tilting module $T = \rep{2\\3}\oplus\rep{2}\oplus\rep{1\\2}$. The minimal $\uc$-coresolution of $T$ is given by
\[
\Rep{2\\3} \oplus \Rep{2} \oplus \Rep{1\\2} \longrightarrow \Rep{2\\3} \oplus \Rep{1\\2} \oplus \Rep{1\\2} \longrightarrow \Rep{1} \longrightarrow 0.
\]
It hence follows from \cref{cor:sameeitherway} that $(U^T,P_{\uc}) = \left(\rep{2\\3}\oplus \rep{1\\2} \oplus\rep{1}, 0\right)$ is a $\tau_2$-rigid pair in $\mc$, as we have already seen in \cref{ex:running1}.
\end{example}

\subsection{APR tilting and higher torsion classes} \label{subsec: d-APR}

Throughout this subsection, let $\mc$ be a \mbox{$d$-cluster} tilting subcategory of $\mod A$. Our goal is to prove that the $d$-APR tilting modules introduced in \cite{IyamaOppermann} arise as $\Ext^d$-projective generators of functorially finite split \mbox{$d$-torsion} classes in $\mc$. This is done by showing that the minimal $\mc$-coresolutions of the classical APR tilting modules, introduced in \cite{APRtilting}, consist of sums of summands of the corresponding $d$-APR tilting modules. 

Assume that $A$ is basic and that $P$ is a simple projective but not injective $A$-module. We write $A = P \oplus Q$. The $A$-module
\[
T^P_1 = \tau^- P \oplus Q
\]
is called the \textit{APR tilting module associated to $P$}. It is well-known that $T_1^P$ is a tilting module in $\mod A$ and that $\fac T_1^P$ is a functorially finite split torsion class; see e.g.\ \cite[Example VI 2.8 (c)]{bluebook1}. This leads to the following result. Recall that a $d$-torsion class is faithful if its annihilator is zero.

\begin{proposition}\label{prop:dAPRtorsion}
    Let $T_1^P$ be an APR tilting module. Then $\fac T_1^P$ induces a functorially finite faithful split $d$-torsion class $\uc = \fac T_1^P \cap \mc$ in $\mc$.
\end{proposition}

\begin{proof}
As $\fac T_1^P$ is a split torsion class in $\mod A$, it induces a split $d$-torsion class \mbox{$\uc = \fac T_1^P \cap \mc$} in $\mc$ by \cref{Prop:SplitTorsionGivesdTorsion}. Since $\fac T_1^P$ is functorially finite, so is $\uc$ by \cref{prop:T ff implies U ff}. As $T_1^P$ is tilting, the subcategory $\fac T_1^P$ is faithful, and thus there is an object $X\in \fac T^P_1$ and a monomorphism $A \to X$. 
Consider the minimal left \mbox{$\uc$-approximation} $X\to U$ of $X$, which is a monomorphism by \cref{cor: minimal M-appr. in U}.
This gives a monomorphism $A \to U$, so $\uc$ is faithful.
\end{proof}

The module
$
T^P_d = \tau_d^- P \oplus Q
$
is called the \textit{weak $d$-APR tilting module associated to $P$} \cite[Definition 3.1]{IyamaOppermann}. Any weak \mbox{$d$-APR} tilting module is a $d$-tilting module \cite[Theorem 3.2]{IyamaOppermann}. We now recover this fact using the theory of $d$-torsion classes. 

\begin{theorem}\label{thm:d-APR}
    Let $T_d^P$ be a weak $d$-APR tilting module. Then $T_d^P$ is the basic $\Ext^d$-projective generator of $\uc = \fac T_1^P \cap \mc$. In particular, the module $T_d^P$ is $d$-tilting.
\end{theorem}

\begin{proof}
Consider a minimal injective $d$-copresentation
\[
0\to P\to I_0\to I_1 \to \cdots \to I_{d}
\]
of $P$. Applying the Nakayama functor $\nu^-=\Hom_A(DA,-)$ and using that \mbox{$\Ext^i_A(DA,P)=0$} for $i=1,\dots,d-1$, we get the exact sequence
\[
0\to \nu^-P\to P_0\to P_1\to \cdots \to P_{d}\to \tau_d^-P\to 0
\]
where $P_i=\nu^-I_i$. As $\tau^-P$ is isomorphic to the cokernel of $P_0\to P_1$ by definition, we have an exact sequence
\[
0\to \tau^-P\to P_2\to \cdots \to P_{d}\to \tau_d^{-}P\to 0.
\]
This gives a minimal $\mc$-coresolution of $\tau^-P$. Adding the identity morphism $Q \xrightarrow{1_Q} Q$ to the first morphism in this sequence, we get a minimal $\mc$-coresolution
\begin{equation*}
0\to T_1^P \to Q \oplus P_2\to \cdots \to P_{d}\to \tau_d^{-}P\to 0
\end{equation*}
of $T_1^P$. Since $P$ is simple and projective, it cannot be a summand of $P_2\oplus \dots \oplus P_d$, as this would contradict the minimality of the sequence. Hence, we have $P_2\oplus \dots \oplus P_d \in\add (Q)$ and 
\[
\operatorname{add}(T_d^P)=\add (Q \oplus P_2\oplus \cdots \oplus P_{d}\oplus \tau_d^{-}P).
\]
By \cref{thm:U_T-approx} and \cref{prop:dAPRtorsion}, the module $Q \oplus P_2\oplus \cdots \oplus P_{d}\oplus \tau_d^{-}P$ is $d$-tilting and an $\Ext^d$-projective generator of $\uc$, so the same must hold for $T_d^P$. This proves the claim.
\end{proof}

\begin{remark}
    Given a $d$-APR tilting module $T_d^P$, it is known from \cite[Theorem 9.2]{Jorgensen} that $\fac T_d^P \cap \mc$ is a split $d$-torsion class in $\mc$. We can recover this fact by combining \cref{prop:generators,prop:dAPRtorsion,thm:d-APR}.
\end{remark}

\subsection{Slices and higher torsion classes} \label{sec:slices}
In \cite{IyamaOppermann} the authors introduce and consider slices in a higher analogue of the derived category, which can be associated to any $d$-cluster tilting subcategory in the case where the algebra $A$ is of global \mbox{dimension $d$}. In particular, it is shown that slices give rise to tilting complexes \cite[Theorem 4.15]{IyamaOppermann}. In this subsection we show that the same conclusion holds without the assumption on the global dimension \mbox{of $A$}, using the theory of $d$-torsion classes. Note that we only consider slices in the $d$-cluster tilting subcategory itself, due to the lack of a good higher analogue of the derived category in this setting. 

Throughout this subsection we fix a  $d$-cluster tilting module $X$ in $\modA$, meaning that the subcategory \mbox{$\mc=\operatorname{add}(X) \subseteq \mod A$} is $d$-cluster tilting. Let $\operatorname{ind}\mc$ be the class of indecomposable $A$-modules in $\mc$. A \textit{path} $M \leadsto N$ between $M,N\in \operatorname{ind}\mc$ is a sequence 
\[
M\to M_1 \to \cdots \to M_n\to N
\]
of non-zero morphisms with $M_1,\dots , M_n\in \operatorname{ind}\mc$. We denote the class of indecomposable direct summands of an object $S \in \mc$ by $\operatorname{ind}S$.
\begin{definition}\label{Slice}
    An object $S\in \mc$ is a \textit{slice} if the following hold:
    \theoremlistfix
    \begin{enumerate}
        \item For any indecomposable injective $A$-module $I$, there exists an integer $j\geq 0$ such that $0\neq \tau^j_d(I)\in \operatorname{add}(S)$.
        \phantomsection
        \label{Slice:1}
        \item For any non-projective $M\in  \operatorname{ind}\mc$, at most one of $M$ and $\tau_d(M)$ lies in $\add (S)$.
        \phantomsection
        \label{Slice:2}
        \item Any path which starts and ends in $\operatorname{ind}S$ must lie entirely in $\operatorname{ind}S$. 
        \phantomsection
        \label{Slice:3}
    \end{enumerate}
\end{definition}

For a slice $S$ in $\mc$, there is a faithful split $d$-torsion class with $S$ as an \mbox{$\Ext^d$-projective generator.}

\begin{theorem}\label{BigResultSlices}
    Let $S$ be a slice in $\mc$. Then
    \begin{align*}
       & \uc\colonequals \operatorname{add}\{M \in \operatorname{ind}\mc \mid \text{there exists a path }S'\leadsto M \text{ with }S'\in \operatorname{ind}S\}
    \end{align*}
    is a functorially finite faithful split $d$-torsion class in $\mc$. Furthermore, the slice $S$ is an $\Ext^d$-projective generator of $\uc$.
\end{theorem}

\begin{proof}
We first show that $\uc$ is a split $d$-torsion class. For $M\in \mc$, consider the decomposition $M\cong U_M\oplus V_M$ where $U_M$ is a direct summand of $M$ which is contained in $\uc$ and is maximal with this property. Let $V$ be an indecomposable direct summand of $V_M$. Since $U_M$ is maximal, the object $V$ is not in $\uc$. Hence, there exists no path $S'\leadsto V$ with $S'\in \operatorname{ind}S$, so there are no non-zero morphisms from objects in $\uc$ to $V$. Since $V$ was arbitrary, this implies that $\Hom_A(U,V_M)=0$ for all $U\in \uc$, and $\uc$ is thus a split $d$-torsion class in $\mc$. 
  
The fact that $\uc$ is functorially finite follows from $\mc$ having finitely many isomorphism classes of indecomposable objects as $\mc = \add (X)$. We now show that $\uc$ is faithful. First, note that for any non-projective $M\in \operatorname{ind}\mc$, the $d$-almost split sequence
  \[
  0\to \tau_d(M)\to E_1\to \cdots \to E_d\to M\to 0
  \]
gives a path $\tau_d(M)\leadsto M$. To show faithfulness, it suffices to prove that any indecomposable injective $A$-module $I$ is contained in $\uc$. Given such an $I$, there exists an integer $j\geq 0$ with $0 \neq \tau_{d}^j(I)\in \operatorname{add}(S)$ by the definition of a slice. If $j=0$, then $\tau_d^j(I)=I \in \add(S)$, so $I \in \uc$. Otherwise, note that $\tau_d^i(I)$ is non-projective for $0 \leq i \leq j-1$. By the argument above, we hence have paths 
  \[
  \tau_d^j(I)\leadsto \tau_d^{j-1}(I)\leadsto \cdots \leadsto\tau_d(I)\leadsto I.
  \]
Concatenation gives a path $\tau_d^j(I)\leadsto I$, and so $I \in\uc$. 
  
Next, we prove that $S$ is $\Ext^d$-projective in $\uc$. Let $S'\in \operatorname{ind}S$. If $S'$ is projective, then it is clearly $\Ext^d$-projective. Assume hence that $S'$ is non-projective, and consider the path $\tau_d(S')\leadsto S'$. If $\tau_d(S')\in \uc$, then there exists an $S''\in \operatorname{ind}S$ and a path $S''\leadsto \tau_d(S')$. But then the concatenation $S''\leadsto \tau_d(S')\leadsto S'$ starts and ends in $\operatorname{ind}S$, and hence it must lie entirely in $\operatorname{ind}S$ by \cref{Slice} \eqref{Slice:3}. This implies that $\tau_d(S')\in \operatorname{add}(S)$, which contradicts \cref{Slice} \eqref{Slice:2}. We can thus conclude that $\tau_d(S')$ is not contained in $\uc$. Since $\tau_d(S')$ is indecomposable, it must be contained in the split $d$-torsion free class
  \[
  \vc=\{M\in \mc\mid \Hom_A(U,M)=0 \text{ for all }U\in \uc\}
  \]
associated to $\uc$. By \cref{thm: equivalence of Extd projectivity}, this implies that $S'$ is $\Ext^d$-projective in $\uc$, and it follows that $S$ is $\Ext^d$-projective in $\uc$.
  
We finish by showing that $S$ is an $\Ext^d$-projective generator of $\uc$. Note first that we can construct an injective map from isomorphism classes of indecomposable injective $A$-modules to isomorphism classes of indecomposable summands of $S$, given by
  \[
  I\mapsto \tau_d^j(I) \text{ for some choice of an integer }j\geq 0 \text{ such that }0\neq \tau_d^j(I)\in \operatorname{add}(S).
  \]
This yields $|A|\leq |S|$. By \cref{thm: d-tilting}, the $A$-module $U^A$ is $d$-tilting and an $\Ext^d$-projective generator of $\uc$, as $\uc$ is faithful and functorially finite. In particular, we have $|U^A|=|A|$ since $U^A$ is $d$-tilting. As $U^A$ is an $\Ext^d$-projective generator, it follows that $S$ is contained in $\operatorname{add}(U^A)$, so $|S|\leq |U^A|$. Combining this, we get that $|S|=|U^A|$ and that the isomorphism classes of indecomposable summands of $U^A$ are equal to the ones of $S$. This gives \mbox{$\operatorname{add}(S)=\operatorname{add}(U^A)$}, which shows that $S$ is an $\Ext^d$-projective generator of $\uc$.
\end{proof}

We get the following immediate consequence of the theorem above.

\begin{corollary}\label{Corollary:SliceIsTilting}
     Let $S$ be a slice in $\mc$. Then $S$ is a $d$-tilting $A$-module.
\end{corollary}

\begin{proof}
    By \Cref{BigResultSlices}, the slice $S$ is an $\Ext^d$-projective generator of a functorially finite faithful $d$-torsion class. The claim thus follows from \cref{thm: d-tilting}.
\end{proof}

The examples below demonstrate that our results extend those in \cite{IyamaOppermann}, as we do not require global dimension $d$.
\begin{example}\label{Example:Slice1}
Let $A$ be the path algebra of the quiver \(Q=(1 \xrightarrow{\alpha_1} 2 \xrightarrow{\alpha_2} 3\xrightarrow{\alpha_3}4 \xrightarrow{\alpha_4} 5)\) modulo the ideal generated by the relations $\alpha_{i}\alpha_{i+1}$ for $1\leq i\leq 3$. Consider its Auslander--Reiten quiver
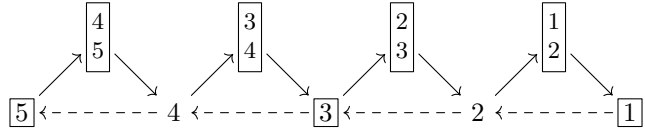
\[
\begin{tikzpicture}[line cap=round,line join=round, inner sep = 2, outer sep = 2]
\node[draw] (5) at (0,0) {$5$};
\node (4) at (2,0) {$4$};
\node[draw] (3) at (4,0) {$3$};
\node (2) at (6,0) {$2$};
\node[draw] (1) at (8,0) {$1$};

\node[draw] (45) at (1,1) {$\Rep{4\\5}$};

\node[draw] (34) at (3,1) {$\Rep{3\\4}$};

\node[draw] (23) at (5,1) {$\Rep{2\\3}$};

\node[draw] (12) at (7,1) {$\Rep{1\\2}$};

\draw[->] (5) -- (45);
\draw[->] (45) -- (4);
\draw[->] (4) -- (34);
\draw[->] (34) -- (3);
\draw[->] (3) -- (23);
\draw[->] (23) -- (2);
\draw[->] (2) -- (12);
\draw[->] (12) -- (1);
\draw[->, dashed] (1) -- (2);
\draw[->, dashed] (2) -- (3);
\draw[->, dashed] (3) -- (4);
\draw[->, dashed] (4) -- (5);
\end{tikzpicture}
\]
\noindent
where the dashed arrows indicate the Auslander--Reiten translation and the generators of the $2$-cluster tilting subcategory 
\[
\mc = \add\left\{ \rep{5} \oplus \rep{4\\5} \oplus \rep{3\\4} \oplus \rep{3}\oplus \rep{2\\3} \oplus \rep{1\\2} \oplus \rep{1}	 \right\} 
\] 
are marked with rectangles.
A straightforward computation shows that 
$$S=\rep{4\\5} \oplus \rep{3\\4} \oplus \rep{3}\oplus \rep{2\\3} \oplus \rep{1\\2}$$ 
is a slice. This provides an example of a slice of a $2$-cluster tilting subcategory of an algebra of global dimension $4$. The endomorphism algebra $B_S$ of $S$ is isomorphic to the path algebra of the quiver $Q$ modulo the relations $\alpha_1\alpha_2$ and $\alpha_3\alpha_4$. By \cref{Corollary:SliceIsTilting}, the module $S$ is $2$-tilting, so $B_S$ is derived equivalent to $A$. However, the module category of $B_S$ cannot contain a $2$-cluster tilting subcategory by \cite[Theorem 3.1]{HauglandJacobsenSchroll}.
\end{example}

Note that the slice in \Cref{Example:Slice1} is a weak $2$-APR tilting module. Below we provide an example of a slice that is neither given by a weak $2$-APR tilting module, nor by a $2$-APR tilting complex in the sense of \cite[Definition 3.14]{IyamaOppermann}.

\begin{example}\label{Example:Slice2}

For this example, we refer the reader to \Cref{subsec: higher nakayama} or \cite[Section 2]{Higher Nakayama} for the definition and basic properties of higher Nakayama algebras. Let $A_{\underline{\ell}}^2$ be the higher Nakayama algebra associated to the Kupisch series $\ell=(1,2,3,2,3)$. This algebra has a $2$-cluster tilting subcategory depicted below
    
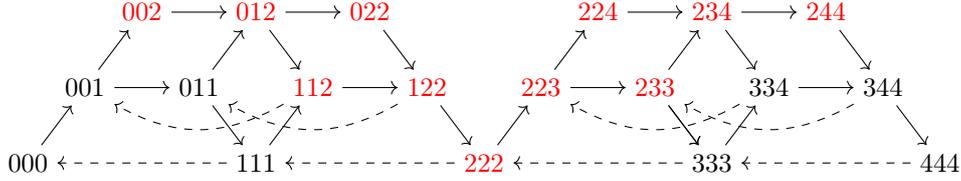
\[
    \begin{tikzpicture}[xscale=0.75]
	\node (000) at (-1,0) {000};
	\node (111) at (3,0) {111};
	\node (222) at (7,0) {\color{red}222};
	\node (333) at (11,0) {333};
    \node (444) at (15,0) {444};
	
	\node (001) at (0,1) {001};
	\node (011) at (2,1) {011};
	\node (112) at (4,1) {\color{red}112};
	\node (122) at (6,1) {\color{red}122};	
	\node (223) at (8,1) {\color{red}223};
	\node (233) at (10,1) {\color{red}233};
    \node (334) at (12,1) {334};
    \node (344) at (14,1) {344};

    \node (002) at (1,2) {\color{red}002};
	\node (012) at (3,2) {\color{red}012};	
	\node (022) at (5,2) {\color{red}022};
	\node (224) at (9,2) {\color{red}224};	
	\node (234) at (11,2) {\color{red}234};
    \node (244) at (13,2) {\color{red}244};
	
	\draw[->] (000) -- (001);
	\draw[->] (001) -- (011);
	\draw[->] (011) -- (111);
	\draw[->] (111) -- (112);
	\draw[->] (112) -- (122);
	\draw[->] (122) -- (222);
	\draw[->] (222) -- (223);
	\draw[->] (223) -- (233);
	\draw[->] (233) -- (333);
    \draw[->] (233) -- (333);
	\draw[->] (333) -- (334);
	\draw[->] (334) -- (344);
	\draw[->] (344) -- (444);
    \draw[->] (001) -- (002);
    \draw[->] (002) -- (012);
    \draw[->] (011) -- (012);
    \draw[->] (012) -- (112);
    \draw[->] (012) -- (022);
    \draw[->] (022) -- (122);
    \draw[->] (223) -- (224);
    \draw[->] (233) -- (234);
    \draw[->] (224) -- (234);
    \draw[->] (234) -- (334);
    \draw[->] (234) -- (244);
    \draw[->] (244) -- (344);
	
	\draw[->, dashed] (111)--(000);
	\draw[->, dashed] (222)--(111);
	\draw[->, dashed] (333)--(222);
    \draw[->, dashed] (444)--(333);
	
	\draw[->, dashed] (344) to[bend left=23] (233);
	\draw[->, dashed] (334) to[bend left=23] (223);
    \draw[->, dashed] (122) to[bend left=23] (011);
	\draw[->, dashed] (112) to[bend left=23] (001);
\end{tikzpicture}
\]
where the vertex $i$ represent the indecomposable $M_i$ in $\mc_{\underline{\ell}}^2$ up to isomorphism, the arrows represent irreducible morphisms in $\mc_{\underline{\ell}}^2$ and the dashed arrows indicate the higher Auslander--Reiten translation. Then $\mc_{\underline{\ell}}^2$ admits a slice given by the vertices marked in red. The indecomposable represented by $222$ is not projective, and
the same is true for its $2$-Auslander--Reiten translate. Hence, this slice is neither a weak $2$-APR tilting module nor a $2$-APR tilting complex. Moreover, since we have an exact sequence
\[
0\to M_{222}\to M_{223}\to M_{233}\to M_{333}\to 0
\]
where $M_{223}$ and $M_{233}$ are projective and $M_{222}$ is not projective, the algebra $A^2_{\underline{\ell}}$ must have global dimension greater than $2$.
\end{example}

\section{\texorpdfstring{$d$}{d}-cluster tilting: from modules to \texorpdfstring{$(d+1)$}{(d+1)}-term complexes}\label{subsec:dCT}

In the main result of this section, we prove that any $d$-cluster tilting subcategory of $\modA$ induces a $d$-cluster tilting subcategory of the category of $(d+1)$-term complexes; see \cref{theorem:dCT exact}. This result will be instrumental for the proofs in \cref{sec:silting}. Our approach moreover produces novel examples of $d$-cluster tilting subcategories and $d$-exact categories that are interesting in their own right.

In \cref{ssec:notation and main result} we introduce the necessary notation and state the main results. Next, in \cref{ssec:exact and d-tilting} we recall some definitions and facts concerning exact categories. Following \cite{Iyama2007a,Jasso}, we introduce the notion of a $d$-cluster tilting subcategory in this context, before proving some convenient results for showing that a subcategory of an exact category is $d$-cluster tilting. \cref{Subsection:From $(d+1)$-term complexes to its homotopy category,subsec: (d+1)-term} are concerned with the category of $(d+1)$-term complexes of projectives and its homotopy category. Finally, in \cref{subsec: pf of main results} we give the proofs of the main results. 

\subsection{Notation and statement of main results}\label{ssec:notation and main result}
Throughout \cref{ssec:notation and main result}, we assume that \mbox{$\mc\subseteq\mod A$} is a $d$-cluster tilting subcategory. The notation $C^{[-d,0]}(\proj A)$ is used for the subcategory of $C^b(\proj A)$ consisting of $(d+1)$-term complexes, i.e.\ complexes concentrated in degrees $-d,\ldots,0$. An object $P_\bullet \in C^{[-d,0]}(\proj A)$ is inner acyclic if it satisfies $H_i(P_\bullet)=0$ for $i \notin \{-d,0\}$. The subcategory of $C^{[-d,0]}(\proj A)$ given by inner acyclic complexes is denoted by $C_\text{\normalfont inn\,ac}^{[-d,0]}(\proj A)$. Similarly, we write $K^{[-d,0]}(\proj A)$ and $K_\text{\normalfont inn\,ac}^{[-d,0]}(\proj A)$ for the subcategories of $K^b(\proj A)$ consisting of $(d+1)$-term complexes and inner acyclic $(d+1)$-term complexes, respectively. 

We will study the subcategories
\begin{align*}
\pc_d(\mc) &\colonequals \{P_\bullet \in C_\text{\normalfont inn\,ac}^{[-d,0]}(\proj A) \mid H_0(P_\bullet) \in \mc \} \subseteq C^{[-d,0]}(\proj A) \\
\widetilde{\pc}_d(\mc) &\colonequals \{P_\bullet \in K_\text{\normalfont inn\,ac}^{[-d,0]}(\proj A) \mid H_0(P_\bullet) \in \mc \} \subseteq K^{[-d,0]}(\proj A).
\end{align*}
It should be noted that $\pc_d(\mc)$ is the additive closure of all minimal projective $d$-presentations $P_\bullet^M$ of \mbox{$M \in \mc$}, in addition to $P[d]$ for $P \in \proj A$ and all contractible complexes in $C^{[-d,0]}(\proj A)$. The subcategory $\widetilde{\pc}_d(\mc)$ is the image of $\pc_d(\mc)$ under the natural functor \mbox{$C^b(\proj A) \to K^b(\proj A)$}.

The category $C^{[-d,0]}(\proj A)$ is endowed with the structure of an exact category whose conflations are the degree-wise split short exact sequences; see \cref{ssec:exact and d-tilting} for more details. The main result of this section is the theorem below, for which we note that the notion of a $d$-cluster tilting subcategory of an exact category is given in \cref{def:d-CT exact}.

\begin{theorem}
\label{theorem:dCT exact}
The category $\pc_d(\mc)$ is a $d$-cluster tilting subcategory of $C^{[-d,0]}(\proj A)$.
\end{theorem}

By \cite[Theorem 4.14]{Jasso}, the following corollary is an immediate consequence of \cref{theorem:dCT exact}.

\begin{corollary}
\label{corollary:d-exact}
The category $\pc_d(\mc)$ carries the structure of a $d$-exact category induced by the exact structure of $C^{[-d,0]}(\proj A)$.
\end{corollary}

As intermediate steps towards proving \cref{theorem:dCT exact}, we demonstrate that statements concerning functorial finiteness and rigidity can be naturally transferred between the categories $C^{[-d,0]}(\proj A)$ and $K^{[-d,0]}(\proj A)$; see \cref{lemma: ext for complexes} and \cref{lem:lifting functorial finiteness}. To prove that $\pc_d(\mc)$ is \mbox{$d$-cluster} tilting in $C^{[-d,0]}(\proj A)$, it hence suffices to show that $\widetilde{\pc}_d(\mc)$ is $d$-cluster tilting in $K^{[-d,0]}(\proj A)$; see \cref{prop:lifting dCT}. The next result will therefore imply \cref{theorem:dCT exact}. For the definition of a $d$-cluster tilting subcategory of $K^{[-d,0]}(\proj A)$, see \cref{def: d-CT tri}.   

\begin{theorem}
\label{theorem: d-CT extriangulated}
The category $\widetilde{\pc}_d(\mc)$ is a $d$-cluster tilting subcategory of $K^{[-d,0]}(\proj A)$.
\end{theorem}

\begin{remark}\label{rem:extri in intro}
The category $K^{[-d,0]}(\proj A)$ is an example of an extriangulated category as introduced in \cite{NakaokaPalu}, and \cref{def: d-CT tri} is a special case of the notion of a $d$-cluster tilting subcategory of an extriangulated category in the sense of \cite{HerschendLiuNakaokaII}. In the statements and proofs throughout this section, we have avoided the language of extriangulated categories for the sake of simplicity and because the extriangulated setup is not important for the rest of the paper. However, everything is consistent with and might be reformulated using the extriangulated structure.

We moreover note that it does not immediately follow from \cref{theorem: d-CT extriangulated} that $\widetilde{\pc}_d(\mc)$ is a \mbox{$d$-exangulated} category in the sense of \cite{HerschendLiuNakaoka}. In particular, $\widetilde{\pc}_d(\mc)$ would need to satisfy some additional conditions; see \cite[Theorem 3.41]{HerschendLiuNakaokaII} and \cite[Theorem 3.1]{HuZhangZhou}.
\end{remark}

\subsection{Exact categories and $d$-cluster tilting}
\label{ssec:exact and d-tilting}

We start this subsection by briefly recalling the notion of an exact category in the sense of Quillen. We refer the reader to \cite{Buhler}, and the references therein, for more details.

An \textit{exact category} is an additive category $\ec$ endowed with a class of kernel-cokernel pairs, called \textit{conflations} (or \textit{admissible short exact sequences}), that are stable under isomorphisms of diagrams and satisfy the axioms below. A conflation $(i,p)$ is typically depicted as 
\[
X\overset{i}{\infl}Y\overset{p}{\defl}Z,
\]
where $i$ is called an \textit{inflation} and $p$ a \textit{deflation}. The axioms can be given as follows:

\theoremlistfix
\begin{itemize}
    \item[(Ex0)] For any $X\in\ec$, the epimorphism $X\to 0$ is a deflation.
    \item[(Ex1)] The class of deflations is closed under composition.
    \item[(Ex2)] The pullback of a deflation along any morphism exists and yields a deflation.
    \item[(Ex2$^\text{op}$)] The pushout of an inflation along any morphism exists and yields an inflation.
\end{itemize}

The structure of an exact category is self-dual, so the duals of (Ex0) and (Ex1) involving inflations also hold in any exact category. The class of conflations contains all split exact sequences and is stable under direct sums. 

Let $\ec$ be an exact category. An object $P\in \ec$ is called \textit{projective} if for any deflation $X\defl Y$ in $\ec$ the sequence of abelian groups $\Hom_\ec(P,X)\to \Hom_\ec(P,Y)\to 0$ is exact. An \textit{injective} object is defined dually. The exact category $\ec$ is said to have \textit{enough projective objects} if for every $X\in \ec$, there exists a deflation $P\defl X$ where $P$ is projective in $\ec$. The notion of \textit{enough injective objects} is defined dually.

A sequence $0\to X_0\to X_1 \to X_2\to \cdots \to X_d \to X_{d+1}\to 0$ in $\ec$ is \textit{exact} if there exist conflations $Y_i\infl X_{i+1}\defl Y_{i+1}$ for $i=0,\ldots,d-1$ with $Y_0=X_0$, $Y_{d}=X_{d+1}$ and such that the sequence is obtained by splicing these conflations. The higher extension groups $\operatorname{Ext}^i_\ec(Z,X)$ can be defined via Yoneda equivalence classes of exact sequences $0\to X\to Y_1\to \cdots \to Y_i\to Z\to 0 $
in an analogous way as for abelian categories; for details, see e.g.\  \cite[Chapter 6]{FrerickSieg}. We say that $\ec$ has \mbox{\textit{global dimension} $n$}, denoted $\operatorname{gl.dim.}\ec=n$, if $n$ is the smallest integer for which $\Ext^n_\ec(-,-)=0$. 

Our main example of interest is the category $C^{[-d,0]}(\proj A)$ of $(d+1)$-term complexes of finitely generated projective $A$-modules. It is a subcategory of the category $C^b(\proj A)$ of bounded complexes of finitely generated projective $A$-modules, which is an exact category with the degree-wise split short exact sequences being its conflations; see \cite[Lemma 9.1]{Buhler}. Since $C^{[-d,0]}(\proj A)$ is closed under extensions in $C^b(\proj A)$, it is an exact category whose conflations are the degree-wise split short exact sequences in $C^{[-d,0]}(\proj A)$; see \cite[Lemma 10.20]{Buhler}. The following proposition gathers results from \cite{BautistaSalorioZuazua} on the homological properties of  $C^{[-d,0]}(\proj A)$. For the definition of the dominant dimension of an exact category, see e.g.\ \cite[Definition 4.12]{HenrardKvammeRoosmalen}.

\begin{proposition}
\label{prop:enough proj/inj}
\theoremlistfix
\begin{enumerate}
    \item The projective-injective objects in $C^{[-d,0]}(\proj A)$ are the contractible complexes. The projective objects are the sums of contractible complexes and complexes concentrated in degree 0, while the injective objects are the sums of contractible complexes and complexes concentrated in degree $-d$.
    \phantomsection
    \label{prop:enough proj/inj1}
    \item The category $C^{[-d,0]}(\proj A)$ has enough projective and injective objects.
    \phantomsection
    \label{prop:enough proj/inj2}
    \item We have
    \[
    \operatorname{gl.dim.} C^{[-d,0]}(\proj A) \leq d \leq \operatorname{dom.dim.} C^{[-d,0]}(\proj A).
    \] 
    \phantomsection
    \label{prop:enough proj/inj3}
\end{enumerate}
\end{proposition}

\begin{proof}
Part (\ref{prop:enough proj/inj1}) is \cite[Corollary 3.8]{BautistaSalorioZuazua}. Part (\ref{prop:enough proj/inj2}) and the first inequality in (\ref{prop:enough proj/inj3}) are \cite[Propositions 3.6 and 3.7]{BautistaSalorioZuazua}.
For any integer $i$, we let $D_i(A)$ be the contractible complex 
\[
\cdots \to 0\to A\xrightarrow{1_A}A\to 0\to \cdots
\]
concentrated in degrees $i$ and $i+1$.
Looking at the injective coresolution 
\[ 0\to A\to D_{-1}(A)\to \cdots \to D_{-d}(A)\to A[d]\to 0\]
of $A$ and noting that the first $d$ terms $D_{-1}(A),\ldots,D_{-d}(A)$ are also projective, yields the second inequality in (\ref{prop:enough proj/inj3}).
\end{proof}

In the remaining part of this subsection, we review some basic results regarding $d$-cluster tilting subcategories in exact categories. Let $\ec$ be an exact category. A subcategory $\mc$ of $\ec$ is \textit{cogenerating} if, for any $X\in\ec$, there is an inflation $X\infl M$ with $M\in\mc$. The dual notion is called \textit{generating}. If $\mc$ is cogenerating and covariantly finite, then, for any $X\in\ec$, there is a left $\mc$-approximation $X \infl M$ which is an inflation. The dual statement holds for generating and contravariantly finite subcategories.

\begin{definition}
\label{def:d-CT exact}
A generating-cogenerating subcategory $\mc$ of an exact category $\ec$ is called \textit{\mbox{$d$-cluster} tilting} if it is functorially finite and satisfies
\begin{align*}
\mc  &= \{X \in \ec \mid \Ext_{\ec}^i(X,M)=0 \text{ for all $M \in \mc$ and all $i = 1, \dots, d-1$}\}\\
&= \{Y \in \ec \mid \Ext_{\ec}^i(M,Y)=0 \text{ for all $M \in \mc$ and all $i = 1, \dots, d-1$}\}.
\end{align*}
\end{definition}

The following lemma is useful for proving that a given subcategory is $d$-cluster tilting, c.f.\ \cite[Theorem 5.3]{Beligiannis}.

\begin{lemma}
\label{lemma: contravariantly finite for free}
Let $\ec$ be an exact category with a subcategory $\mc\subseteq\ec$. Assume that $\mc$ is generating-cogenerating and satisfies
\begin{align*}
\mc  &= \{X \in \ec \mid \Ext_{\ec}^i(X,M)=0 \text{ for all $M \in \mc$ and all $i = 1, \dots, d-1$}\}\\
&= \{Y \in \ec \mid \Ext_{\ec}^i(M,Y)=0 \text{ for all $M \in \mc$ and all $i = 1, \dots, d-1$}\}.
\end{align*}
Then $\mc$ is covariantly finite if and only if it is contravariantly finite. 
\end{lemma}

\begin{proof}
Assume that $\mc$ is covariantly finite, and let $X$ be any object in $\ec$. Since $\mc$ is generating, there is a sequence of conflations $X_{i+1}\infl M_i\defl X_i$ for $i=0,\ldots,d-2$, where $X_0=X$ and $M_i\in\mc$ for all $i$. Let $Y_0=X_{d-1}$. Because $\mc$ is covariantly finite and cogenerating, there are conflations $Y_{i-1}\infl N_i\defl Y_i$ for $i=1,\ldots,d-1$, where the inflations are left $\mc$-approximations. Moreover, we see that $Y_{d-1}$ belongs to $\mc$ by noting that the proof of \cite[Proposition 3.17]{Jasso} works in our generality. We let $N_d=Y_{d-1}$. Using the definition of being a left approximation, we thus get a morphism of exact sequences
\[\begin{tikzcd}
	{X_{d-1}} & {N_1} & {N_2} & \cdots & {N_{d-3}} & {N_{d-2}} & {N_{d-1}} & {N_d} \\
	{X_{d-1}} & {M_{d-2}} & {M_{d-3}} & \cdots & {M_2} & {M_1} & {M_0} & X
	\arrow[tail, from=1-1, to=1-2]
	\arrow[from=1-2, to=1-3]
	\arrow[from=1-2, to=2-2]
	\arrow[from=1-3, to=1-4]
	\arrow[from=1-3, to=2-3]
	\arrow[from=1-4, to=1-5]
	\arrow["\cdots"{description}, draw=none, from=1-4, to=2-4]
	\arrow[from=1-5, to=1-6]
	\arrow[from=1-5, to=2-5]
	\arrow[from=1-6, to=1-7]
	\arrow[from=1-6, to=2-6]
	\arrow[two heads, from=1-7, to=1-8]
	\arrow[from=1-7, to=2-7]
	\arrow[from=1-8, to=2-8]
	\arrow[Rightarrow, no head, from=2-1, to=1-1]
	\arrow[tail, from=2-1, to=2-2]
	\arrow[from=2-2, to=2-3]
	\arrow[from=2-3, to=2-4]
	\arrow[from=2-4, to=2-5]
	\arrow[from=2-5, to=2-6]
	\arrow[from=2-6, to=2-7]
	\arrow[two heads, from=2-7, to=2-8]
\end{tikzcd}\]
from which we obtain, using \cite[Lemma 10.3]{Buhler}, the exact sequence
\[
N_1\infl M_{d-2}\oplus N_2\to M_{d-3}\oplus N_3\to\cdots \to M_1\oplus N_{d-1} \to M_0\oplus N_d\defl X.
\]
As all the terms in the exact sequence except $X$ are in $\mc$ and $\Ext^{i}_\ec(\mc,\mc)=0$ for \mbox{$i=1,\dots,d-1$}, the deflation $M_0\oplus N_d\defl X$ is a right $\mc$-approximation of $X$. Therefore, the subcategory $\mc$ is contravariantly finite. The proof that contravariant finiteness implies covariant finiteness follows by dual arguments. 
\end{proof}

We also need the following reformulation of $d$-cluster tilting, which was first shown in \cite[Proposition 2.2.2]{Iyama2007a}. For a proof in arbitrary exact categories, see e.g.\ \cite[Proposition 4.4]{Higher Singularity}.
\begin{proposition}\label{Reformulation:d-CT}
 	Let $\ec$ be an exact category with a subcategory $\mc\subseteq\ec$. Then $\mc$ is \mbox{$d$-cluster} tilting if and only if \mbox{$\Ext^i_\ec(\mc,\mc)=0$} for all $i=1,\ldots,d-1$ and, for any $X\in \ec$, there exist exact sequences
  \begin{align*}
     & 0\to X\to M^1\to \cdots \to M^d\to 0 \\
     & 0\to M_d\to \cdots \to M_1\to X\to 0
  \end{align*}
  with $M_1,\dots, M_d\in \mc$ and $M^1,\dots, M^d\in \mc$.
 \end{proposition}	

\subsection{From $(d+1)$-term complexes to the homotopy category}\label{Subsection:From $(d+1)$-term complexes to its homotopy category}

Consider the exact category $C^{[-d,0]}(\proj A)$ described in 
\cref{ssec:exact and d-tilting}. Taking the ideal quotient by its projective-injectives, one obtains the homotopy category $K^{[-d,0]}(\proj A)$ of $(d+1)$-term complexes. The category $K^{[-d,0]}(\proj A)$ is no longer an exact category, but it is closed under extensions in the triangulated category $K^b(\proj A)$; see e.g.\ \cite[Example 5.19]{IyamaNakaokaPalu}. Moreover, the image in $K^{[-d,0]}(\proj A)$ of a conflation in $C^{[-d,0]}(\proj A)$ is a distinguished triangle in $K^b(\proj A)$.

In this subsection we define what it means for a subcategory of $K^{[-d,0]}(\proj A)$ to be $d$-cluster tilting and explain in which way it corresponds to \mbox{$d$-cluster} tilting in $C^{[-d,0]}(\proj A)$. Note that the definition below is a special case of the definition of a $d$-cluster tilting subcategory of an extriangulated category with enough projectives and \mbox{injectives as introduced in \cite[Definition 3.21]{HerschendLiuNakaokaII}.}

\begin{definition}\label{def: d-CT tri}
A functorially finite subcategory $\mathcal{T}$ of $K^{[-d,0]}(\proj A)$ is \textit{$d$-cluster tilting} if
\begin{align*}
\mathcal{T} &=\{X_\bullet\in K^{[-d,0]}(\proj A)\;|\; \Hom_{K^{b}(\proj A)}(X_\bullet,T_\bullet[i])=0 
\text{ for } i=1,\ldots,d-1 \text{ and } T_\bullet\in\mathcal{T} \} \\
  &=\{Y_\bullet\in K^{[-d,0]}(\proj A)\;|\; \Hom_{K^b(\proj A)}(T_\bullet,Y_\bullet[i])=0 \text{ for } i=1,\ldots,d-1 \text{ and } T_\bullet\in\mathcal{T} \}.
\end{align*}
\end{definition}

The goal of this subsection is to prove the following result.

\begin{proposition}\label{prop:lifting dCT}
  Let $\widetilde{\mathcal{N}}$ be a subcategory of $K^{[-d,0]}(\proj A)$ and let $\mathcal{N}$ be its preimage under the canonical projection $C^{[-d,0]}(\proj A)\to K^{[-d,0]}(\proj A)$. Then the following statements are equivalent:
  \theoremlistfix
\begin{enumerate}
       \item The category $\mathcal{N}$ is a $d$-cluster tilting subcategory of $C^{[-d,0]}(\proj A)$.
       \item The category $\widetilde{\mathcal{N}}$ is a $d$-cluster tilting subcategory of $K^{[-d,0]}(\proj A)$.
   \end{enumerate}
\end{proposition}

As a first step towards \Cref{prop:lifting dCT}, we compare extension groups in $C^{[-d,0]}(\proj A)$ with Hom-spaces in $K^{[-d,0]}(\proj A)$.

\begin{lemma}
\label{lemma: ext for complexes}
For any $X_\bullet,Y_\bullet\in C^{[-d,0]}(\proj A)$ and any $i>0$, we have
\[
\Ext^i_{C^{[-d,0]}(\proj A)}(X_\bullet,Y_\bullet) \cong \Hom_{K^{b}(\proj A)}(X_\bullet,Y_\bullet[i]).
\]
\end{lemma}

\begin{proof}
We first consider the case $i=1$. For any $X_\bullet,Y_\bullet\in C^{[-d,0]}(\proj A)$, we have
\[
\Ext^1_{C^{[-d,0]}(\proj A)}(X_\bullet,Y_\bullet)\cong \Ext^1_{C^b(\proj A)}(X_\bullet,Y_\bullet)\cong \Hom_{K^b(\proj A)}(X_\bullet,Y_\bullet[1]),
\]
where the first isomorphism holds since $C^{[-d,0]}(\proj A)$ is extension-closed in $C^b(\proj A)$ and the second one because $C^b(\proj A)$ is Frobenius exact with $K^b(\proj A)$ as its stable category; see e.g.\ \cite[Lemma 4.1.1 and Lemma 3.3.3]{Krause22}. 

Since $C^{[-d,0]}(\proj A)$ has enough projectives by \cref{prop:enough proj/inj} (\ref{prop:enough proj/inj2}), there exist conflations 
\[
\Omega^{j+1}X_\bullet\infl P^{j}_\bullet\defl \Omega^{j} X_\bullet
\]
in $C^{[-d,0]}(\proj A)$ with $P^{j}_\bullet$ projective for any $j \geq 0$. This gives isomorphisms
\begin{eqnarray*}
 \Ext^i_{C^{[-d,0]}(\proj A)}(X_\bullet,Y_\bullet) & \cong & \Ext^1_{C^{[-d,0]}(\proj A)}(\Omega^{i-1}X_\bullet,Y_\bullet) \\
 & \cong & \Hom_{K^b(\proj A)}(\Omega^{i-1}X_\bullet,Y_\bullet[1])
\end{eqnarray*}
for each $i>0$. The conflations also induce distinguished triangles $$\Omega^{j+1}X_\bullet\to P^{j}_\bullet\to \Omega^{j} X_\bullet\to (\Omega^{j+1}X_\bullet)[1]$$ in $K^b(\proj A)$. Applying the functor $\Hom_{K^b(\proj A)}(-,Y_\bullet[k])$ to these triangles gives long exact sequences
\begin{multline*}
   \cdots \to \Hom_{K^b(\proj A)}(P^{j}_\bullet,Y_\bullet[k])\to \Hom_{K^b(\proj A)}(\Omega^{j+1}X_\bullet,Y_\bullet[k])\\ \to \Hom_{K^b(\proj A)}(\Omega^{j}X_\bullet[-1],Y_\bullet[k]) \to \Hom_{K^b(\proj A)}(P^{j}_\bullet[-1],Y_\bullet[k]) \to \cdots
\end{multline*}
Since $P^{j}_\bullet$ is isomorphic in $K^b(\proj A)$ to an object in $\add (A)$ by \cref{prop:enough proj/inj} (\ref{prop:enough proj/inj1}), we have
\[ 
\Hom_{K^b(\proj A)}(P^{j}_\bullet,Y_\bullet[k])=0=\Hom_{K^b(\proj A)}(P^{j}_\bullet[-1],Y_\bullet[k])
\]
for $k>0$, and so 
\begin{eqnarray*}
 \Hom_{K^b(\proj A)}(\Omega^{j+1}X_\bullet,Y_\bullet[k]) & \cong & \Hom_{K^b(\proj A)}(\Omega^{j}X_\bullet[-1],Y_\bullet[k]) \\
 & \cong & \Hom_{K^b(\proj A)}(\Omega^{j}X_\bullet,Y_\bullet[k+1]).
\end{eqnarray*}
From this, we deduce the isomorphisms
\begin{eqnarray*}
 \Hom_{K^b(\proj A)}(\Omega^{i-1}X_\bullet,Y_\bullet[1]) & \cong & \Hom_{K^b(\proj A)}(\Omega^{i-2}X_\bullet,Y_\bullet[2])  \\ & \smash{\vdots} & \\
 & \cong & \Hom_{K^b(\proj A)}(X_\bullet,Y_\bullet[i]),
\end{eqnarray*}
which proves the result.
\end{proof}

Next we compare functorial finiteness in $C^{[-d,0]}(\proj A)$ and in $K^{[-d,0]}(\proj A)$.

\begin{lemma}\label{lem:lifting functorial finiteness}
    Let $\xc$ be a subcategory of $C^{[-d,0]}(\proj A)$ containing all projective-injective objects and denote its image under the natural functor $C^{[-d,0]}(\proj A) \to K^{[-d,0]}(\proj A)$ \mbox{by $\widetilde{\xc}$}. The following statements hold:
    \theoremlistfix
\begin{enumerate}
    \item $\xc$ is covariantly finite if and only if  $\widetilde{\xc}$ is covariantly finite.
    \phantomsection
    \label{lem:lifting functorial finiteness1}
    \item $\xc$ is contravariantly finite if and only if $\widetilde{\xc}$ is contravariantly finite.
    \phantomsection
    \label{lem:lifting functorial finiteness2}
\end{enumerate}
\end{lemma}

\begin{proof}
The indecomposable projective-injective objects in $C^{[-d,0]}(\proj A)$ are isomorphic to the complexes of the form $\cdots \to 0\to P \xrightarrow{1_P} P\to 0 \to \cdots$ with $P \in \proj A$ indecomposable by \cref{prop:enough proj/inj} (\ref{prop:enough proj/inj1}). In particular, there are finitely many of them up to isomorphism. This implies that the subcategory of projective-injective objects is functorially finite. The claim now follows by similar arguments as in \cite[Lemma 5.2]{CortesCriveiSaorin}.
\end{proof}

We are now ready to prove the main result of this subsection.

\begin{proof}[Proof of \Cref{prop:lifting dCT}]
Since $C^{[-d,0]}(\proj A)$ has enough projectives and injectives by \cref{prop:enough proj/inj} \eqref{prop:enough proj/inj2}, the subcategory $\mathcal{N}$ is automatically generating-cogenerating if the other conditions for being $d$-cluster tilting hold. The claim hence follows from \cref{lemma: ext for complexes} and \cref{lem:lifting functorial finiteness}.
\end{proof}

We finish this subsection by deducing an analogue of \cref{lemma: contravariantly finite for free} for $K^{[-d,0]}(\proj A)$. 

\begin{lemma}
\label{lemma: contravariantly finite for free ext}
Let $\tc$ be a subcategory of $K^{[-d,0]}(\proj A)$.
Assume that $\tc$ satisfies
\begin{align*}
\mathcal{T} &=\{X_\bullet\in K^{[-d,0]}(\proj A)\;|\; \Hom_{K^{b}(\proj A)}(X_\bullet,T_\bullet[i])=0 
\text{ for } i=1,\ldots,d-1 \text{ and } T_\bullet\in\mathcal{T} \} \\
  &=\{Y_\bullet\in K^{[-d,0]}(\proj A)\;|\; \Hom_{K^b(\proj A)}(T_\bullet,Y_\bullet[i])=0 \text{ for } i=1,\ldots,d-1 \text{ and } T_\bullet\in\mathcal{T} \}.
\end{align*}

Then $\tc$ is covariantly finite if and only if it is contravariantly finite.
\end{lemma}

\begin{proof}
    Let $\xc \subseteq C^{[-d,0]}(\proj A)$ denote the preimage of the subcategory $\tc$ under the natural functor $C^{[-d,0]}(\proj A) \to K^{[-d,0]}(\proj A)$. Combining \cref{lemma: ext for complexes} with the assumptions on $\tc$, one obtains
    \begin{align*}
    \xc  &= \{Y_\bullet \in C^{[-d,0]}(\proj A) \mid \Ext_{C^{[-d,0]}(\proj A)}^i(Y_\bullet,\xc)=0 \text{ for all $i = 1, \dots, d-1$}\}\\
    &= \{Z_\bullet \in C^{[-d,0]}(\proj A) \mid \Ext_{C^{[-d,0]}(\proj A)}^i(\xc,Z_\bullet)=0 \text{ for all $i = 1, \dots, d-1$}\}.
    \end{align*}
    Since $C^b(\proj A)$ has enough projectives and injectives by \cref{prop:enough proj/inj} (\ref{prop:enough proj/inj2}), it moreover follows that $\xc$ is generating-cogenerating. Hence, the subcategory $\xc$ is covariantly finite if and only if it is contravariantly finite by \cref{lemma: contravariantly finite for free}. By \cref{lem:lifting functorial finiteness}, the same statement therefore also holds for $\tc$.
\end{proof}

\subsection{From $(d+1)$-term complexes to $A$-modules}\label{subsec: (d+1)-term}

In this subsection we develop a series of technical results relating the categories $K^{[-d,0]}(\proj A)$ and $\modA$ that will be instrumental in the proofs of our main results.

\begin{lemma}
\label{lemma:lifts}
Consider $X_\bullet\in K^{[-d,0]}(\proj A)$ and $Y_\bullet\in K_\text{\normalfont inn\,ac}^{[-d,0]}(\proj A)$.
Then any morphism \mbox{$H_0(X_\bullet)\to H_0(Y_\bullet)$} lifts to a morphism $X_\bullet\to Y_\bullet$.
Moreover, such a lift is unique up to adding a morphism that factors through $\add (A[d])$.
\end{lemma}

\begin{proof}
Consider a morphism $a\colon H_0(X_\bullet)\to H_0(Y_\bullet)$. Since $Y_\bullet$ is inner acyclic, there is a distinguished triangle 
\[
H_{-d}(Y_\bullet)[d]\to Y_\bullet\to H_0(Y_\bullet)\to H_{-d}(Y_\bullet)[d+1]
\]
in the bounded derived category $D^b(A)$. The complex $X_\bullet$ is a bounded complex of projectives and has component 0 in degree $(-d-1)$, which implies that 
\[
\Hom_{D^b(A)}(X_\bullet,H_{-d}(Y_\bullet)[d+1]) \cong \Hom_{K^b(\modA)}(X_\bullet,H_{-d}(Y_\bullet)[d+1])=0.
\]
Thus, from the long exact sequence obtained by applying the functor $\Hom_{D^b(A)}(X_\bullet,-)$ to the triangle above, we get an epimorphism 
\[
\Hom_{D^b(A)}(X_\bullet,Y_\bullet) \rightarrow \Hom_{D^b(A)}(X_\bullet, H_0(Y_\bullet)).
\]
Therefore, we have a morphism $f \colon X_\bullet\xrightarrow{}Y_\bullet$ making the diagram

\begin{center}
\begin{tikzpicture}[scale = 1.4]
\node (X-left) at (1,1) {$X_\bullet$};
\node (HX-left) at (2,1) {$H_0(X_\bullet)$};
\node[anchor=east] (HdYd-left) at (.5,0) {$H_{-d}(Y_\bullet)[d]$};
\node (Y-left) at (1,0) {$Y_\bullet$};
\node (HY-left) at (2,0) {$H_0(Y_\bullet)$};
\node[anchor=west]  (HdYd+left) at (2.8,0) {$H_{-d}(Y_\bullet)[d+1]$};

\draw[->] (X-left) -- (HX-left);
\draw[dashed, ->] (X-left) to node[left]{$f$} (Y-left);
\draw[->] (HdYd-left) -- (Y-left);
\draw[->] (Y-left) -- (HY-left);
\draw[->] (HX-left) to node[right]{$a$} (HY-left);
\draw[->] (HY-left) to (HdYd+left);

\end{tikzpicture}
\end{center}
\noindent
commute. 

Assume now that $f_1$ and $f_2$ are two such lifts of $a$, and write $g$ for their difference.
By definition, the composition $X_\bullet\xrightarrow{g}Y_\bullet\to H_0(Y_\bullet)$ vanishes, which implies the existence of a morphism \mbox{$h \colon X_\bullet\xrightarrow{}H_{-d}(Y_\bullet)[d]$} making the diagram 

 \begin{center}
 \begin{tikzpicture}[scale = 1.4]

\begin{scope}[xshift = 6.2cm]
\node (X-right) at (1,1) {$X_\bullet$};
\node (HX-right) at (2,1) {$H_0(X_\bullet)$};
\node[anchor=east] (HdYd-right) at (.3,0) {$H_{-d}(Y_\bullet)[d]$};
\node (Y-right) at (1,0) {$Y_\bullet$};
\node (HY-right) at (2,0) {$H_0(Y_\bullet)$};
\node[anchor=west]  (HdYd+right) at (2.7,0){$H_{-d}(Y_\bullet)[d+1]$};

\draw[->] (X-right) -- (HX-right);
\draw[->] (X-right) to node[left]{$g$} (Y-right);
\draw[->] (HdYd-right) -- (Y-right);
\draw[->] (Y-right) -- (HY-right);
\draw[->] (HX-right) to node[right]{$0$} (HY-right);
\draw[->] (HY-right) to (HdYd+right);
\draw[->, dashed] (X-right) to node[above]{$h$} (HdYd-right);
\end{scope}
\end{tikzpicture}
\end{center}

\noindent commute. Choose an epimorphism $P\defl H_{-d}(Y_\bullet)$ in $\modA$ with $P$ projective. Because $X_{-d}$ is projective, the morphism $h$ factors through $P[d]\to H_{-d}(Y_\bullet)[d]$. Hence, the morphism $g$ factors through $P[d]$ as well. 
\end{proof}

The following result is a consequence of \Cref{lemma:lifts}, and shows that we can recover $\modA$ as an ideal quotient of $K_\text{\normalfont inn\,ac}^{[-d,0]}(\proj A)$.

\begin{proposition}\label{Proposition:EquivAdditivQuot}
The $0$-th homology functor
\[
H_0(-)\colon K_\text{\normalfont inn\,ac}^{[-d,0]}(\proj A)\to \modA
\]
induces an equivalence $K_\text{\normalfont inn\,ac}^{[-d,0]}(\proj A)/\operatorname{add}(A[d])\cong \modA$.
\end{proposition} 

\begin{proof}
Since $H_0(-)$ sends $A[d]$ to $0$, it induces a functor
\[
K_\text{\normalfont inn\,ac}^{[-d,0]}(\proj A)/\operatorname{add}(A[d])\to \modA.
\]
The functor is clearly dense, and by \Cref{lemma:lifts} it must be full and faithful. This proves the claim.
\end{proof}

\begin{remark}
\cref{Proposition:EquivAdditivQuot} can also be deduced from \cite[Proposition 3.1]{Gupta}, where it is shown that the truncation functor
    \[
    \tau_{\geq -d+1}\colon K^{[-d,0]}(\proj A)\to D^b(A) 
    \]
    induces an equivalence
    \[
    K^{[-d,0]}(\proj A)/\operatorname{add}(A[d])\cong D^{[-d+1,0]}(A),
    \]
    where $D^{[-d+1,0]}(A)$ denotes the subcategory of $D^b(A)$ consisting of complexes with non-zero homology concentrated in degrees $-d+1,\dots ,0$. Restricting this to $K_\text{\normalfont inn\,ac}^{[-d,0]}(\proj A)$ gives an equivalence between $K_\text{\normalfont inn\,ac}^{[-d,0]}(\proj A)/\operatorname{add}(A[d])$ and $\modA$.
\end{remark}

In the next lemma, we describe certain $\Ext$-groups in terms of $\Hom$-sets in $K^b(\proj A)$.

\begin{lemma}
\label{lemma: inner extensions}
Let $X_\bullet,Y_\bullet\in K^{[-d,0]}_\text{\normalfont inn\,ac}(\proj A)$. For any $i=1,\ldots, d-1$, we have
\[
\Hom_{K^{b}(\proj A)}(X_\bullet,Y_\bullet[i]) \cong\Ext^i_A(H_0(X_\bullet),H_0(Y_\bullet)).
\]
\end{lemma}

\begin{proof}
Since $X_\bullet$ and $Y_\bullet$ are inner acyclic, there are distinguished triangles
 \begin{align*}
     H_{-d}(X_\bullet)[d] &\to X_\bullet\to H_0(X_\bullet)\to H_{-d}(X_\bullet)[d+1] \\
     H_{-d}(Y_\bullet)[d] &\to Y_\bullet\to H_0(Y_\bullet)\to H_{-d}(Y_\bullet)[d+1]
 \end{align*}
in the bounded derived category ${D}^b(A)$. Applying $\Hom_{D^b(A)}(X_\bullet,-)$ to the second triangle and $\Hom_{D^b(A)}(-,H_0(Y_\bullet))$ to the first triangle gives 
\begin{align*}
    \Hom_{D^b(A)}(X_\bullet,Y_\bullet[i]) &\cong \Hom_{D^b(A)}(X_\bullet,H_0(Y_\bullet)[i])\\
    \Hom_{D^b(A)}(X_\bullet,H_0(Y_\bullet)[i]) &\cong \Hom_{D^b(A)}(H_0(X_\bullet),H_0(Y_\bullet)[i])
\end{align*}
for $i=1,\ldots,d-1$. This implies that
\begin{eqnarray*}
 \Hom_{K^b(\proj A)}(X_\bullet,Y_\bullet[i]) & \cong & \Hom_{D^b(A)}(X_\bullet,Y_\bullet[i]) \\
 & \cong & \Hom_{D^b(A)}\left(X_\bullet,H_0(Y_\bullet)[i]\right) \\
 & \cong & \Hom_{D^b(A)}\left(H_0(X_\bullet),H_0(Y_\bullet)[i]\right) \\
 & \cong & \Ext^i_A(H_0(X_\bullet),H_0(Y_\bullet))
\end{eqnarray*}
for any $i=1,\ldots,d-1$.
\end{proof}

\subsection{Proofs of main results}
\label{subsec: pf of main results}

Throughout this subsection, we assume that $\mc$ is a $d$-cluster tilting subcategory of $\mod A$. Our aim is to prove \cref{theorem:dCT exact}. 

Recall that
\begin{align*}
\pc_d(\mc) & = \{P_\bullet \in C_{\text{inn\,ac}}^{[-d,0]}(\proj A) \mid H_0(P_\bullet) \in \mc \} \subseteq C^{[-d,0]}(\proj A) \\
\widetilde{\pc}_d(\mc) & = \{P_\bullet \in K_\text{inn\,ac}^{[-d,0]}(\proj A) \mid H_0(P_\bullet) \in \mc \} \subseteq K^{[-d,0]}(\proj A).
\end{align*}
Our first main aim is to establish \cref{theorem: d-CT extriangulated}, which states that $\widetilde{\pc}_d(\mc)$ is $d$-cluster tilting in $K^{[-d, 0]}(\proj A)$. By \cref{prop:lifting dCT}, this will automatically imply that $\pc_d(\mc)$ is $d$-cluster tilting in $C^{[-d, 0]}(\proj A)$. In order to prove \cref{theorem: d-CT extriangulated}, we first present two lemmas concerning rigidity properties of $\widetilde{\pc}_d(\mc)$ and a lemma concerning its functorial finiteness.

\begin{lemma}
\label{lemma: Pd(M) d-rigid}
Let $X_\bullet,Y_\bullet\in \widetilde{\pc}_d(\mc)$. Then \(\Hom_{K^b(\proj A)}(X_\bullet,Y_\bullet[i])=0\) for \( i=1,\ldots,d-1\).
\end{lemma}

\begin{proof}
Let $X_\bullet,Y_\bullet\in\widetilde{\pc}_d(\mc)$.
Since $X_\bullet$ and $Y_\bullet$ are inner acyclic, \cref{lemma: inner extensions} yields, for each $i=1,\ldots, d-1$, the isomorphism \mbox{$\Hom_{K^b(\proj A)}(X_\bullet,Y_\bullet[i])\cong  \Ext^i_A(H_0(X_\bullet),H_0(Y_\bullet))$}. The latter is zero because $H_0(X_\bullet)$ and $H_0(Y_\bullet)$ belong to $\mathcal{M}$.
\end{proof}

\begin{lemma}
\label{lemma: Pd(M) right d-orthogonal}
The category $\widetilde{\pc}_d(\mc)$ satisfies
\begin{align*}
\widetilde{\pc}_d(\mc) &=\{X_\bullet\in K^{[-d,0]}(\proj A)\;|\; \Hom_{K^{b}(\proj A)}(\widetilde{\pc}_d(\mc),X_\bullet[i])=0 \text{ for } i=1,\ldots,d-1 \} \\
  &=\{Y_\bullet\in K^{[-d,0]}(\proj A)\;|\; \Hom_{K^b(\proj A)}(Y_\bullet,\widetilde{\pc}_d(\mc)[i])=0 \text{ for } i=1,\ldots,d-1 \}.
\end{align*}
\end{lemma}

\begin{proof}
The inclusion of $\widetilde{\pc}_d(\mc)$ into each of the subcategories appearing in the statement follows from \cref{lemma: Pd(M) d-rigid}. 

In order to obtain the first equality, assume that an object $X_\bullet\in K^{[-d,0]}(\proj A)$ satisfies \mbox{$\Hom_{K^b(\proj A)}(P_\bullet,X_\bullet[i])=0$} for {$i=1,\ldots,d-1$} and for any $P_\bullet\in\widetilde{\pc}_d(\mc)$. We first note that $X_\bullet$ is then inner acyclic. This is because $A[d]$ belongs to $\widetilde{\pc}_d(\mc)$, and we thus have the isomorphisms
\[
0=\Hom_{K^b(\proj A)}(A[d],X_\bullet[i]) \cong  \operatorname{Hom}_{K^b(\proj A)}(A,X_\bullet[i-d]) \cong  H_{i-d}(X_\bullet)
\]
for $i=1,\ldots,d-1$. As a next step, we prove that $H_0(X_\bullet)$ belongs to $\mc$. For this, consider $M\in\mc$ and let $P_\bullet^M$ be a projective $d$-presentation of $M$. By definition, the object $P_\bullet^M$ belongs to $\widetilde{\pc}_d(\mc)$. This yields $\Hom_{K^b(\proj A)}(P_\bullet^M,X_\bullet[i])=0$ for $i=1,\ldots,d-1$. By \cref{lemma: inner extensions}, this implies that $\Ext^i_A(M,H_0(X_\bullet))=0$ for $i=1,\ldots,d-1$. Therefore, we have $H_0(X_\bullet) \in \mc$ and $X_\bullet \in \widetilde{\pc}_d(\mc)$.

For the second equality, let $Y_\bullet\in K^{[-d,0]}(\proj A)$ be such that \mbox{$\Hom_{K^b(\proj A)}(Y_\bullet,Q_\bullet[i])=0$} for $i=1,\ldots,d-1$ and for any $Q_\bullet\in\widetilde{\pc}_d(\mc)$. To see that $Y_\bullet$ is then inner acyclic, let $P_\bullet$ be a projective $d$-presentation of the injective cogenerator $DA$. We have
\[
 \Hom_{K^b(\proj A)}(Y_\bullet,P_\bullet[i])\cong \Hom_{D^b(A)}(Y_\bullet,P_\bullet[i]) \cong \Hom_{D^b(A)}(Y_\bullet,DA[i])\cong DH_{-i}(Y_\bullet)
\]
for $i=1,\ldots,d-1$, where the second isomorphism is obtained by applying $\Hom_{D^b(A)}(Y_\bullet,-)$ to the distinguished triangle
\[
H_{-d}(P_\bullet)[d] \to P_\bullet \to DA \to H_{-d}(P_\bullet)[d+1]
\]
in $D^b(A)$. Since $P_\bullet$ belongs to $\widetilde{\pc}_d(\mc)$, this yields $H_{-i}(Y_\bullet)=0$ for $i=1,\ldots, d-1$. Analogous arguments as for the first equality now show that $H_0(Y_\bullet) \in \mc$, which implies $Y_\bullet \in \widetilde{\pc}_d(\mc)$.
\end{proof}

Before proving \cref{theorem: d-CT extriangulated}, we deal with the functorial finiteness of $\widetilde{\pc}_d(\mc)$. 

\begin{lemma}
\label{lemma: Pd(M) cov finite}
The subcategory $\widetilde{\pc}_d(\mc)$ is covariantly finite in $K^{[-d,0]}(\proj A)$.
\end{lemma}

\begin{proof}
Let $P_\bullet \in K^{[-d,0]}(\proj A)$ and consider a left $\mc$-approximation $H_0(P_\bullet)\to M$ of $H_0(P_\bullet)$.
Let $P_\bullet^M$ be a projective $d$-resolution of $M$. Then $P_\bullet^M$ belongs to $K_\text{\normalfont inn\,ac}^{[-d,0]}(\proj A)$. By \cref{lemma:lifts}, the left $\mc$-approximation of $H_0(P_\bullet)$ thus lifts to a morphism $P_\bullet\to P_\bullet^M$. Consider additionally the morphism $P_\bullet \to P^{-d}[d]$ given by the identity in degree $-d$.

We claim that the induced morphism $P_\bullet\to P_\bullet^M\oplus P^{-d}[d]$ is a left $\widetilde{\pc}_d(\mc)$-approximation of $P_\bullet$. To prove this, let $P_\bullet \to Q_\bullet$ be any morphism with $Q_\bullet\in\widetilde{\pc}_d(\mc)$. Since $H_0(Q_\bullet)$ belongs to $\mc$, the induced morphism $H_0(P_\bullet)\to H_0(Q_\bullet)$ factors through the approximation \mbox{$H_0(P_\bullet) \to M$}. Again using \cref{lemma:lifts}, we obtain a morphism $P_\bullet^M\to Q_\bullet$ such that the composition $P_\bullet\to P_\bullet^M\to Q_\bullet$ equals the given morphism $P_\bullet\to Q_\bullet$ up to a morphism factoring through an object of $\add (A[d])$, hence up to a morphism factoring through \mbox{$P_\bullet\to P^{-d}[d]$}. This shows that $P_\bullet\to P_\bullet^M\oplus P^{-d}[d]$ is indeed a left $\widetilde{\pc}_d(\mc)$-approximation.
\end{proof}

We are now ready to prove \cref{theorem: d-CT extriangulated} and \cref{theorem:dCT exact}. 

\begin{proof}[Proof of \cref{theorem: d-CT extriangulated}] 
By \cref{lemma: Pd(M) right d-orthogonal}, we know that $\widetilde{\pc}_d(\mc)$ satisfies
\begin{align*}
\widetilde{\pc}_d(\mc)  &= \{X_\bullet \in K^{[-d,0]}(\proj A) \mid \Hom_{K^b(\proj A)}(X_\bullet,\widetilde{\pc}_d(\mc)[i])=0 \text{ for all $i = 1, \dots, d-1$}\}\\
&= \{Y_\bullet \in K^{[-d,0]}(\proj A) \mid \Hom_{K^b(\proj A)}(\widetilde{\pc}_d(\mc),Y_\bullet[i])=0 \text{ for all $i = 1, \dots, d-1$}\}.
\end{align*}
It follows from \cref{lemma: Pd(M) cov finite,lemma: contravariantly finite for free ext} that $\widetilde{\pc}_d(\mc)$ is functorially finite in $K^{[-d,0]}(\proj A)$, allowing us to conclude that it is a $d$-cluster tilting subcategory.
\end{proof}

\begin{proof}[Proof of \cref{theorem:dCT exact}]
    \cref{theorem:dCT exact} follows by combining \cref{prop:lifting dCT} and \cref{theorem: d-CT extriangulated}.
\end{proof}

\section{From \texorpdfstring{$d$}{d}-torsion classes to silting complexes}\label{sec:silting}
The main goal of this section is to prove the following result, which is \cref{Theorem:Silting} from the introduction. 

\begin{theorem}\label{thm:silting}
Let $\mc$ be a $d$-cluster tilting subcategory of $\modA$ and consider a functorially finite $d$-torsion class $\uc$ in $\mc$. The complex $P_\bullet^{(M_\uc,P_\uc)}$ is silting in $K^b(\proj A)$.
\end{theorem}
Recall that we here use the notation $(M_\uc,P_\uc)$ for the basic $\tau_d$-rigid pair in $\mc$ associated to $\uc \subseteq \mc$ as in \cref{thm:maximaltaunrigid} and $
P_\bullet^{(M_\uc,P_\uc)}$ for the complex $P_\bullet^{M_\uc} \oplus P_\uc[d],
$
where $P_\bullet^{M_\uc}$ is the minimal projective $d$-presentation of $M_\uc$. As a consequence of the theorem above, we see in \cref{cor: maximality} that the $\tau_d$-rigid pair $(M_\uc, P_\uc)$ is maximal,
finishing the proof of \cref{MainResultIntro} from the introduction.

\begin{remark}
    In the proof of \cref{thm:silting}, we will use that the $\tau_d$-rigid pair $(M_\uc,P_\uc)$ arises from a functorially finite $d$-torsion class $\uc$. In the special case where $A$ is a homogeneous Nakayama algebra, it is shown in \cite[Theorem 6.7]{RundsveenVaso} that this assumption can be omitted, meaning that $P_\bullet^{M}\oplus P[d]$ is silting for any basic maximal $\tau_d$-rigid pair $(M,P)$ satisfying that $|M| + |P| =|A|$.
\end{remark}

Note that $P_\bullet^{(M_\uc,P_\uc)}$ is an inner acyclic $(d+1)$-term presilting complex in $K^b(\proj A)$ by \cref{prop:presilting}. Moreover, it has $|A|$ isomorphism classes of indecomposable direct summands by \cref{thm:maximaltaunrigid}. However, as pointed out in \cref{rmk:howmanyisenough}, this is not enough to conclude that $P_\bullet^{(M_\uc,P_\uc)}$ is silting. The difficulty lies in showing that $A$ is in the thick closure of $P_\bullet^{(M_\uc,P_\uc)}$ in $K^b(\proj A)$.

To prove \Cref{thm:silting}, we compare $d$-torsion classes in $\mc$ with $d$-torsion classes in 
\begin{align*}
\pc_d(\mc) & = \{P_\bullet \in C_\text{\normalfont inn\,ac}^{[-d,0]}(\proj A) \mid H_0(P_\bullet) \in \mc \},
\end{align*}
which is a $d$-cluster tilting subcategory of the Krull--Schmidt exact category $C^{[-d,0]}(\proj A)$ by \cref{theorem:dCT exact}. In order to do this, we first need to introduce $d$-torsion classes in \mbox{$d$-cluster} tilting subcategories of Krull--Schmidt exact categories. This is done in \cref{ssec:d-torsion in exact setup}. In \cref{subsec:cores exact} we develop results regarding coresolutions by $d$-torsion classes in the exact setup. The aim of \cref{subsec:bijection} is to establish a bijection between $d$-torsion classes in $\mc$ and faithful $d$-torsion classes in $\pc_d(\mc)$; see \cref{thm: dtorsion to faithfuldtorsion}. Finally, we give the proof of \cref{thm:silting} in \cref{subsec: silting proof} and consider consequences and examples.

\subsection{$d$-torsion classes in the exact setup}\label{ssec:d-torsion in exact setup}
While $d$-torsion classes in $d$-abelian categories are introduced in \cite{Jorgensen}, the analogous notion has so far not been considered within the more general setup of $d$-exact categories \cite{Jasso}.
The main goal of this subsection is to generalise the definition of a $d$-torsion class to this setup, which is done in \cref{def: d-torsion exact setup}. This will be an important ingredient in the proof of \cref{thm:silting}.

\begin{remark}
It is shown in \cite[Theorem 4.14]{Jasso} that any $d$-cluster tilting subcategory of a weakly idempotent complete exact category carries the structure of a $d$-exact category. Furthermore, any such $d$-exact category embeds into a uniquely determined weakly idempotent complete exact category as a \mbox{$d$-cluster} tilting subcategory \cite[Theorem A]{Kvamme25}. Instead of introducing $d$-torsion classes in $d$-exact categories, we hence equivalently choose to study $d$-torsion classes in $d$-cluster tilting subcategories of exact categories.
\end{remark}

We start by recalling some general notions related to exact categories. A morphism $f\colon X\to Y$ in an exact category $\ec$ is called \textit{admissible} if it can be written as a composite \mbox{$X\overset{}{\defl}Z\overset{}{\infl}Y$} where $X\overset{}{\defl}Z$ is a deflation and $Z\overset{}{\infl}Y$ is an inflation; see \cite[Section 8]{Buhler}. We call $Z$ the \textit{image} of $f$, and write $\operatorname{Im}f\colonequals Z$. Note that the image is well-defined up to isomorphism, since if $f$ is equal to another composite $X\overset{}{\defl}Z'\overset{}{\infl}Y$ of a deflation $X\overset{}{\defl}Z'$ and an inflation $Z'\overset{}{\infl}Y$, then there exists a unique isomorphism  $Z\cong Z'$ making the diagram 
\[
\begin{tikzcd}[]
	 X \arrow[r, two heads] \arrow[d, equal] & Z \arrow[r, tail] \arrow[d,"\cong"] & Y \arrow[d,equal] \\
	 X \arrow[r, two heads] & Z' \arrow[r, tail] & Y
	\end{tikzcd}
\]
commute; see \cite[Lemma 8.4]{Buhler}. Note that any admissible morphism $f\colon X\to Y$ in an exact category $\ec$ has a kernel and a cokernel, which coincide with the kernel of $X\overset{}{\defl}\operatorname{Im}f$ and cokernel of $\operatorname{Im}f\overset{}{\infl}Y$, respectively. We say that a complex
\[
\dots \to X_{i-1}\xrightarrow{f}X_{i}\xrightarrow{g}X_{i+1}\to \dots
\]
in $\ec$ is \textit{exact} at $X_i$ if $f$ and $g$ are admissible morphisms and the sequence
\(\operatorname{Im}f \infl X_i\defl \operatorname{Im}g \)
is a conflation in $\ec$. In particular, a complex $0\to X\to Y_1\to \cdots \to Y_i\to Z\to 0$ is exact as defined in \Cref{ssec:exact and d-tilting} if and only if it is exact at all objects $X,Y_1,\dots,Y_i,Z$.

Consider next a $d$-cluster tilting subcategory $\mc$ of an exact category $\ec$ in the sense of \cref{def:d-CT exact}. We now give an overview of some notions and results regarding sequences in $\mc$ that are necessary for defining $d$-torsion classes in this setup. A complex 
\[
X_0\xrightarrow{f_0}X_1 \xrightarrow{f_1} X_2 \xrightarrow{f_2} \cdots \xrightarrow{f_{d-1}} X_{d} \xrightarrow{f_d} X_{d+1} \to 0
\]
in $\mc$ is called a \textit{right $d$-exact sequence} if it is exact at $X_1,\dots, X_{d+1}$. In this case, the sequence
\begin{equation}\label{eq:d-cokernel exact}
X_1 \xrightarrow{f_1} X_2 \xrightarrow{f_2} \cdots \xrightarrow{f_{d-1}} X_{d} \xrightarrow{f_d} X_{d+1} \to 0
\end{equation}
is called a \textit{$d$-cokernel} of $f_0$ in $\mc$. Note that the morphisms $f_0,\dots,f_{d-1}$ are then admissible and $f_d$ is a deflation. The notions of \textit{left $d$-exact sequences} and \textit{$d$-kernels} are defined dually. Applying $\Hom_\ec(-,M)$ for $M \in \mc$ to the right $d$-exact sequence above gives an exact sequence
\[
0\to \Hom_\ec(X_{d+1},M)\to \Hom_\ec(X_{d},M)\to \cdots \to \Hom_\ec(X_{0},M)
\]
of abelian groups, where we use that $\Ext^i_\ec(X_j,M)=0$ for $i=1,\dots,d-1$. In particular, this means that $(f_1, \dots, f_{d})$ is a $d$-cokernel of $f_0$ in the sense of \cite[Definition 2.2]{Jasso}. A \mbox{\textit{$d$-quotient}} of an object $X_1 \in \mc$ is defined to be a $d$-cokernel $(f_1, \dots, f_{d})$ of some morphism $f_0\colon X_0\to X_1$.

Similarly as in the abelian setup, given an exact sequence
\[
0\to X\to Y_1\to \cdots \to Y_d\to Z\to 0
\]
where $X$ and $Z$ are contained in $\mc$, there exists a Yoneda equivalent exact sequence with all terms in $\mc$; see \cite[Proposition A.1]{Iyama2007a} and \cite[Theorem 1.2]{Ebrahimi2}. Such an exact sequence in $\mc$ with $d$ middle terms is called a \textit{$d$-exact sequence} or \textit{$d$-extension}. Based on the discussion above, we see that these are $d$-exact in the sense of \cite[Definition 2.4]{Jasso}. Furthermore, if two $d$-exact sequences
\begin{align*}
& 0\to X\to Y_1\to \cdots \to Y_d\to Z\to 0 \\
& 0\to X\to Y'_1\to \cdots \to Y'_d\to Z\to 0
\end{align*}
 in $\mc$ are Yoneda equivalent, then there exists a commutative diagram
 \[
	\begin{tikzcd}
	0 \arrow[r] & X \arrow[r] \arrow[d, equal] & Y_1 \arrow[r] \arrow[d] & \cdots \arrow[r] & Y_d \arrow[r] \arrow[d] & Z \arrow[d, equal]  \arrow[r]  & 0\phantom{.}\\
	0 \arrow[r] & X \arrow[r] & Y'_1 \arrow[r] & \cdots \arrow[r] & Y'_{d} \arrow[r] & Z\arrow[r]  & 0;
	\end{tikzcd}
\]    
see \cite[Lemma 3.15]{Kvamme25}. In particular, Yoneda equivalence of $d$-exact sequences in $\mc$ coincides with the notion of equivalence from \cite[Definition 2.9]{Jasso}.

We define \textit{minimal} $d$-cokernels, $d$-kernels and $d$-extensions in a $d$-cluster tilting subcategory $\mc$ of an exact category $\ec$ analogously as in the abelian setup; see \cref{def: minimal}. For instance, this means that the $d$-cokernel \eqref{eq:d-cokernel exact}
 is \textit{minimal} if $f_i \in \Rad_{\ec}(X_i,X_{i+1})$ for $i=2,\ldots,d$. This is equivalent to $f_1,\dots,f_{d-1}$ being left minimal by \cref{lem: minimal iff radical}. A \mbox{$d$-quotient} is called \textit{minimal} if it is given by a minimal $d$-cokernel.

In the rest of this subsection, we work within the setup of a Krull--Schmidt exact category. For more details and basic properties of Krull--Schmidt categories, we refer the reader \mbox{to \cite{Krause}.}  Below is a version of \cite[Proposition 2.4]{HerschendJorgensen} in this setup; see also \cite[Lemma 2.6]{Klapproth} and \cite[Lemma 3.4]{HeZhou}. Note that when referring to direct summands in the proposition below, we mean direct summands in the category of $\mc$-complexes.

\begin{proposition}\label{prop:exact analogue of HJ}
Let $\mc$ be a $d$-cluster tilting subcategory of a Krull--Schmidt exact category $\ec$. 
\theoremlistfix
\begin{enumerate}
    \item Let $f \colon X\to Y$ be an admissible morphism in $\mc$. Then there exists a unique minimal $d$-cokernel of $f$, which is a direct summand of any other $d$-cokernel of $f$.
    \phantomsection
    \label{prop:exact analogue of HJ part 1}
    \item Let $f \colon X\to Y$ be an admissible morphism in $\mc$. Then there exists a unique minimal $d$-kernel of $f$, which is a direct summand of any other $d$-kernel of $f$.
    \phantomsection
    \label{prop:exact analogue of HJ part 2}
    \item In every equivalence class of $d$-extensions in $\mc$, there exists a unique minimal representative. Moreover, this minimal $d$-extension is a direct summand of any other $d$-extension in the same equivalence class.
    \phantomsection
    \label{prop:exact analogue of HJ part 3}
\end{enumerate}
\end{proposition}

\begin{proof}
Since $f \colon X \to Y$ is admissible, the cokernel $\operatorname{Coker}f$ exists, and the data of a \mbox{$d$-cokernel} \mbox{of $f$} is equivalent to the data of an $\mc$-coresolution of $\operatorname{Coker}f$. In particular, the morphism $f$ has a $d$-cokernel. An analogue argument shows that $f$ has a $d$-kernel. The remainder of the proof hence follows by similar arguments as in the proof of \cite[Proposition 2.4]{HerschendJorgensen}, using \cite[Example 2.1.25]{Krause22} whenever necessary.
\end{proof}

Inspired by the characterisation in \cref{theorem: extension closed}, our definition of a $d$-torsion class in a $d$-cluster tilting subcategory of an exact category involves closure under $d$-extensions and $d$-quotients. Let us now define these notions in the exact setup.

\begin{definition}\label{def:dtorsExact}
    Let $\ec$ be a Krull--Schmidt exact category and $\mc$ a $d$-cluster tilting subcategory of $\ec$. Consider a subcategory $\uc \subseteq \mc$.
    \theoremlistfix
\begin{enumerate}
    \item We say that $\uc$ is \textit{closed under $d$-extensions} if for every pair of objects $X,Z\in \uc$ and every minimal $d$-extension
\[
0 \xrightarrow{} X \xrightarrow{} Y_1 \xrightarrow{} \cdots \xrightarrow{} Y_d \xrightarrow{} Z \xrightarrow{} 0,
\]
we have $Y_i \in \uc$ for $1,\dots,d$. 
 \item We say that $\uc$ is \textit{closed under $d$-quotients} if for every minimal $d$-quotient 
\[
X \xrightarrow{} Y_1 \xrightarrow{} \cdots \xrightarrow{} Y_{d} \to 0
\]
of an object $X \in \uc$, we have $Y_i \in \uc$ for $i=1,\dots,d$.
\end{enumerate}
\end{definition}

We are now ready to present the definition of a $d$-torsion class in the generality of this subsection.

\begin{definition}\label{def: d-torsion exact setup}
    Let $\ec$ be a Krull--Schmidt exact category and $\mc$ a $d$-cluster tilting subcategory of $\ec$. A subcategory $\uc$ of $\mc$ is called a \textit{$d$-torsion class} of $\mc$ if it is closed under $d$-extensions and $d$-quotients.
\end{definition}

\begin{remark}
We do not know if a $d$-torsion class as in \cref{def:dtorsExact} satisfies a similar property as in \cref{def:ntorsionclass}, since it is not clear how to extend the proof of \Cref{theorem: extension closed} to exact categories. The proof relies on \cite[Proposition 3.15]{AHJKPT1}, which again relies on being able to take the image of an arbitrary morphism. This is not possible in an exact category. In particular, it is not clear that a $d$-torsion class in the sense of \cref{def:dtorsExact} is contravariantly finite.
\end{remark}

\subsection{Coresolutions by $d$-torsion classes in the exact setup}\label{subsec:cores exact} 
In this subsection, we consider a Krull--Schmidt exact category $\ec$ and a $d$-cluster tilting subcategory $\mc \subseteq \ec$. Our main goal is to prove an analogue of \cref{thm:U_i Ext^d-projective} in this setup. 

Let us first introduce some necessary terminology. Similarly as in the abelian case, see \cref{defn:ExtdProj}, an object $X$ in a $d$-torsion class $\uc \subseteq \mc$ is called $\Ext^d$\textit{-projective \mbox{in $\uc$}} if $\Ext^d_\ec(X,U)=0$ for all $U\in \uc$. We say that such an object $X$ is an \textit{$\Ext^d$-projective generator of $\uc$} if any other $\Ext^d$-projective in $\uc$ is contained in $\add(X)$. A \textit{$\uc$-coresolution} of an object $E \in \ec$ is a complex
\begin{equation}\label{eq:U-cores exact}
 E\xrightarrow{f_0} U_0\xrightarrow{f_1} U_1\xrightarrow{f_2} \cdots \xrightarrow{f_d} U_d\to 0
\end{equation}
which is exact at $U_0,\dots,U_d\in\uc$, and where the morphism $f_0 \colon E\xrightarrow{}U_0$ is a left $\uc$-approximation. In this case, the complex
\[
    0\to \Hom_\ec(U_d,U)\to \Hom_\ec(U_{d-1},U)\to \cdots \to \Hom_\ec(U_0,U)\to \Hom_\ec(E,U)\to 0
\]
is exact for all $U\in \uc$. A $\uc$-coresolution \eqref{eq:U-cores exact}
is called \textit{minimal} if $f_0,\dots,f_{d-1}$ are left minimal.

\begin{theorem}\label{Ext^d-projectiveExactCat}
Let $\ec$ be a Krull--Schmidt exact category and $\mc$ a $d$-cluster tilting subcategory of $\ec$. Consider a $d$-torsion class $\uc$ in $\mc$. Suppose $P\in \ec$ is projective and let 
\[
P \xrightarrow{f_0} U_0 \xrightarrow{f_1} U_1 \xrightarrow{f_2} \cdots \xrightarrow{f_d} U_d \to 0
\]
be a minimal $\uc$-coresolution of $P$. Then $U_i$ is $\Ext^d$-projective in $\uc$ for $i=0,\dots,d$.
\end{theorem}

\begin{proof}
We first prove that $U_0$ is $\Ext^d$-projective in $\uc$, following analogous arguments as in the proof of \cref{lem:casei=0}. Let $U\in \uc$ and consider $\delta\in \Ext^d_{\ec}(U_0,U)$. We can represent $\delta$ by a minimal $d$-exact sequence
\[
0 \to U \to E_1 \to \cdots \to E_{d-1} \to E_d \xrightarrow{g} U_0 \to 0.
\]
Since $\uc$ is closed under $d$-extensions, it follows that $E_i \in \uc$ for $i=1,\dots,d$. In particular, we have a deflation $g\colon E_d\to U_0$ with $E_d\in \uc$. Since $P$ is projective, the morphism \mbox{$f_0\colon P\to U_0$} lifts along $g$ to a morphism $h\colon P\to E_d$. As $E_d\in \uc$ and $f_0\colon P\to U_0$ is a left \mbox{$\uc$-approximation}, there exists a morphism $k\colon U_0\to E_d$ such that $k f_0=h$. It follows that $f_0=g k f_0$, and since $f_0$ is left minimal, the composition $g k$ is an isomorphism. Hence, the morphism $g$ is a split epimorphism, which yields $\delta=0$ in $\Ext^d_{\ec}(U_0,U)$. Since $\delta$ was arbitrary, it follows that $\Ext^d_{\ec}(U_0,U)=0$, so $U_0$ is $\Ext^d$-projective in $\uc$.

The fact that $U_i$ is $\Ext^d$-projective for $i=1,\dots,d$ now follows by induction in the same way as in the proof of \cref{thm:U_i Ext^d-projective}. 
\end{proof}

We end this subsection by giving a criterion for when all objects have $\uc$-coresolutions. For this, we say that a covariantly finite subcategory $\uc \subseteq \mc$ is \textit{admissibly covariantly finite} if the minimal left $\uc$-approximation of any $M \in \mc$ is an admissible morphism. Note that later, our focus will be on faithful $d$-torsion classes in the sense of \cref{def:faithful exact}, and such $d$-torsion classes are admissibly covariantly finite whenever they are covariantly finite.

\begin{proposition}\label{AdmittingUCores}
Let $\ec$ be a Krull--Schmidt exact category and $\mc$ a $d$-cluster tilting subcategory of $\ec$. Consider an admissibly covariantly finite $d$-torsion class $\uc$ in $\mc$. Then any object in $\ec$ admits a minimal $\uc$-coresolution.
\end{proposition}

\begin{proof}
    Let $E\in \ec$. As $\uc$ is admissibly covariantly finite, the minimal left $\uc$-approximation \mbox{$f_0\colon E\to U_0$} of $E$ is admissible. There hence exists a minimal $d$-cokernel
    \[
    U_0\xrightarrow{f_1}U_1\xrightarrow{f_2} \cdots \xrightarrow{f_d}U_d\to 0
    \]
    of $f_0$ by \cref{prop:exact analogue of HJ} \eqref{prop:exact analogue of HJ part 1}. Since $\uc$ is closed under $d$-quotients, it follows that $U_i\in \uc$ for $i=1,\dots,d$. We furthermore have $f_i \in \Rad_{\ec}(U_{i-1},U_{i})$ for $i=2,\dots,d$. By \cref{lem: minimal iff radical}, this implies that $f_1,\dots,f_{d-1}$ are left minimal. Hence, the sequence
    \[
    E \xrightarrow{f_0} U_0 \xrightarrow{f_1} U_1 \xrightarrow{f_2} \cdots \xrightarrow{f_d} U_d \to 0
    \]
is a minimal $\uc$-coresolution of $E$.
\end{proof}

\subsection{A bijection of $d$-torsion classes}\label{subsec:bijection}

In this subsection we work with a $d$-cluster tilting subcategory $\mc$ of $\modA$. Our goal is to construct a bijection between $d$-torsion classes in $\mc$ and faithful $d$-torsion classes in $\pc_d(\mc)$; see \cref{def:faithful exact} and \cref{thm: dtorsion to faithfuldtorsion}. For this, we need the lemma below, where we gather a series of results concerning the homology functor
\[
H_0(-)\colon C^{[-d,0]}(\proj A) \to \modA.
\]

\begin{lemma}\label{Lemma:PropertiesZeroHomology}
    Let $\mc$ be a $d$-cluster tilting subcategory of $\modA$. The following statements hold:
    \theoremlistfix
    \begin{enumerate}
        \item We have a natural isomorphism $H_0(-)\cong \Ext^d_{C^{[-d,0]}(\proj A)}(A[d],-)$.
        \phantomsection
        \label{Lemma:PropertiesZeroHomology:1}
        \item If $X_\bullet^0\to X_\bullet^1 \to \cdots \to X_\bullet^d\to X_\bullet^{d+1}\to 0$ is a right $d$-exact sequence in $\pc_d(\mc)$, then 
        \[
        H_0(X_\bullet^0)\to H_0(X_\bullet^1) \to \cdots \to H_0(X_\bullet^d)\to H_0(X_\bullet^{d+1})\to 0
        \]
        is a right $d$-exact sequence in $\mc$.
        \phantomsection
        \label{Lemma:PropertiesZeroHomology:2}
        \item The functor $H_0(-)$ induces an equivalence 
        \[
        \pc_d(\mc)/\mathcal{I}\cong \mc,
        \]
        where $\mathcal{I}$ is the subcategory of injective objects in $C^{[-d,0]}(\proj A)$.
        \phantomsection
        \label{Lemma:PropertiesZeroHomology:3}
        \item If a morphism $f_\bullet\colon X_\bullet\to Y_\bullet$ in $\pc_d(\mc)$ is contained in the radical of $C^{[-d,0]}(\proj A)$, then $H_0(f_\bullet)$ is contained in the radical of $\mod A$.
        \phantomsection
        \label{Lemma:PropertiesZeroHomology:4}
    \end{enumerate}
\end{lemma}

\begin{proof}
The first statement is an immediate consequence of \cref{lemma: ext for complexes}, since
\[
\Ext^d_{C^{[-d,0]}(\proj A)}(A[d],-)\cong\Hom_{K^{b}(\proj A)}(A[d],-[d])\cong\Hom_{K^{b}(\proj A)}(A,-)\cong H_0(-).
\]

To prove \eqref{Lemma:PropertiesZeroHomology:2}, assume that $X_\bullet^0\to X_\bullet^1 \to \cdots \to X_\bullet^d\to X_\bullet^{d+1}\to 0$ is a right $d$-exact sequence in $\pc_d(\mc)$. Let $C_\bullet^i$ denote the image of the morphism $X^{i-1}_\bullet \to X^i_\bullet$. Exactness means that we have conflations
 \[
 0\to C^i_\bullet \to  X^i_\bullet\to C^{i+1}_\bullet \to 0
 \]
for $1 \leq i \leq d$, where $C^{d+1}_\bullet=X^{d+1}_\bullet$. By the definition of $\pc_d(\mc)$, the complex
\[
    H_0(X_\bullet^0)\to H_0(X_\bullet^1) \to \cdots \to H_0(X_\bullet^d)\to H_0(X_\bullet^{d+1})\to 0
\]
is in $\mc$. To prove that it is right $d$-exact, it suffices to show that the sequences
 \[
0\to H_0(C^i_\bullet) \to  H_0(X^i_\bullet)\to H_0(C^{i+1}_\bullet) \to 0
 \]
are exact for $2\leq i\leq d$ and that the sequence
 \[
H_0(X^0_\bullet) \to  H_0(X^1_\bullet)\to H_0(C^{2}_\bullet) \to 0
 \]
 is right exact. Since we have a natural isomorphisms $H_0(-)\cong \Ext^d_{C^{[-d,0]}(\proj A)}(A[d],-)$ from \eqref{Lemma:PropertiesZeroHomology:1} and $C^{[-d,0]}(\proj A)$ has global dimension $\leq d$ by \cref{prop:enough proj/inj} (\ref{prop:enough proj/inj3}), the functor $H_0(-)$ is right exact. This immediately proves that $H_0(X^0_\bullet) \to  H_0(X^1_\bullet)\to H_0(C^{2}_\bullet) \to 0$ is right exact. 
 
 Next, we apply $\Hom_{C^{[-d,0]}(\proj A)}(A[d],-)$ to the conflation $0\to C^i_\bullet \to  X^i_\bullet\to C^{i+1}_\bullet \to 0$ for $2\leq i\leq d$. This gives a long exact sequence 
\begin{multline*} 
\cdots \to \Ext^{d-1}_{C^{[-d,0]}(\proj A)}(A[d],X^{i}_\bullet) \to \Ext^{d-1}_{C^{[-d,0]}(\proj A)}(A[d],C^{i+1}_\bullet) \to  \\  \Ext^d_{C^{[-d,0]}(\proj A)}(A[d],C^i_\bullet) \to  \Ext^d_{C^{[-d,0]}(\proj A)}(A[d],X^i_\bullet)  \to 
  \Ext^d_{C^{[-d,0]}(\proj A)}(A[d],C^{i+1}_\bullet) \to 
0.
\end{multline*}
 Since $H_0(-)\cong \Ext^d_{C^{[-d,0]}(\proj A)}(A[d],-)$ by \eqref{Lemma:PropertiesZeroHomology:1}, it suffices to show that 
 \[
 \Ext^{d-1}_{C^{[-d,0]}(\proj A)}(A[d],C^{i+1}_\bullet)=0
 \]
 for all $2\leq i\leq d$. As $A[d]$ and $X^k_\bullet$ both belong to the subcategory $\pc_d(\mc)$, which is \mbox{$d$-cluster} tilting in $C^{[-d,0]}(\proj A)$ by \cref{theorem:dCT exact}, we have $\Ext^j_{C^{[-d,0]}(\proj A)}(A[d],X^k_\bullet)=0$ for $1 \leq j \leq d-1$ and $0\leq k \leq d+1$. Therefore, by dimension shifting, we get that
\begin{eqnarray*}
 \Ext^{d-1}_{C^{[-d,0]}(\proj A)}(A[d],C^{i+1}_\bullet) & \cong & \Ext^{d-2}_{C^{[-d,0]}(\proj A)}(A[d],C^{i+2}_\bullet)  \\ & \smash{\vdots} & \\
 & \cong & \Ext^{i-1}_{C^{[-d,0]}(\proj A)}(A[d],C^{d+1}_\bullet).
\end{eqnarray*}
 This must be zero since $C^{d+1}_\bullet=X^{d+1}_\bullet\in \pc_d(\mc)$ and $2\leq i\leq d$, proving \eqref{Lemma:PropertiesZeroHomology:2}. 

 For \eqref{Lemma:PropertiesZeroHomology:3}, note that by \cref{prop:enough proj/inj} (\ref{prop:enough proj/inj1}), the subcategory $\mathcal{I} \subseteq C^{[-d,0]}(\proj A)$ of injective objects is precisely the additive closure of $A[d]$ and the contractible complexes 
 $$D_i(A)=(\cdots \to 0\to A\xrightarrow{1}A\to 0\to \cdots)$$ 
 concentrated in degrees $i$ and $i+1$ for $-d\leq i\leq -1$. Since the quotient of $\pc_d(\mc)$ by the contractible complexes is equivalent to $\widetilde{\pc}_d(\mc)$, the claim thus follows from \Cref{Proposition:EquivAdditivQuot}.

For \eqref{Lemma:PropertiesZeroHomology:4}, we may without loss of generality assume that $X_\bullet$ and $Y_\bullet$ are indecomposable. Suppose that $H_0(f_\bullet)$ is not in the radical. It follows that the objects $X_\bullet$ and $Y_\bullet$ cannot be injective in $C^{[-d,0]}(\proj A)$, since this would give $H_0(f_\bullet)=0$ by \cref{prop:enough proj/inj} (\ref{prop:enough proj/inj1}). Furthermore, note that $H_0(X_\bullet)$ and $H_0(Y_\bullet)$ are indecomposable, yielding that $H_0(f_\bullet)$ is an isomorphism. By \eqref{Lemma:PropertiesZeroHomology:3}, this implies the existence of injective objects $I_\bullet^1,I_\bullet^2\in \mathcal{I}$ such that $$\begin{pmatrix}f_\bullet&g^1_\bullet\\g^2_\bullet&g^3_\bullet\end{pmatrix}\colon X_\bullet\oplus I_\bullet^1 \xrightarrow{\cong} Y_\bullet\oplus I_\bullet^2$$  
is an isomorphism for some morphisms $g^1_\bullet,g^2_\bullet$ and $g^3_\bullet$. In particular, this yields that $$\begin{pmatrix}f_\bullet\\g^2_\bullet\end{pmatrix}\colon X_\bullet\to Y_\bullet\oplus I_\bullet^2$$ is a split monomorphism. Since $X_\bullet$ is indecomposable and non-injective, the morphism \mbox{$f_\bullet\colon X_\bullet\to Y_\bullet$} must itself be a split monomorphism. Hence, it is an isomorphism since $Y_\bullet$ is indecomposable. We can thus conclude that $f_\bullet$ is not contained in the radical.
\end{proof}

In order to state \cref{thm: dtorsion to faithfuldtorsion}, we need the notion of a faithful $d$-torsion class \mbox{in the exact setup.}

\begin{definition}\label{def:faithful exact}
Let $\ec$ be a Krull--Schmidt exact category and $\mc$ a $d$-cluster tilting subcategory of $\ec$. We say that a $d$-torsion class $\uc$ in $\mc$ is \textit{faithful} if for any object $E\in \ec$, there exists an inflation $E \to U$ with $U\in \uc$.
\end{definition}

Note that if a $d$-torsion class is faithful, then it contains all the injective objects in $\ec$. The converse holds whenever $\ec$ has enough injectives.

\begin{remark}
    In the setup of the definition above, assume additionally that $\ec$ has enough projectives. Then a $d$-torsion class $\uc$ is faithful if and only if for any projective $P\in \ec$, there exists an inflation $P\to U$ with $U\in \uc$. To see that the latter implies faithfulness, consider the pushout of a deflation $P \to E$ from a projective $P$ to an object $E$ along an inflation $P \to U$ with $U \in \uc$, and use that $\uc$ is closed under $d$-quotients. As a consequence,  \cref{def:faithful exact} recovers the notion of a faithful $d$-torsion class in the case where $\ec = \mod A$.
\end{remark}

We are now ready to prove the main result of this subsection.

\begin{theorem}\label{thm: dtorsion to faithfuldtorsion}
Let $\mc$ be a $d$-cluster tilting subcategory of $\modA$. We have a bijection
\begin{align*}
\begin{array}{ccc}
\begin{Bmatrix}
    \text{$d$-torsion classes in $\mc$}
\end{Bmatrix}
&
    \to 
    &
\begin{Bmatrix}
    \text{faithful $d$-torsion classes in $\pc_d(\mc)$}
\end{Bmatrix}\!,\\
\uc &\mapsto  &\pc_d(\uc)
\end{array}
\end{align*}
where $\pc_d(\uc)\colonequals\{X_\bullet\in \pc_d(\mc)\mid H_0(X_\bullet)\in \uc\}$. Furthermore, a $d$-torsion class $\uc$ is functorially finite if and only if $\pc_d(\uc)$ is covariantly finite.
\end{theorem}

\begin{proof}
Let $\uc$ be a subcategory of $\mc$, and consider
\[
\pc_d(\uc) = \{X_\bullet\in \pc_d(\mc)\mid H_0(X_\bullet)\in \uc\} \subseteq \pc_d(\mc).
\]
Note that $M\in \uc$ if and only if its minimal projective \mbox{$d$-presentation} $P_\bullet^M$ is in $\pc_d(\uc)$. Hence, the subcategory $\pc_d(\uc)$ uniquely determines $\uc$. By the equivalence in \cref{Lemma:PropertiesZeroHomology} \eqref{Lemma:PropertiesZeroHomology:3} and the description of injectives in $C^{[-d,0]}(\proj A)$ given in \cref{prop:enough proj/inj} (\ref{prop:enough proj/inj1}), we see that a subcategory of $\pc_d(\mc)$ contains all the injectives if and only if it is of the form $\pc_d(\uc)$ for some subcategory $\uc$ of $\mc$.    
 
Recall that since $C^{[-d,0]}(\proj A)$ has enough injectives by \cref{prop:enough proj/inj} (\ref{prop:enough proj/inj2}), the faithful \mbox{$d$-torsion} classes of $\pc_d(\mc)$ are precisely the $d$-torsion classes that contain all injectives. The faithful $d$-torsion classes of $\pc_d(\mc)$ must thus be of the form $\pc_d(\uc)$ for some subcategory $\uc$ of $\mc$. To prove the bijection, it hence suffices to show that $\pc_d(\uc)$ is a $d$-torsion class if and only if $\uc$ is a $d$-torsion class.

For this, assume first that $\uc$ is a $d$-torsion class in $\mc$. We want to show that \mbox{$\pc_d(\uc) \subseteq \pc_d(\mc)$} is closed under $d$-quotients and $d$-extensions. To this end, let 
\[
X_\bullet^1\to X_\bullet^2 \to \cdots \to X_\bullet^{d+1}\to 0
\]
be a minimal $d$-cokernel of a morphism $X_\bullet^0\to X_\bullet^1$ in $\pc_d(\mc)$, and assume $X_\bullet^1\in \pc_d(\uc)$. By \cref{Lemma:PropertiesZeroHomology} \eqref{Lemma:PropertiesZeroHomology:2}, the sequence 
\[
H_0(X_\bullet^1)\to H_0(X_\bullet^2) \to \cdots \to H_0(X_\bullet^{d+1})\to 0
\]
in $\mc$ is a \mbox{$d$-cokernel} of \mbox{$H_0(X_\bullet^0)\to H_0(X_\bullet^1)$}, and this $d$-cokernel is minimal by \mbox{\cref{Lemma:PropertiesZeroHomology} \eqref{Lemma:PropertiesZeroHomology:4}}. Since \mbox{$H_0(X_\bullet^1)\in \uc$} and $\uc$ is closed under $d$-quotients, it follows that $H_0(X_\bullet^i)\in \uc$ for \mbox{$i=2,\dots,d+1$}. This yields $X_\bullet^i\in \pc_d(\uc)$ for $i=2,\dots,d+1$, which proves that $\pc_d(\uc)$ is closed under \mbox{$d$-quotients.}

Consider next a minimal $d$-exact sequence 
\[
0\to U_\bullet^0\to Y_\bullet^1\to \cdots \to Y_\bullet^d\to U_\bullet^{d+1}\to 0
\]
in $\pc_d(\mc)$ with $U_\bullet^0,U_\bullet^{d+1}\in \pc_d(\uc)$. Applying $H_0$ and using \Cref{Lemma:PropertiesZeroHomology} \eqref{Lemma:PropertiesZeroHomology:2} and \eqref{Lemma:PropertiesZeroHomology:4},  we get a right $d$-exact sequence
\[
H_0(U_\bullet^0)\to H_0(Y_\bullet^1)\to \cdots \to H_0(Y_\bullet^d)\to H_0(U_\bullet^{d+1})\to 0
\]
in $\mc$, where $H_0(Y_\bullet^i)\to H_0(Y_\bullet^{i+1})$ are radical morphisms for $i=1,\dots,d-1$. By \cref{Reformulation:d-Tors}, it follows that $H_0(Y_\bullet^i)\in \uc$ for $i=1,\dots,d$ since $H_0(U_\bullet^0),H_0(U_\bullet^{d+1})\in \uc$. This means that $Y_\bullet^i\in \pc_d(\uc)$ for $i=1,\dots,d$, so $\pc_d(\uc)$ is closed under $d$-extensions, allowing us to conclude that $\pc_d(\uc)$ is a $d$-torsion class in $\pc_d(\mc)$.

Assume now that $\pc_d(\uc)$ is a $d$-torsion class. We show that $\uc$ is a $d$-torsion class in $\mc$ by proving closure under $d$-quotients and $d$-extensions. Consider hence a minimal $d$-cokernel
\[
M_1\to \cdots \to M_{d+1}\to 0
\]
in $\mc$ of a morphism $M_0\to M_1$, and assume $M_1\in \uc$. Choose a lift $P_\bullet^{M_0} \to P_\bullet^{M_1}$ of $M_0 \to M_1$, and let
\[
P_\bullet^{M_1}\to P_\bullet^2\to \cdots \to P_\bullet^{d+1}\to 0
\]
be its minimal $d$-cokernel in $\pc_d(\mc)$. Note that $P_\bullet^{M_1} \in \pc_d(\uc)$ since $M_1 \in \uc$. As $\pc_d(\uc)$ is closed under $d$-quotients, we have $P_\bullet^i \in \pc_d(\uc)$ for $i = 2,\dots,d+1$. Applying $H_0$ and using \cref{Lemma:PropertiesZeroHomology} \eqref{Lemma:PropertiesZeroHomology:2} and \eqref{Lemma:PropertiesZeroHomology:4}, we get a minimal $d$-cokernel
\[
M_1\to H_0(P_\bullet^2)\to \cdots \to H_0(P_\bullet^{d+1})\to 0
\]
of $M_0\to M_1$. Since minimal $d$-cokernels are unique up to isomorphism by \mbox{\cref{prop:exact analogue of HJ} (\ref{prop:exact analogue of HJ part 1})}, it follows that $M_i\cong H_0(P_\bullet^i) \in \uc$ for $i = 2,\dots,d+1$. This shows that $\uc$ is closed under $d$-quotients.

Consider next a minimal $d$-extension 
\[
0\to U_0\to M_1\to \cdots \to M_d\to U_{d+1}\to 0
\]
in $\mc$ with $U_0,U_{d+1}\in \uc$. Let $K_i$ denote the image of $M_i\to M_{i+1}$ for $i=1,\dots,d-1$. We hence have short exact sequences
\begin{align*}
& 0\to K_i\to M_{i+1}\to K_{i+1}\to 0 \quad \quad \text{for } i=1\dots,d-2, \\ 
& 0\to U_0\to M_1\to K_1\to 0 \quad \text{and} \quad 0\to K_{d-1}\to M_d\to U_{d+1}\to 0.  
\end{align*}
Consider the minimal projective $d$-presentations $P_\bullet^{U_0},P_\bullet^{K_1},\dots ,P_\bullet^{K_{d-1}},P_\bullet^{U_{d+1}}$ of the modules $U_0,K_1,\\ \dots,K_{d-1},U_{d+1}$. By the horseshoe lemma, we can find a projective $d$-presentation $Q_\bullet^{M_i}$ of $M_i$ for $i=1,\dots,d$ such that
\begin{align*}
& 0\to P_\bullet^{K_i}\to Q_\bullet^{M_{i+1}}\to P_\bullet^{K_{i+1}}\to 0 \quad \quad \text{for }i=1\dots,d-2, \\ 
& 0\to P_\bullet^{U_0}\to Q_\bullet^{M_{1}}\to P_\bullet^{K_{1}}\to 0 \quad \text{and} \quad 0\to P_\bullet^{K_{d-1}}\to Q_\bullet^{M_{d}}\to P_\bullet^{U_{d+1}}\to 0 
\end{align*}
are conflations in $C^{[-d,0]}(\proj A)$. Hence, we get a $d$-exact sequence
\[
0\to P_\bullet^{U_0}\to Q_\bullet^{M_1}\to \cdots \to Q_\bullet^{M_d}\to P_\bullet^{U_{d+1}}\to 0
\]
in $\pc_d(\mc)$. It follows from \cref{prop:exact analogue of HJ} (\ref{prop:exact analogue of HJ part 3}) that this sequence must be a direct sum of a minimal \mbox{$d$-extension} 
\[
0\to P_\bullet^{U_0}\to R_\bullet^{1}\to \cdots \to R_\bullet^{d}\to P_\bullet^{U_{d+1}}\to 0
\]
and exact sequences of the form $\cdots \to 0\to R_\bullet\xrightarrow{1}R_\bullet\to 0\to \cdots$. Since 
\[
0\to H_0(P_\bullet^{U_0})\to H_0(Q_\bullet^{M_1})\to \cdots \to H_0(Q_\bullet^{M_d})\to H_0(P_\bullet^{U_{d+1}})\to 0
\]
is isomorphic to the original $d$-exact sequence in $\mc$, it is minimal. Hence, it must be isomorphic to 
\[
0\to H_0(P_\bullet^{U_0})\to H_0(R_\bullet^{1})\to \cdots \to H_0(R_\bullet^{d})\to H_0(P_\bullet^{U_{d+1}})\to 0.
\]
Since $\pc_d(\uc)$ is closed under $d$-extensions and $P_\bullet^{U_0},P_\bullet^{U_{d+1}}\in \pc_d(\uc)$, we have $R_\bullet^i\in \pc_d(\uc)$ for $i=1,\dots,d$, and hence $$M_i\cong H_0(Q_\bullet^{M_i})\cong H_0(R_\bullet^i)\in \uc.$$ 
This shows that $\uc$ is closed under $d$-extensions, and  hence is a $d$-torsion class.

Finally, we want to prove that a $d$-torsion class $\uc$ is functorially finite if and only if $\pc_d(\uc)$ is covariantly finite. As any $d$-torsion class in $\mc$ is contravariantly finite, we need to show that $\uc$ is covariantly finite if and only if $\pc_d(\uc)$ is covariantly finite. Since $\mc$ and $\pc_d(\mc)$ are both covariantly finite, it suffices to show that all objects in $\mc$ have a left $\uc$-approximation if and only if all objects in $\pc_d(\mc)$ have a left $\pc_d(\uc)$-approximation. To see this, consider $P_\bullet\in \pc_d(\mc)$. By \cref{prop:enough proj/inj} (\ref{prop:enough proj/inj2}), we may choose an inflation $P_\bullet\to Q_\bullet$ to an injective object $Q_\bullet$ in $C^{[-d,0]}(\proj A)$. Since $H_0(-)$ induces an equivalence $ \pc_d(\mc)/\mathcal{I}\xrightarrow{} \mc$ by \cref{Lemma:PropertiesZeroHomology} \eqref{Lemma:PropertiesZeroHomology:3}, a morphism $P_\bullet\to R_\bullet$ must give a left $\uc$-approximation $H_0(P_\bullet)\to H_0(R_\bullet)$ if and only if the morphism $P_\bullet\to R_\bullet\oplus Q_\bullet$ is a left $\pc_d(\uc)$-approximation. Since $H_0(-)\colon \pc_d(\mc)\to \mc$ is full and dense, this proves the claim. 
\end{proof}

\subsection{Silting complexes}\label{subsec: silting proof}
In this subsection we consider a $d$-cluster tilting subcategory $\mc$ of $\modA$.
The goal is to prove \cref{thm:silting}, as well as to describe some consequences and examples. In particular, this includes proving that the $\tau_d$-rigid pair $(M_\uc, P_\uc)$ associated with a functorially finite $d$-torsion class $\uc$ in $\mc$ as in \cref{thm:maximaltaunrigid} is maximal; see \cref{cor: maximality}. This finishes the proof of \cref{MainResultIntro} from the introduction.  

\cref{Proposition:Ext^dProjGendComplex} below will be used in the proof of \cref{thm:silting}. Recall that 
\[
P_\bullet^{(M_\uc,P_\uc)} = P_\bullet^{M_\uc} \oplus P_\uc[d],
\]
where $P_\bullet^{M_\uc}$ is the minimal projective $d$-presentation of $M_\uc$.

\begin{proposition}\label{Proposition:Ext^dProjGendComplex}
    Let $\mc$ be a $d$-cluster tilting subcategory of $\modA$ and consider a functorially finite $d$-torsion class $\uc$ in $\mc$. Then the complex $P_\bullet^{(M_\uc,P_\uc)}$ is homotopy equivalent to an $\Ext^d$-projective generator of $\pc_d(\uc)$.
\end{proposition}

\begin{proof}
    It suffices to show that a non-contractible indecomposable complex $P_\bullet$ in $\pc_d(\uc)$ is $\Ext^d$-projective in $\pc_d(\uc)$ if and only if it is a direct summand of $P_\bullet^{(M_\uc,P_\uc)}$.
    
    Assume $P_\bullet$ is injective in $C^{[-d,0]}(\proj A)$. Then it must be isomorphic to $P[d]$ for some projective module $P$ in $\modA$ by \cref{prop:enough proj/inj} (\ref{prop:enough proj/inj1}). By a similar argument as in the proof of \cref{Lemma:PropertiesZeroHomology} \eqref{Lemma:PropertiesZeroHomology:1}, we have an isomorphism
    \[
    \Ext^d_{C^{[-d,0]}(\proj A)}(P[d],Q_\bullet)\cong \Hom_{A}(P,H_0(Q_\bullet))
    \]
    for any complex $Q_\bullet\in C^{[-d,0]}(\proj A)$. This yields $\Ext^d_{C^{[-d,0]}(\proj A)}(P[d],Q_\bullet)=0$ for all \mbox{$Q_\bullet\in\pc_d(\uc)$} if and only if $\Hom_A(P,U)=0$ for all $U\in \uc$. Hence, the object $P[d]$ is $\Ext^d$-projective in $\pc_d(\uc)$ if and only if $P$ is a direct summand of $P_\uc$. The conclusion now follows as $P_{\bullet}^{M_{\uc}}$ is a truncation of a minimal projective resolution and hence does not have $P[d]$ as a direct summand.
    
    Assume next that $P_\bullet$ is not injective in $C^{[-d,0]}(\proj A)$. Then $P_\bullet$ is isomorphic to a minimal projective $d$-presentation $P_\bullet^U$ of an object $U\in \uc$ by \cref{Lemma:PropertiesZeroHomology} \eqref{Lemma:PropertiesZeroHomology:3}. By \cref{lemma: ext for complexes}, we have
    \[
    \Ext^d_{C^{[-d,0]}(\proj A)}(P_\bullet^U,P_\bullet^V)\cong  \Hom_{K^b(\proj A)}(P_\bullet^U,P_\bullet^V[d])
    \]
for any object $V\in \uc$. Next, we apply \cite[Lemma 3.4]{AdachiIyamaReiten} to $\Omega^{d-1}U$ and $V$ to obtain that \mbox{$\Hom_{K^b(\proj A)}(P_\bullet^U,P_\bullet^V[d])=0$} if and only if $\Hom_A(V, \tau_d U)=0$, as $P_d^U \to P_{d-1}^U$ is the projective presentation of $\Omega^{d-1}U$.
Combining these observations with \cref{thm: equivalence of Extd projectivity}, it follows that $P_\bullet^U$ is $\Ext^d$-projective in $\pc_d(\uc)$ if and only if $U$ is $\Ext^d$-projective in $\uc$. The latter is equivalent to $U$ being a direct summand of $M_\uc$, which is again equivalent to $P_\bullet^U$ being a direct summand of $P_\bullet^{M_\uc}$. This proves the claim. 
\end{proof}

\begin{remark}
  Assume $A$ is basic. Then the basic $\Ext^d$-projective generator of $\pc_d(\uc)$ is
  \[
  P_\bullet^{(M_\uc,P_\uc)}\oplus \bigoplus_{i=-1}^{-d}D_i(A),
  \]
  where $D_i(A)=(\cdots \to 0\to A\xrightarrow{1_A}A\to 0\to \cdots)$ is concentrated in degrees $i$ and $i+1$. If $A$ is not basic, then the same statement holds if we replace $D_i(A)$ with $D_i(P)$ for a basic projective generator $P$ of $\modA$.
\end{remark}

We are now ready to give the proof of \cref{thm:silting}.

\begin{proof}[Proof of \Cref{thm:silting}]
  We know that $P_\bullet^{(M_\uc,P_\uc)}$ is a presilting complex in $K^b(\proj A)$ by \cref{thm:maximaltaunrigid} and \cref{prop:presilting}, so it suffices to show that its thick closure in $K^b(\proj A)$ contains $A$. Since $\uc$ is a functorially finite $d$-torsion class in $\mc$, the subcategory $\pc_d(\uc)$ is a covariantly finite faithful $d$-torsion class in $\pc_d(\mc)$ by \Cref{thm: dtorsion to faithfuldtorsion}. Hence, we can find a minimal $\pc_d(\uc)$-coresolution 
  \begin{equation}\label{ImportantExactSeq}
       0\to A\to P^0_\bullet\to \cdots \to P^{d}_\bullet\to 0 
  \end{equation}
  in $C^{[-d,0]}(\proj A)$ by \Cref{AdmittingUCores}, where $A$ is considered as a stalk complex concentrated in degree $0$. Note that the morphism $A \to P_\bullet^0$ is an inflation since $\pc_d(\uc)$ is faithful. As $A$ is projective in $C^{[-d,0]}(\proj A)$ by \cref{prop:enough proj/inj} (\ref{prop:enough proj/inj1}), the object $P^i_\bullet$ is $\Ext^d$-projective in $\pc_d(\uc)$ for $i=0,\dots,d$ by \cref{Ext^d-projectiveExactCat}. By \cref{Proposition:Ext^dProjGendComplex}, we know that $P_\bullet^{(M_\uc,P_\uc)}$ is isomorphic in $K^b(\proj A)$ to an $\Ext^d$-projective generator of $\pc_d(\uc)$. Each $P_\bullet^i$ must thus be isomorphic in $K^b(\proj A)$ to an object in $\operatorname{add}(P_\bullet^{(M_\uc,P_\uc)})$. This implies that $A$ is in the thick closure of $P_\bullet^{(M_\uc,P_\uc)}$ in $K^b(\proj A)$, since the exact sequence \eqref{ImportantExactSeq} realises $A$ as an iterated cocone of morphisms between objects in $\add (P_\bullet^{(M_\uc,P_\uc)})$. This proves the claim.
\end{proof}

In the remaining part of this subsection, we investigate some consequences of \cref{thm:silting}. We start by considering the map $\psi_d$ sending $\uc$ to $P_\bullet^{(M_\uc,P_\uc)}$ from \mbox{diagram \eqref{diag:silting intro}} in the introduction. 

\begin{proposition}\label{prop: injective silting}
The map $\psi_d$ is injective, with partial inverse given by
\[
P_\bullet^{(M_\uc,P_\uc)}
\mapsto \fac H_0(P_\bullet^{(M_\uc,P_\uc)}) \cap \mc.
\]
\end{proposition}

\begin{proof}
Observe that $H_0(P_\bullet^{(M_\uc,P_\uc)}) = H_0(P_\bullet^{M_\uc})=M_\uc$. The result hence follows by \cref{prop: injective}.
\end{proof}

For $d > 1$, the map $\psi_d$ is in general not surjective, as demonstrated below.

\begin{example}\label{ex:running7}
Let $A$ and $\mc$ be as in \cref{ex:running1}. In \cref{tab:intersection3} we give the inner acyclic \mbox{$3$-term} silting complexes $P_\bullet^{(M_{\uc},P_{\uc})}$ associated to all the $2$-torsion classes $\uc$ in $\mc$ by $\psi_d$, extending \cref{tab:intersection2} from \cref{ex:running1}. 

A simple computation shows that the complex $\rep{1\\2} \rightarrow 0 \rightarrow \rep{3} \oplus \rep{2\\3}$ is an inner acyclic $3$-term silting complex in $K^b(\proj A)$. However, inspecting \cref{tab:intersection3} shows that this silting complex is not in the image of the map $\psi_d$. Hence, not every inner acyclic $(d+1)$-term silting complex $S$ with $H_0(S)\in\mc$ can be obtained from a functorially finite $d$-torsion class as in \cref{thm:silting}. 

\begin{table}[t]\label{tab:intersection3}
    \centering
    \setlength{\extrarowheight}{6pt}
    \begin{tabular}{|c|c|c|c|}
    \hline
        $\uc$ & $\tc$ & $(M_\uc,P_\uc)$ & $P_\bullet^{(M_\uc,P_\uc)}$ \\[6pt]
        \hline
        \hline
         $\mc$ & $\modA$ & $\left(\rep{3}\oplus \rep{2\\3} \oplus \rep{1\\2}, 0\right)$ & $0\rightarrow 0 \rightarrow \rep{3}\oplus \rep{2\\3} \oplus \rep{1\\2}$\\[6pt]
         \hline
        $\add\left\{\rep{2\\3}\oplus \rep{1\\2} \oplus \rep{1}\right\}$ & $\add\left\{\rep{2\\3}\oplus \rep{2}\oplus \rep{1\\2} \oplus \rep{1}\right\}$ & $\left( \rep{2\\3} \oplus \rep{1\\2}\oplus \rep{1}, 0\right)$& $\rep{3}\rightarrow \rep{2\\3} \rightarrow \rep{2\\3}\oplus \rep{1\\2} \oplus \rep{1\\2}$\\[6pt]
        \hline
        $\add\left\{\rep{1\\2} \oplus \rep{1}\right\}$ & $\add\left\{ \rep{1\\2} \oplus \rep{1}\right\}$ & $\left( \rep{1\\2} \oplus \rep{1}, \rep{3}\right)$& $\rep{3}\oplus \rep{3}\rightarrow \rep{2\\3} \rightarrow \rep{1\\2} \oplus \rep{1\\2}$
        \\[6pt]
        \hline
        $\add\left\{\rep{1}\right\}$ & $\add\left\{\rep{1}\right\}$ & $\left(\rep{1},  \rep{3} \oplus \rep{2\\3}\right)$
        & $\rep{3}\oplus\rep{3}\oplus\rep{2\\3}\rightarrow \rep{2\\3} \rightarrow \rep{1\\2} $\\[6pt]
        \hline
        $\add\left\{\rep{3}\right\}$ & $\add\left\{\rep{3}\right\}$ & $\left(\rep{3}, \rep{2\\3} \oplus \rep{1\\2}\right)$& $\rep{2\\3}\oplus\rep{1\\2} \rightarrow 0 \rightarrow \rep{3}$\\[6pt]
        \hline
        $\left\{0\right\}$ & $\left\{0\right\}$ & $\left(0, \rep{3}\oplus \rep{2\\3} \oplus \rep{1\\2}\right)$& $\rep{3}\oplus \rep{2\\3} \oplus \rep{1\\2} \rightarrow 0 \rightarrow 0$\\[6pt]
        \hline
    \end{tabular}
    \caption{List of all $2$-torsion classes $\uc$ in $\mc$ and their associated minimal torsion classes $\tc$ in $\mod A$, \mbox{$\tau_2$-rigid} pairs $(M_\uc,P_\uc)$ in $\mc$ and $3$-term silting complexes $P_\bullet^{(M_\uc,P_\uc)}$ in $K^b(\proj A)$.}
\end{table}
\end{example}

We next show that \cref{thm:silting} produces silting complexes in the bounded derived category $D^b(A)$ of $\mod A$ whenever $A$ has global dimension less than or equal to $d$. This is for instance true for higher Auslander algebras of type $\mathbb{A}$ as introduced in \cite{Iyama2011}; see also \Cref{subsec: higher nakayama}.

\begin{corollary}\label{Theorem:SiltingInDerived}
   Let $\mc$ be a $d$-cluster tilting subcategory of $\modA$ and consider a functorially finite $d$-torsion class $\uc$ in $\mc$. If $A$ has global dimension less than or equal to $d$, then $M_\uc \oplus P_\uc[d]$ is a silting complex in $D^b(A)$.
\end{corollary}

\begin{proof}
 As $A$ has finite global dimension, we have $K^b(\proj A)\cong D^b(A)$. Moreover, since the global dimension of $A$ is less than or equal to $d$, the projective $d$-presentation $P_\bullet^{M_\uc}$ is a projective resolution of $M_\uc$. It follows that $M_\uc \oplus P_\uc[d] \cong P_\bullet^{(M_\uc,P_\uc)}$ in $D^b(A)$. As $P_\bullet^{(M_\uc,P_\uc)}$ is a silting complex by \cref{thm:silting}, so is $M_\uc \oplus P_\uc[d]$.
\end{proof}

We next show that the $\tau_d$-rigid pair $(M_\uc,P_\uc)$ associated to a functorially finite $d$-torsion class $\uc$ in $\mc$ as in \cref{thm:maximaltaunrigid} is maximal in the sense of \cref{def: tau-d-maximality}. This is an immediate consequence of the following stronger result, whose proof relies heavily on \cref{thm:silting}. 

\begin{proposition}
\label{corollary:maximal tau_d rigid}
Let $\mc$ be a $d$-cluster tilting subcategory of $\modA$ and consider a functorially finite $d$-torsion class $\uc$ in $\mc$. The following statements hold for $(M_\uc,P_\uc)$:
\theoremlistfix
\begin{enumerate}
            \item If $N$ is in $\mod A$, then
            \[
            N \in \add(M_\uc) \iff \begin{cases}
            \Ext_A^i(M_\uc,N)=0 \text{ for } i=1,\dots,d-1, \\
            \Ext_A^i(N,M_\uc)=0 \text{ for }i=1,\dots,d-1, \\
            \Hom_A(M_\uc,\tau_d N) = 0, \\
            \Hom_A(N,\tau_d M_\uc) = 0, \\
            \Hom_A(P_\uc,N) = 0.
\end{cases}
            \]
            \phantomsection
            \label{corollary:maximal tau_d rigid:1}
            \item If $Q$ is in $\proj A$, then
            \[
            Q \in \add(P_\uc) \iff \Hom_A(Q,M_\uc) = 0.
            \]
            \phantomsection
            \label{corollary:maximal tau_d rigid:2}
        \end{enumerate}
\end{proposition}

\begin{proof}
Let $N \in \mod A$. Observe first that $N\in \add (M_\uc)$ if and only if $P_\bullet^{N} \in \operatorname{add}(P_\bullet^{(M_\uc, P_\uc)})$. Indeed, if $P_\bullet^{N}$ is in $\operatorname{add}(P_\bullet^{(M_\uc, P_\uc)})$, then applying $H_0$ gives $N\in \add (M_\uc)$, while the converse is clear. Next note that as $P_\bullet^{(M_\uc,P_\uc)}$ is a silting complex in $K^b(\proj A)$ by \cref{thm:silting}, we have $P_\bullet^{N} \in \operatorname{add}(P_\bullet^{(M_\uc, P_\uc)})$ if and only if 
\begin{equation}\label{Equation:rigidityclaim}
\Hom_{K^b(\proj A)}(P_\bullet^{(M_\uc, P_\uc)}, P_\bullet^{N}[i]) =0= \Hom_{K^b(\proj A)}(P_\bullet^{N}, P_\bullet^{(M_\uc, P_\uc)}[i])  \quad \text{for }i>0
\end{equation}
by \cref{lem:maximal}. To prove \eqref{corollary:maximal tau_d rigid:1}, it thus remains to show that $N$ satisfies the five conditions on the right-hand side of the statement if and only if \eqref{Equation:rigidityclaim} holds.

To show this, note first that $N$ satisfies the first four conditions if and only if
\[
\Hom_{K^b(\proj A)}(P_\bullet^{M_{\uc}}, P_\bullet^{N}[i]) =0= \Hom_{K^b(\proj A)}(P_\bullet^{N}, P_\bullet^{M_{\uc}}[i]) \quad \text{for } i > 0
\]
by \cref{lem:tau-rigidKbProj}. Furthermore, we have 
\[
\Hom_{K^b(\proj A)}(P_\uc[d], P_\bullet^{N}[i]) \cong \Hom_{K^b(\proj A)}(P_\uc, P_\bullet^{N}[i-d]) =0 \quad \text{for }i>0 \text{ and } i \neq d
\]
since $P_\uc$ is projective and $P_\bullet^N$ is exact in degrees different from $-d$ and $0$. As $$\Hom_{K^b(\proj A)}(P_\uc[d], P_\bullet^{N}[d])\cong \Hom_A(P_\uc,N),$$ we thus have $\Hom_A(P_\uc,N)=0$ if and only if 
$\Hom_{K^b(\proj A)}(P_\uc[d], P_\bullet^{N}[i])=0$ for $i>0$.
Finally, note that $\Hom_{K^b(\proj A)}(P_\bullet^{N},P_\uc[d][i])=0$ for $i>0$ since $P_\bullet^{N}$ and $P_\uc[d][i]$ are concentrated in different degrees. The proof of \eqref{corollary:maximal tau_d rigid:1} is finished by combining these observations. 

The proof of \eqref{corollary:maximal tau_d rigid:2} follows a similar strategy as above. Consider $Q \in \proj A$, and note first that $Q \in \add(P_\uc)$ if and only if $Q[d] \in \add(P_\bullet^{(M_\uc,P_\uc)})$, since $P_{\bullet}^{M_{\uc}}$ does not have a summand of the form $Q'[d]$ with $Q'$ projective. As $P_\bullet^{(M_\uc,P_\uc)}$ is a silting complex in $K^b(\proj A)$, we have $Q[d] \in \add(P_\bullet^{(M_\uc,P_\uc)})$ if and only if
\[
\Hom_{K^b(\proj A)}(Q[d],P_\bullet^{(M_\uc, P_\uc)}[i])=0=\Hom_{K^b(\proj A)}(P_\bullet^{(M_\uc, P_\uc)},Q[d][i])\quad \text{for }i>0
\]
by \cref{lem:maximal}. It remains to prove that this is equivalent to having $\Hom_A(Q,M_\uc)=0$. To this end, observe first that $\Hom_{K^b(\proj A)}(P_\bullet^{(M_\uc, P_\uc)},Q[d][i])=0$ for $i>0$ since $P_\bullet^{(M_\uc, P_\uc)}$ and $Q[d][i]$ are concentrated in different degrees. By a similar argument as before, we moreover have
\[
\Hom_{K^b(\proj A)}(Q[d],P_\bullet^{(M_\uc, P_\uc)}[i])=0 \quad \text{for }i>0 \text{ and } i\neq d.
\]
Combining this with the isomorphism $\Hom_{K^b(\proj A)}(Q[d],P_\bullet^{(M_\uc, P_\uc)}[d])\cong \Hom_A(Q,M_{\uc})$ finishes the proof.
\end{proof}

\begin{corollary}\label{cor: maximality}
Let $\mc$ be a $d$-cluster tilting subcategory of $\modA$ and consider a functorially finite $d$-torsion class $\uc$ in $\mc$. Then $(M_\uc, P_\uc)$ is a maximal $\tau_d$-rigid pair in $\mc$.
\end{corollary}

The corollary above finishes the proof of \Cref{MainResultIntro} from the introduction. In particular, the map $\phi_d$ from \Cref{MainResultIntro} is well-defined. We finish this section by giving an example showing that it is in general not surjective.

\begin{example}\label{Example:Phi_dNotSurj}
    Let $A$ and $\mc$ be as in \cref{ex:running1}. Recall that the the $\tau_2$-rigid pairs $(M_\uc,P_\uc)$ associated to each $2$-torsion class $\uc$ in $\mc$ are shown in the third column of \cref{tab:intersection3}. Notice that the pair $(M,P)=\left(\rep{3}\oplus \rep{2\\3}, \rep{1\\2} \right)$ is maximal $\tau_2$-rigid in $\mc$, but does not appear in the table even though it has $|A|$ indecomposable direct summands. This demonstrates that the map $\phi_d$ from \Cref{MainResultIntro} is not always surjective. It is moreover worth noting that $P_\bullet^{(M,P)}=P_\bullet^M \oplus P[d]$ is the silting complex $\rep{1\\2} \rightarrow 0 \rightarrow \rep{3} \oplus \rep{2\\3}$ discussed in \cref{ex:running7}. 
\end{example}

\section{Computations for higher Auslander and higher Nakayama algebras}\label{subsec: higher nakayama}

Higher Auslander and higher Nakayama algebras were introduced in \cite{Iyama2011} and \cite{Higher Nakayama}, respectively, and constitute important classes of algebras in higher Auslander--Reiten theory. In this section we show that our theory yields an injective, but in general not surjective, map from the functorially finite $d$-torsion classes of higher Auslander algebras to Oppermann--Thomas cluster tilting objects in the associated $d$-dimensional cluster category. Moreover, we give explicit combinatorial descriptions of maximal $\tau_d$-rigid pairs and silting complexes associated to functorially finite $d$-torsion classes for higher Auslander and higher Nakayama algebras of type $\mathbb{A}$. These combinatorial descriptions are implemented in a Python code, allowing the computation of a plethora of explicit examples; see \cref{rem: code}. Our descriptions rely on the characterisation of $d$-torsion classes for higher Auslander and higher Nakayama algebras of type $\mathbb{A}$ obtained in \cite{AHJKPT1}. 

We first recall the definition of higher Auslander and higher Nakayama algebras, as well as the combinatorial description of their modules, following \cite[Sections 1 and 2]{Higher Nakayama}; see also \cite[Sections 5.1 and 6.1]{AHJKPT1}. For a positive integer $n$, consider $N_n=\{0,1, \dots, n-1\}$ with the natural poset structure. Let  
\[
N_{n}^d = \underbrace{N_n\times\cdots\times N_n}_{d\text{ times}}
\]
be the set of sequences $x=(x_0,\dots,x_{d-1})$ over $N_n$ endowed with the product order. This poset can naturally be considered as a quiver with relations, whose path algebra modulo the ideal generated by the relations is equal to 
\[
A_n^{\otimes d}=\underbrace{A_n\otimes_\kk A_n\otimes_\kk \cdots \otimes_\kk A_n}_{d\text{ times}}\hspace{0.111em},
\]
where $A_n$ is the path algebra of the quiver $0 \rightarrow \dots \rightarrow n-1$. Let $\os_n^d$ denote the subset of $N_n^d$ whose elements are the tuples $x=(x_0,\dots,x_{d-1})$ with $x_{0}\leq x_{1}\leq\dots\leq x_{d-1}$. The \textit{higher Auslander algebra} $A_n^{d}$ is defined to be the idempotent quotient
\[
A_n^{d} \colonequals A_n^{\otimes d}/A_n^{\otimes d}e_{n,d}A_n^{\otimes d}
\]
where $e_{n,d}$ is the sum of the idempotents corresponding to the elements of $N_n^d\setminus \os_n^d$. Note that $A_n^{d}$ can equivalently be defined as the path algebra of a quiver with relations, where the (opposite) of this quiver has vertices given by the set $\os_n^d$ and an arrow from vertex $x$ to vertex $y$ if $y_i=x_i+1$ for precisely one $0\leq i\leq d-1$ and $y_j=x_j$ for $j\neq i$. For a description of the associated relations, see \cite{HerschendJorgensen}.

For each $x\in \os_n^{d+1}$, there is an indecomposable $A_n^{d}$-module $M_x$, which is given by $\kk$ at the vertices $y\in \os_n^d$ for which $x_{0}\leq y_{0} \leq x_{1} \leq \cdots \leq x_{d-1} \leq y_{d-1} \leq x_{d} $ and zero otherwise. Set $M^d_{n}=\bigoplus_{x\in \os_n^{d+1}} M_x$. Then 
\[
\mc^d_{n} \colonequals \add(M^d_{n}) \subseteq \mod A^d_n
\]
is a $d$-cluster tilting subcategory of $\operatorname{mod}A^d_n$, and $\End_{A_n^d}(M^d_{n})$ and $A^{d+1}_{n}$ are isomorphic as algebras by \cite[Corollary 1.16]{Iyama2011}. For the explicit description of $\mc^d_{n}$ presented above, see also \cite[Theorem/Construction 3.4]{OppermannThomas} and \cite[Theorem 2.4]{Higher Nakayama}. 

Let $\mathcal{O}_{A_n^d}$ denote the $d$-dimensional cluster category of $A_n^d$; see \cite[Section 5]{OppermannThomas} for details. We call the cluster tilting objects defined in \cite[Definition 5.3]{OppermannThomas} \textit{Oppermann--Thomas cluster tilting}, following \cite{JJ}. As a consequence of our work, we obtain the following result.

\begin{theorem}\label{Theorem:d-torsionToClusterTilting}
There is an injective map from functorially finite $d$-torsion classes in $\mc^d_{n}$ to basic Oppermann--Thomas cluster tilting objects in $\mathcal{O}_{A_n^d}$.
\end{theorem}

\begin{proof}
By \Cref{thm:maximaltaunrigid} and \Cref{prop: injective}, there is an injective map from functorially finite $d$-torsion classes in $\mc^d_{n}$ to basic $\tau_d$-rigid pairs $(M,P)$ in $\mc^d_{n}$ with $|M|+|P|= |A_n^d|$. Now by  \cite[Theorem 3.5]{JJ}, these $\tau_d$-rigid pairs are in bijection with isomorphism classes of $d$-rigid objects in $\mathcal{O}_{A_n^d}$ with $|A_n^d|$ summands. Finally, by \cite[Theorem 2.4 and Theorem 6.4]{OppermannThomas}, any $d$-rigid object in $\mathcal{O}_{A_n^d}$ with $|A_n^d|$ summands must be Oppermann--Thomas cluster tilting.
\end{proof}

For $d=1$, the map in \Cref{Theorem:d-torsionToClusterTilting} recovers the bijection between functorially finite torsion classes for $A_n$ and cluster tilting objects in the cluster category of $A_n$; see \cite[Theorem 0.5]{AdachiIyamaReiten}. For $d\geq 2$, the map is in general not surjective since the map \eqref{map: ffntorsion to taun rigid} is not surjective; see \Cref{ex:running1}. 

Next we define higher Nakayama algebras. From now on we assume $d\geq 2$. Let $\underline{\ell}=(\ell_0,\dots, \ell_{n-1})$ be a \textit{(connected) Kupisch series of type }$\mathbb{A}_n$, i.e.\ a tuple of positive integers satisfying 
\[
\ell_0=1 \quad \text{and} \quad 2\leq \ell_i\leq \ell_{i-1}+1  \quad \text{for} \quad i=1,\ldots,n-1. 
\]
Consider the subset $\os_{\underline{\ell}}^{d}\colonequals\{y\in \os_{n}^{d}\mid \ell\ell(y)\leq \ell_{y_{d-1}}\}$ of $\os_{n}^{d}$, where $\ell\ell(y)=y_{d-1}-y_0+1$. The $d$-\textit{th Nakayama algebra with Kupisch series $\underline{\ell}$} is given as the idempotent quotient
\[
A_{\underline{\ell}}^d \colonequals A_n^d/A_n^d e_{\underline{\ell}}A_n^d
\]
where $e_{\underline{\ell}}$ is the sum of the idempotents corresponding to the elements of $\os_{n}^{d}\setminus\os_{\underline{\ell}}^{d}$. In particular, if $\underline{\ell}=(1,2,\dots,n)$, then $A_{\underline{\ell}}^d=A_n^d$. 

Now assume that $\underline{\ell}$ is a Kupisch series of type $\mathbb{A}_n$. By \cite[Proposition 2.24]{Higher Nakayama}, the subcategory $$\mc_{\underline{\ell}}^d \colonequals \mc_{n}^d\cap \operatorname{mod}A_{\underline{\ell}}^d$$ is $d$-cluster tilting in $\operatorname{mod}A_{\underline{\ell}}^d$. If we set $$M_{\underline{\ell}}^d=\bigoplus_{x\in \os_{\underline{\ell}}^{d+1}}M_x,$$ then $\mc_{\underline{\ell}}^d=\operatorname{add}(M_{\underline{\ell}}^d)$. Furthermore, for $x\in \os_{\underline{\ell}}^{d+1}$, the $A_{\underline{\ell}}^d$-module $M_x$ is projective  if and only if $x_0=x_d+1-\ell_{x_d}$ by \cite[Proposition 2.22]{Higher Nakayama}. It follows that the isomorphism classes of indecomposable projective $A^d_{\underline{\ell}}$-modules are in bijection with $\os^d_{\underline{\ell}}$, where \mbox{$y=(y_1,\dots,y_d)\in \os^d_{\underline{\ell}}$} corresponds to the module
\[
P_y\colonequals M_{(y_0,y_1,\cdots , y_d)}
\]
with $y_0=y_d+1-\ell_{y_d}$. The higher Auslander--Reiten translation in $\mc_{\underline{\ell}}^d$ is given by
\[
\tau_d(M_x)=\begin{cases}
0 &\text{if } M_x \text{ is a projective }A_{\underline{\ell}}^d \text{-module} \\
	M_{\tau_d(x)} & \text{otherwise}  \\ 
\end{cases}\]
by \cite[Proposition 2.26]{Higher Nakayama}, where $\tau_d\colon \mathbb{Z}^{d+1}\to \mathbb{Z}^{d+1}$ is defined by
\begin{align*}
 \tau_d(y_0,\dots,y_{d})=(y_0-1,y_1-1,\dots, y_{d}-1).
\end{align*}
In particular, if $x\in \os_{\underline{\ell}}^{d+1}$ and $M_x$ is not projective, then $\tau_d(x)\in \os_{\underline{\ell}}^{d+1}$. Now let
\begin{align*}
	x\rightsquigarrow y &\text{ if and only if } x_{0}\leq y_{0} \leq x_{1}\leq y_{1} \leq \cdots  \leq  x_{d}\leq y_{d}
\end{align*}
be a relation on the set $\mathbb{Z}^{d+1}$. Note that $x\rightsquigarrow y$ implies $x\leq y$. It follows from \cite[Theorem 3.6 (3)]{OppermannThomas} and \cite[Proposition 2.8]{Higher Nakayama} that
\begin{equation}\label{Equation:HomSpacesNakayama}
\dim \Hom_{A_n^d}(M_x,M_y)=\begin{cases}
	1 & \text{if } x\rightsquigarrow y \\ 0 &\text{otherwise}
\end{cases}
\end{equation}
for $x,y\in \os_n^{d+1}$.  For a subset $I \subseteq \os_{\underline{\ell}}^{d+1}$, let  \begin{align*}
        \uc_I\colonequals \add \{ M_{y} \in \mc_{\underline{\ell}}^d \mid y \in I\}
    \end{align*}
    be the associated subcategory of $\mc_{\underline{\ell}}^d$. This gives a bijection between subset of $\os_{\underline{\ell}}^{d+1}$ and subcategories of $\mc_{\underline{\ell}}^d$. By \cite[Theorem 6.1]{AHJKPT1}, the $d$-torsion classes in $\mc_{\underline{\ell}}^d$ are precisely the subcategories $\uc_I$ where $I$ satisfies the following conditions for any elements $x,z\in \os_{\underline{\ell}}^{d+1}$:
    \theoremlistfix
    \begin{enumerate}[label=\textbf{T\arabic*}]
        \item If $x\leq z$ and $x_d=z_d$, then $x\in I$ implies $z\in I$.
        \phantomsection
        \label{thm:higherNakayamaAlg:1}
        \item If $x\rightsquigarrow \tau_d(z)$ and $x,z\in I$, then any $y\in \os_{\underline{\ell}}^{d+1}$ with $y_i \in \{x_i,z_i\}$ for each $i$ must be in $I$. 
        \phantomsection
        \label{thm:higherNakayamaAlg:2}
    \end{enumerate} 

Using this characterisation, we obtain a simple combinatorial description of the $\Ext^d$-projectives in a $d$-torsion class and their associated projective $d$-presentations. 

\begin{lemma}\label{Theorem:CombinatorialFormulaNakayama}
Let $\underline{\ell}$ be a Kupisch series of type $\mathbb{A}_n$ and consider a subset $I$ of $\os_{\underline{\ell}}^{d+1}$ satisfying \eqref{thm:higherNakayamaAlg:1} and \eqref{thm:higherNakayamaAlg:2}. The following statements hold for $x\in \os^{d+1}_{\underline{\ell}}$:
\theoremlistfix
\begin{enumerate}
    \item Assume $x\in I$. The module $M_x$ is $\Ext^d$-projective in $\uc_I$ if and only if 
    \theoremlistfix
    \begin{itemize}
        \item $x_0=x_d+1-\ell_{x_d}$, or
        \item there is no relation $y\rightsquigarrow \tau_d(x)$ with $y\in I$.
    \end{itemize}
    \phantomsection
    \label{Theorem:CombinatorialFormulaNakayama:1}
    \item $M_x$ is a projective $A^d_{\underline{\ell}}$-module and satisfies $\Hom_{A_{\underline{\ell}}^d}(M_x,\uc_I)=0$ if and only if 
    \theoremlistfix
    \begin{itemize}
        \item $x_0=x_d+1-\ell_{x_d}$, and
        \item there is no relation $x\rightsquigarrow y$ with $y\in I$.
    \end{itemize}
    \phantomsection
    \label{Theorem:CombinatorialFormulaNakayama:2}
    \item Assume that $M_x$ is not projective. The minimal projective $d$-presentation of $M_x$ is 
    \[
    P_{\bullet}^{M_x}\colon \cdots \xrightarrow{} 0 \xrightarrow{} P_{y^d} \xrightarrow{f_{d}} \cdots \xrightarrow{f_3} P_{y^2} \xrightarrow{f_2} P_{y^1} \xrightarrow{f_1} P_{y^0} \xrightarrow{} 0 \xrightarrow{} \cdots
    \]
    with $y^i=(y^i_1,\dots, y^i_d)$ and $$y^i_j=\begin{cases}
	x_{j-1}-1 & \text{if } j\leq i \\ x_{j} &\text{if } j>i,
\end{cases}$$ and where the morphism $f_i$ is induced from the relation $y^i\rightsquigarrow y^{i-1}$ for $i=1,\dots,d$.  
\phantomsection
\label{Theorem:CombinatorialFormulaNakayama:3}
\end{enumerate}
\end{lemma}

\begin{proof}
For \eqref{Theorem:CombinatorialFormulaNakayama:1}, note that the module $M_x$ is $\Ext^d$-projective in the $d$-torsion class $\uc_I$ if and only if $$\Hom_{A^d_{\underline{\ell}}}(M_y,\tau_d(M_x))=0$$ for all $y\in I$ by \cref{thm: equivalence of Extd projectivity}. If $M_x$ is projective, then this automatically holds since $\tau_d(M_x)=0$. This corresponds to the condition $x_0=x_d+1-\ell_{x_d}$. If $M_x$ is not projective, then $\tau_d(M_x)=M_{\tau_d(x)}$. Now by the description of the $\Hom$-spaces in \eqref{Equation:HomSpacesNakayama}, we have $\Hom_{A^d_{\underline{\ell}}}(M_y,M_{\tau_d(x)})=0$ if and only there is no relation $y\rightsquigarrow \tau_d(x)$. This proves \mbox{part \eqref{Theorem:CombinatorialFormulaNakayama:1}}. Part \eqref{Theorem:CombinatorialFormulaNakayama:2} follows immediately from the description of projective $A^d_{\underline{\ell}}$-module and the description of the Hom-spaces in $\mc_{\underline{\ell}}^d$. Part \eqref{Theorem:CombinatorialFormulaNakayama:3} follows from the proof of \cite[Proposition 2.25]{Higher Nakayama}.
\end{proof}

Given $I\subseteq \os_{\underline{\ell}}^{d+1}$, we define subsets
\begin {align*}
& I_1=\{x\in I\mid x_0=x_d+1-\ell_{x_d}\} \cup \{x\in I\mid \text{there is no relation }y\rightsquigarrow \tau_d(x) \text{ with }y\in I\} \\
& I_2=\{x\in \os_{\underline{\ell}}^{d+1}\mid x_0=x_d+1-\ell_{x_d}\} \cap \{x\in \os_{\underline{\ell}}^{d+1}\mid \text{there is no relation }x\rightsquigarrow y \text{ with }y\in I\}. 
\end{align*}

Recall that we use the notation $(M_{\uc_I},P_{\uc_I})$ for the basic $\tau_d$-rigid pair associated to a functorially finite $d$-torsion class $\uc_I$ in $\mc_{\underline{\ell}}^d$ as in \cref{thm:maximaltaunrigid}. Moreover, we write
\[
P_\bullet^{(M_{\uc_I},P_{\uc_I})} = P_\bullet^{M_{\uc_I}} \oplus P_{\uc_I}[d]
\]
for the associated silting complex in $K^b(\proj A_{\underline{\ell}}^d)$ as in \cref{thm:silting}. We are now ready to provide a complete combinatorial description of maximal $\tau_d$-rigid pairs and silting complexes associated to functorially finite $d$-torsion classes for higher Auslander and higher Nakayama algebras of type $\mathbb{A}$. Recall that the explicit description of $P_\bullet^{M_x}$ is given in \Cref{Theorem:CombinatorialFormulaNakayama} \eqref{Theorem:CombinatorialFormulaNakayama:3}.

\begin{theorem}\label{cor: higher nakayama}
 Let $\underline{\ell}$ be a Kupisch series of type $\mathbb{A}_n$ and consider a subset $I$ of $\os_{\underline{\ell}}^{d+1}$ satisfying \eqref{thm:higherNakayamaAlg:1} and \eqref{thm:higherNakayamaAlg:2}.  The following statements hold:
 \theoremlistfix
 \begin{enumerate}
 \item We have $M_{\uc_I}=\bigoplus_{x\in I_1}M_x$ and $P_{\uc_I}=\bigoplus_{y\in I_2}M_y$.
 \phantomsection
 \label{combinatorial description 1}
 \item The complex $$P_\bullet^{(M_{\uc_I},P_{\uc_I})} = \bigoplus_{x\in I_1}P_\bullet^{M_x} \oplus \bigoplus_{y\in I_2}M_y[d]$$    is silting in $K^b(\proj A_{\underline{\ell}}^d)$.
 \phantomsection
 \label{combinatorial description 2}
 \item Assume $\underline{\ell}=(1,2,\dots,n)$, so that $A_{\underline{\ell}}^d=A_n^d$ is a higher Auslander algebra. Then $$\bigoplus_{x\in I_1}M_x \oplus \bigoplus_{y\in I_2}M_y[d]$$ is silting in $D^b(A_n^d)$.
 \phantomsection
 \label{combinatorial description 3}
 \end{enumerate}
\end{theorem}

\begin{proof}
    The first part holds by \Cref{Theorem:CombinatorialFormulaNakayama}. Part (\ref{combinatorial description 2}) is a consequence of (\ref{combinatorial description 1}) combined with \cref{thm:silting}, while (\ref{combinatorial description 3}) follows from \cref{Theorem:SiltingInDerived}.
\end{proof}

\begin{remark}\label{rem: code}
Applying the combinatorial descriptions in \cref{Theorem:CombinatorialFormulaNakayama} and \cref{cor: higher nakayama}, we provide Python code that explicitly computes all $d$-torsion classes for the higher Nakayama algebra $A_{\underline{\ell}}^d$ for different choices of $d$ and $\underline{\ell}$, along with the associated maximal $\tau_d$-rigid pairs and silting complexes. The code is available as a Google Colab notebook.\footnote{
\url{https://colab.research.google.com/drive/172Q-UZHvdPOhngGkl1T_xdYLzntg31dY}} 
\end{remark}

\subsection*{Acknowledgements}
This project was initiated during the Junior Trimester Program ``New Trends in Representation Theory'' at the Hausdorff Research Institute for Mathematics in Bonn. We thank HIM for excellent working conditions. We are also grateful to J. Asadollahi, P. J\o rgensen and S. Schroll for allowing HT to share the ideas of \cite{AJST} even before that paper was finished. We thank M. Herschend for helpful discussions and in particular for suggesting the statement of \cref{theorem: d-CT extriangulated}, which led to a simplification of the proofs in \cref{sec:silting}. We moreover thank S. Land for coding optimisation help. 

JA was supported by the Max Planck Institute of Mathematics, and a DNRF Chair from the Danish National Research Foundation (grant no. DNRF156). JH was supported by the project `Pure Mathematics in Norway' funded by the Trond Mohn Research Foundation. KMJ was funded by the Norwegian Research Council via the project ‘Higher homological algebra and tilting theory’ (301046), and is grateful for support by the Danish National Research Foundation (DNRF156). YP was supported by the French ANR grant CHARMS~(ANR-19-CE40-0017). 

HT acknowledges that this work was completed in spite of Argentina's current science policies, which severely hinder basic research in the country.

\end{document}